\documentclass[10pt]{article}
\usepackage{amsthm}
\usepackage{amsmath}
\usepackage{amssymb}
\usepackage{makeidx}

\theoremstyle{definition}
\newtheorem{definition}{Definition}[section]

\newtheorem{example}[definition]{Example}
\newtheorem{openproblem}[definition]{Open problem}

\newtheorem{remark}[definition]{Remark}

\theoremstyle{plain}
\newtheorem{theorem}[definition]{Theorem}
\newtheorem{lemma}[definition]{Lemma}
\newtheorem{proposition}[definition]{Proposition}
\newtheorem{corollary}[definition]{Corollary}

\newcommand{\dotvee}{\stackrel{\bullet }{\vee }}

\numberwithin{equation}{section}

\def\N{{\mathbb N}}

\def\Z{{\mathbb Z}}

\begin{document}
\title{The Reticulation of a Universal Algebra}
\author{George GEORGESCU and Claudia MURE\c SAN\thanks{Corresponding author.}\\ \footnotesize University of Bucharest\\ \footnotesize Faculty of Mathematics and Computer Science\\ \footnotesize Academiei 14, RO 010014, Bucharest, Romania\\ \footnotesize Emails: georgescu.capreni@yahoo.com; cmuresan@fmi.unibuc.ro, c.muresan@yahoo.com}
\date{\today }
\maketitle

\begin{abstract} The {\em reticulation} of an algebra $A$ is a bounded distributive lattice ${\cal L}(A)$ whose prime spectrum of filters or ideals is homeomorphic to the prime spectrum of congruences of $A$, endowed with the Stone topologies. We have obtained a construction for the reticulation of any algebra $A$ from a semi--degenerate congruence--modular variety ${\cal C}$ in the case when the commutator of $A$, applied to compact congruences of $A$, produces compact congruences, in particular when ${\cal C}$ has principal commutators; furthermore, it turns out that weaker conditions than the fact that $A$ belongs to a congruence--modular variety are sufficient for $A$ to have a reticulation. This construction generalizes the reticulation of a commutative unitary ring, as well as that of a residuated lattice, which in turn generalizes the reticulation of a BL--algebra and that of an MV--algebra. The purpose of constructing the reticulation for the algebras from ${\cal C}$ is that of transferring algebraic and topological properties between the variety of bounded distributive lattices and ${\cal C}$, and a {\em reticulation functor} is particularily useful for this transfer. We have defined and studied a reticulation functor for our construction of the reticulation in this context of universal algebra.\end{abstract}

\section{Introduction}
\label{introduction}

The reticulation of a commutative unitary ring $R$ is a bounded distributive lattice ${\cal L}(R)$ whose prime spectrum of ideals is homeomorphic to the prime spectrum of ideals of $R$. Its construction has appeared in \cite{joy}, but it has been extensively studied in \cite{sim}, where it has received the name {\em reticulation}. The mapping $R\mapsto {\cal L}(R)$ sets a covariant functor from the category of commutative unitary rings to that of bounded distributive lattices, through which properties can be transferred between these categories. In \cite{bell}, the reticulation has been defined and studied for non--commutative unitary rings and it has been proven that such a ring has a reticulation (with the topological definition above) iff it is quasi--commutative.

Over the past two decades, reticulations have been constructed for orderred algebras related to logic: MV--algebras \cite{bellsspl,belldinlett}, BL--algebras \cite{leo1,dinggll,leo}, residuated lattices \cite{eu1,eu,eu2,eu4,eu5,eu7}, $0$--distributive lattices \cite{paw}, almost distributive lattices \cite{pash}, Hilbert algebras \cite{bus}, hoops \cite{dche}. All these algebras posess a ``prime spectrum`` which is homeomorphic to the prime spectrum of filters or ideals of a bounded distributive lattice; their reticulations consist of such bounded distributive lattices, whose study involves obtaining a construction for them and using that construction to transfer properties between these classes of algebras and bounded distributive lattices.

The purpose of the present paper is to set the problem of constructing a reticulation in a universal algebra framework and providing a solution to this problem in a case as general as possible, that includes the cases of the varieties above and generalizes the constructions which have been obtained in those particular cases. Apart from the novelty of using commutator theory \cite{cze,mcks} for the study of the reticulation, essentially, the tools needed for obtaining reticulations in this very general setting are quite similar to those which have been put to work for the classes of algebras above, and it turns out that many types of results that hold for their reticulations can be generalized to our setting. In order to obtain strong generalizations, we have worked with hypotheses as weak as possible; all our results in this paper hold for semi--degenerate congruence--modular varieties whose members have the sets of compact congruences closed with respect to the commutator, with just a few exceptions that necessitate, moreover, principal commutators.

The present paper is structured as follows: Section \ref{preliminaries} presents the notations and basic results we use in what follows; Section \ref{commutator} collects a set of results from commutator theory which we use in the sequel; in Section \ref{stonetops}, we present the standard construction of the Stone topologies on prime spectra, specifically the prime spectrum of ideals of a bounded distributive lattice and the prime spectrum of congruences of a universal algebra whose commutator fulfills certain conditions. The results in the following sections that are not cited from other papers, or mentioned as being either known or quite simple to obtain, are new and original.

In Section \ref{reticulatia}, we construct the reticulation for universal algebras whose commutators fulfill certain conditions, prove that this construction has the desired topological property and obtain some related results.

In Section \ref{examples}, we provide some examples of reticulations, study particular cases, such as the congruence--distributive case, show that our construction generalizes constructions for the reticulation which have been obtained for particular varieties, and prove that our construction preserves finite direct products of algebras without skew congruences.

In Section \ref{further}, we obtain some arithmetical properties on commutators that we need in what follows, as well as algebraic properties regarding the behaviour of surjections with respect to commutators and to certain types of congruences.

In Section \ref{boolean} we study the behaviour of Boolean congruences with respect to the reticulation, in the general case, but also in particular ones, such as the case of associative commutators or that of semiprime algebras.

In Section \ref{functor}, we define a reticulation functor; our definition is not ideal, as it only acts on surjections; extending it to all morphisms remains an open problem. In this final section, we also show that the reticulation preserves quotients, and that it is a Boolean lattice exactly in the case of hyperarchimedean algebras, which we also characterize by several other conditions on their reticulation. These characterizations serve as an example for the transfer of properties to and from the category of bounded distributive lattices which the reticulation makes possible.

We intend to further pursue the study of the reticulation in this universal algebra setting and use it to transfer more properties between the variety of bounded distributive lattices and the kinds of varieties that allow a construction for the reticulation. A theme for a potentially extensive future study is characterizing those varieties with the property that the reticulations of their members cover the entire class of bounded distributive lattices.

\section{Preliminaries}
\label{preliminaries}

In this section, we recall some properties on lattices and congruences in universal algebras. For a further study of the following results on universal algebras, we refer the reader to \cite{agl}, \cite{bur}, \cite{gralgu}, \cite{koll}. For those on lattices, we recommend \cite{bal}, \cite{blyth}, \cite{cwdw}, \cite{gratzer}, \cite{schmidt}.

We shall denote by $\N $ the set of the natural numbers and by $\N ^*=\N \setminus \{0\}$. For any set $M$, ${\cal P}(M)$ shall be the set of the subsets of $M$, $id_M:M\rightarrow M$ shall be the identity map, and we shall denote by $\Delta _M=\{(x,x)\ |\ x\in M\}$ and $\nabla _M=M^2$. For any family $(M_i)_{i\in I}$ of sets and any $\displaystyle M\subseteq \prod _{i\in I}M_i$, whenever there is no danger of confusion, by $a=(a_i)_{i\in I}\in M$ we mean $a_i\in M_i$ for all $i\in I$, such that $a\in M$. For any sets $M$, $N$ and any function $f:M\rightarrow N$, we shall denote by ${\rm Ker}(f)=\{(x,y)\in M^2\ |\ f(x)=f(y)\}$, and the direct and inverse image of $f$ in the usual way; we shall denote, simply, $f=f^2:{\cal P}(M^2)\rightarrow {\cal P}(N^2)$ and $f^*=(f^2)^{-1}:{\cal P}(N^2)\rightarrow {\cal P}(M^2)$; so, for any $X\subseteq M^2$ and any $Y\subseteq N^2$, $f(X)=\{(f(a),f(b))\ |\ (a,b)\in X\}$ and $f^*(Y)=\{(a,b)\in M^2\ |\ (f(a),f(b))\in Y\}$, thus ${\rm Ker}(f)=f^*(\Delta _N)$. Also, if $X_i\subseteq M_i^2$ for all $i\in I$, then the direct product of $(X_i)_{i\in I}$ as a family of binary relations shall be denoted just as the one for sets, because there will be no danger of confusion when using this notation: $\displaystyle \prod _{i\in I}X_i=\{((a_i)_{i\in I},(b_i)_{i\in I})\ |\ (\forall \, i\in I)\, ((a_i,b_i)\in X_i)\}\subseteq M^2$. Unless mentioned otherwise, the operations and order relation of a (bounded) lattice shall be denoted in the usual way, and the complementation of a Boolean algebra shall be denoted by $\neg \, $.

Throughout this paper, whenever there is no danger of confusion, any algebra shall be designated by its support set. All algebras shall be considerred non--empty; by {\em trivial algebra} we shall mean one--element algebra, and by {\em non--trivial algebra} we shall mean algebra with at least two distinct elements. Any direct product of algebras and any quotient algebra shall be considerred with the operations defined canonically. For brevity, we shall denote by $A\cong B$ the fact that two algebras $A$ and $B$ of the same type are isomorphic.

Let $L$ be a bounded lattice. By ${\rm Id}(L)$ we shall denote the set of the ideals of $L$, that is the non--empty subsets of $L$ which are closed with respect to the join and to lower bounds. By ${\rm Filt}(L)$ we shall denote the set of the filters of $L$, that is the ideals of the dual of $L$: the non--empty subsets of $L$ which are closed with respect to the meet and to upper bounds. For any $M\subseteq L$ and any $a\in L$, $(M]$, respectively $[M)$, shall denote the ideal, respectively the filter of $L$ generated by $M$, and the principal ideal, $(\{a\}]=\{x\in L\ |\ a\geq x\}$, respectively the principal filter, $[\{a\})=\{x\in L\ |\ a\leq x\}$, generated by $a$ shall also be denoted by $(a]$, respectively $[a)$; whenever we need to specify the lattice $L$, we shall denote $[M)_L$, $(M]_L$, $[a)_L$ and $(a]_L$ instead of $[M)$, $(M]$, $[a)$ and $(a]$, respectively. It is well known that $({\rm Id}(L),\vee ,\cap ,\{0\},L)$ and $({\rm Filt}(L),\vee ,\cap ,\{1\},L)$ are bounded lattices, with $J\vee K=(J\cup K]$ and $F\vee G=[F\cup G)$ for all $J,K\in {\rm Id}(L)$ and all $F,G\in {\rm Filt}(L)$, and they are distributive iff $L$ is distributive; moreover, they are complete lattices, with $\displaystyle \bigvee _{i\in I}J_i=[\bigcup _{i\in I}J_i)$ and $\displaystyle \bigvee _{i\in I}F_i=(\bigcup _{i\in I}F_i]$ for any families $(J_i)_{i\in I}\subseteq {\rm Id}(L)$ and $(F_i)_{i\in I}\subseteq {\rm Filt}(L)$. Obviously, for any $a,b\in L$, $(a]\vee (b]=(a\vee b]$, $(a]\cap (b]=(a\wedge b]$, $[a)\vee [b)=[a\wedge b)$ and $[a)\cap [b)=[a\vee b)$. If $L$ is a complete lattice, then, for any family $(a_i)_{i\in I}\subseteq L$, $\displaystyle \bigvee _{i\in I}(a_i]=(\bigvee _{i\in I}a_i]$, $\displaystyle \bigcap _{i\in I}(a_i]=(\bigwedge _{i\in I}a_i]$, $\displaystyle \bigvee _{i\in I}[a_i)=[\bigwedge _{i\in I}a_i)$ and $\displaystyle \bigcap _{i\in I}[a_i)=[\bigvee _{i\in I}a_i)$. By ${\rm PId}(L)$, respectively ${\rm PFilt}(L)$, we shall denote the set of the principal ideals, respectively the principal filters of $L$. We shall denote by ${\rm Max}_{\rm Id}(L)$, respectively ${\rm Max}_{\rm Filt}(L)$, the set of the maximal ideals, respectively the maximal filters of $L$, that is the maximal elements of the set of proper ideals of $L$, ${\rm Id}(L)\setminus \{L\}$, respectively that of proper filters of $L$, ${\rm Filt}(L)\setminus \{L\}$. By ${\rm Spec}_{\rm Id}(L)$ we shall denote the set of the prime ideals of $L$, that is the proper ideals $P$ of $L$ such that, for any $x,y\in L$, $x\wedge y\in P$ implies $x\in P$ or $y\in P$. Dually, ${\rm Spec}_{\rm Filt}(L)$ shall denote the set of the prime filters of $L$, that is the proper filters $P$ of $L$ such that, for any $x,y\in L$, $x\vee y\in P$ implies $x\in P$ or $y\in P$.

For any algebra $A$, ${\rm Con}(A)$ shall denote the set of the congruences of $A$, and ${\rm Max}(A)$ shall denote the set of the maximal congruences of $A$, that is the maximal elements of the set of proper congruences of $A$: ${\rm Con}(A)\setminus \{\nabla _A\}$. Let $\theta \in {\rm Con}(A)$, $a\in A$, $M\subseteq A$ and $X\subseteq A^2$, arbitrary. Then $a/\theta $ shall denote the congruence class of $a$ with respect to $\theta $, $M/\theta =\{x/\theta \ |\ x\in M\}$, $p_{\theta }:A\rightarrow A/\theta $ shall be the canonical surjective morphism: $p_{\theta }(a)=a/\theta $ for all $a\in A$, $X/\theta =\{(x/\theta ,y/\theta )\ |\ (x,y)\in X\}$ and $Cg_A(X)$ shall be the congruence of $A$ generated by $X$. It is well known that $({\rm Con}(A),\vee ,\cap ,\Delta _A,\nabla _A)$ is a bounded lattice, orderred by set inclusion, where $\phi \vee \psi =Cg_A(\phi \cup \psi )$ for all $\phi ,\psi \in {\rm Con}(A)$; moreover, this is a complete lattice, in which $\displaystyle \bigvee _{i\in I}\phi _i=Cg_A(\bigcup _{i\in I}\phi _i)$ for any family $(\phi _i)_{i\in I}\subseteq {\rm Con}(A)$. For any $a,b\in A$, the principal congruence $Cg_A(\{(a,b)\})$ shall also be denoted by $Cg_A(a,b)$. The set of the principal congruences of $A$ shall be denoted by ${\rm PCon}(A)$. ${\cal K}(A)$ shall denote the set of the finitely generated congruences of $A$, which coincide to the compact elements of the lattice ${\rm Con}(A)$. Clearly, ${\rm PCon}(A)\subseteq {\cal K}(A)$ and $\Delta _A\in {\rm PCon}(A)$, because $\Delta _A=Cg_A(x,x)$ for any $x\in A$.

Throughout the rest of this paper, $\tau $ shall be a universal algebras signature, ${\cal C}$ shall be an equational class of $\tau $--algebras $A$ and $B$ shall be algebras from ${\cal C}$ and $f:A\rightarrow B$ shall be a morphism in ${\cal C}$. Unless mentioned otherwise, by {\em morphism} we shall mean $\tau $--morphism. We recall that $A$ is said to be {\em congruence--modular}, respectively {\em congruence--distributive}, iff the lattice ${\rm Con}(A)$ is modular, respectively distributive, and that ${\cal C}$ is said to be {\em congruence--modular}, respectively {\em congruence--distributive}, iff every algebra in ${\cal C}$ is congruence--modular, respectively congruence--distributive.

\begin{remark} If $\beta \in {\rm Con}(B)$, then $f^*(\beta )\in {\rm Con}(A)$; thus ${\rm Ker}(f)=f^*(\Delta _B)\in {\rm Con}(A)$. Also, $f^*(\beta )\supseteq f^*(\Delta _B)={\rm Ker}(f)$ and $f(f^*(\beta ))=\beta \cap f(A^2)$, thus, if $f$ is surjective, then $f(f^*(\beta ))=\beta $.

If $\alpha \in {\rm Con}(A)$ such that $\alpha \supseteq {\rm Ker}(f)$, then $f(\alpha )\in {\rm Con}(f(A))$, so, if $f$ is surjective, then $f(\alpha )\in {\rm Con}(B)$. Thus, for any $\alpha \in {\rm Con}(A)$, we have $f(\alpha \vee {\rm Ker}(f))\in {\rm Con}(f(A))$, so, if $f$ is surjective, then $f(\alpha \vee {\rm Ker}(f))\in {\rm Con}(B)$. Moreover, $\alpha \mapsto f(\alpha )$ is an order isomorphism from $[{\rm Ker}(f))\in {\rm PFilt}({\rm Con}(A))$ to ${\rm Con}(f(A))$, thus to ${\rm Con}(B)$ if $f$ is surjective, having the corresponding restriction of $f^*$ as inverse.

For any $\theta \in {\rm Con}(A)$, clearly, ${\rm Ker}(p_{\theta })=\theta $. By the above, for all $\alpha \in {\rm Con}(A)$ such that $\alpha \supseteq \theta $, $\alpha /\theta =p_{\theta }(\alpha )=\{(a/\theta ,b/\theta )\ |\ (a,b)\in \alpha \}\in {\rm Con}(A/\theta )$, and $\alpha \mapsto \alpha /\theta $ is a bijection from $[\theta )$ to ${\rm Con}(A/\theta )$.\label{bijlatcongr}\end{remark}

\section{The Commutator}
\label{commutator}

This section is composed of results on the commutator in arbitrary and in congruence--modular varieties, which are either previously known of very easy to derive from previously known results. For a further study of these results, see \cite{agl}, \cite{fremck}, \cite{koll}, \cite{ouwe}.

Out of the various definitions for commutator operations on congruence lattices, we have chosen to work with the {\em term condition commutator}, from the following definition. Recall that, in algebras from congruence--modular varieties, all definitions for the commutator give the same commutator operation. For any term $t$ over $\tau $, we shall denote by $t^A$ the derivative operation of $A$ associated to $t$.

\begin{definition}{\rm \cite{mcks}} Let $\alpha ,\beta \in {\rm Con}(A)$. For any $\mu \in {\rm Con}(A)$, by $C(\alpha ,\beta ;\mu )$ we denote the fact that the following condition holds: for all $n,k\in \N $ and any term $t$ over $\tau $ of arity $n+k$, if $(a_i,b_i)\in \alpha $ for all $i\in \overline{1,n}$ and $(c_j,d_j)\in \beta $ for all $j\in \overline{1,k}$, then $(t^A(a_1,\ldots ,a_n,c_1,\ldots ,c_k),t^A(a_1,\ldots ,a_n,d_1,\ldots ,d_k))\in \mu $ iff $(t^A(b_1,\ldots ,b_n,c_1,\ldots ,c_k),t^A(b_1,\ldots ,b_n,d_1,\ldots ,d_k))\in \mu $. We denote by $[\alpha ,\beta ]_A=\bigcap \{\mu \in {\rm Con}(A)\ |\ C(\alpha ,\beta ;\mu )\}$; we call $[\alpha ,\beta ]_A$ the {\em commutator of $\alpha $ and $\beta $} in $A$.\end{definition}

\begin{remark} Let $\alpha ,\beta \in {\rm Con}(A)$. Clearly, $C(\alpha ,\beta ;\nabla _A)$. Since ${\rm Con}(A)$ is a complete lattice, it follows that $[\alpha ,\beta ]_A\in {\rm Con}(A)$. Furthermore, according to \cite[Lemma 4.4,(2)]{mcks}, for any family $(\mu _i)_{i\in I}\subseteq {\rm Con}(A)$, if $C(\alpha ,\beta ;\mu _i)$ for all $i\in I$, then $\displaystyle C(\alpha ,\beta ;\bigcap _{i\in I}\mu _i)$. Hence $C(\alpha ,\beta ;[\alpha ,\beta ]_A)$, and thus $[\alpha ,\beta ]_A=\min _\{\mu \in {\rm Con}(A)\ |\ C(\alpha ,\beta ;\mu )\}$, which is exactly the definition of the commutator from \cite{meo}.\end{remark}

\begin{definition} The operation $[\cdot ,\cdot ]_A:{\rm Con}(A)\times {\rm Con}(A)\rightarrow {\rm Con}(A)$ is called the {\em commutator of $A$}.\label{1.1}\end{definition}

\begin{theorem}{\rm \cite{fremck}} If ${\cal C}$ is congruence--modular, then, for each member $M$ of ${\cal C}$, $[\cdot ,\cdot ]_M$ is the unique binary operation on ${\rm Con}(M)$ such that, for all $\alpha ,\beta \in {\rm Con}(M)$, $[\alpha ,\beta ]_M=\min \{\mu \in {\rm Con}(M)\ |\ \mu \subseteq \alpha \cap \beta $ and, for any member $N$ of ${\cal C}$ and any surjective morphism $h:M\rightarrow N$ in ${\cal C}$, $\mu \vee {\rm Ker}(h)=h^*([h(\alpha \vee {\rm Ker}(h)),h(\beta \vee {\rm Ker}(h))]_N)\}$.\label{wow}\end{theorem}

\begin{theorem}{\rm \cite{bj}} If ${\cal C}$ is congruence--distributive, then, in each member of ${\cal C}$, the commutator coincides to the intersection of congruences.\label{distrib}\end{theorem}

For brevity, most of the times, we shall use the remarks in this paper without referencing them, and the same goes for the lemmas and propositions that state basic results.

\begin{proposition}{\rm \cite[Lemma 4.6,Lemma 4.7,Theorem 8.3]{mcks}} The commutator is:\begin{itemize}
\item increasing in both arguments, that is, for all $\alpha ,\beta ,\phi ,\psi \in {\rm Con}(A)$, if $\alpha \subseteq \beta $ and $\phi \subseteq \psi $, then $[\alpha ,\phi ]_A\subseteq [\beta ,\psi ]_A$;
\item smaller than its arguments, so, for any $\alpha ,\beta \in {\rm Con}(A)$, $[\alpha ,\beta ]_A\subseteq \alpha \cap \beta $.\end{itemize}

If ${\cal C}$ is congruence--modular, then the commutator is also:\begin{itemize}

\item commutative, that is $[\alpha ,\beta ]_A=[\beta ,\alpha ]_A$ for all $\alpha ,\beta \in {\rm Con}(A)$;
\item distributive in both arguments with respect to arbitrary joins, that is, for any families $(\alpha _i)_{i\in I}$ and $(\beta _j)_{j\in J}$ of congruences of $A$, $\displaystyle [\bigvee _{i\in I}\alpha _i,\bigvee _{j\in J}\beta _j]_A=\bigvee _{i\in I}\bigvee _{j\in J}[\alpha _i,\beta _j]_A$.\end{itemize}\label{1.3}\end{proposition}

\begin{remark} Assume that $[\cdot ,\cdot ]_A$ is commutative. Then the distributivity of $[\cdot ,\cdot ]_A$ in both arguments w.r.t. arbitrary joins is equivalent to its distributivity in one argument w.r.t. arbitrary joins, which in turn is equivalent to its distributivity w.r.t. the join in the case when ${\rm Con}(A)$ is finite, in particular when $A$ is finite.

Obviously, if $[\cdot ,\cdot ]_A$ equals the intersection and it is distributive w.r.t. the join (by Proposition \ref{1.3}, the latter holds if ${\cal C}$ is congruence--modular), then $A$ is congruence--distributive.\end{remark}

\begin{lemma}{\rm \cite{fremck}} If ${\cal C}$ is congruence--modular and $S$ is a subalgebra of $A$, then, for any $\alpha ,\beta \in {\rm Con}(A)$, $[\alpha \cap S^2,\beta \cap S^2]_S\subseteq [\alpha ,\beta ]_A\cap S^2$.\label{1.8}\end{lemma}

\begin{proposition}{\rm \cite[Theorem 5.17, p. 48]{ouwe}} Assume that ${\cal C}$ is congruence--modular, and let $n\in \N ^*$, $M_1,\ldots ,M_n$ be algebras from ${\cal C}$, $\displaystyle M=\prod _{i=1}^nM_i$ and, for all $i\in \overline{1,n}$, $\alpha _i,\beta _i\in {\rm Con}(M_i)$. Then: $\displaystyle [\prod _{i=1}^n\alpha _i,\prod _{i=1}^n\beta _i]_M=\prod _{i=1}^n[\alpha _i,\beta _i]_{M_i}$.\label{comutprod}\end{proposition}

\begin{remark} By Theorem \ref{wow} and Remark \ref{bijlatcongr}, if ${\cal C}$ is congruence--modular, $\alpha ,\beta ,\theta \in {\rm Con}(A)$ and $f$ is surjective, then $[f(\alpha \vee {\rm Ker}(f)),f(\beta \vee {\rm Ker}(f))]_B=f([\alpha ,\beta ]_A\vee {\rm Ker}(f))$, thus $[(\alpha \vee \theta )/\theta ,(\beta \vee \theta )/\theta]_B=([\alpha ,\beta ]_A\vee \theta )/\theta $, hence, if $\theta \subseteq [\alpha ,\beta ]_A$, then $[\alpha /\theta ,\beta /\theta ]_{A/\theta }=[\alpha ,\beta ]_A/\theta $.\label{fsurjcomut}\end{remark}

\begin{definition}{\rm \cite{fremck}} Let $\phi $ be a proper congruence of $A$. Then $\phi $ is called a {\em prime congruence} of $A$ iff, for all $\alpha ,\beta \in {\rm Con}(A)$, $[\alpha ,\beta ]_A\subseteq \phi $ implies $\alpha \subseteq \phi $ or $\beta \subseteq \phi $. $\phi $ is called a {\em semiprime congruence} of $A$ iff, for all $\alpha \in {\rm Con}(A)$, $[\alpha ,\alpha ]_A\subseteq \phi $ implies $\alpha \subseteq \phi $.\label{1.4}\end{definition}

The set of the prime congruences of $A$ shall be denoted by ${\rm Spec}(A)$. ${\rm Spec}(A)$ is called the {\em (prime) spectrum} of $A$ and ${\rm Max}(A)$ is called the {\em maximal spectrum} of $A$.

Following \cite{koll}, we say that ${\cal C}$ is {\em semi--degenerate} iff no non--trivial algebra in ${\cal C}$ has one--element subalgebras. For instance, the class of unitary rings and any class of bounded orderred structures is semi--degenerate.

\begin{lemma}{\rm \cite[Theorem $5.3$]{agl}} If ${\cal C}$ is congruence--modular and semi--degenerate, then:\begin{itemize}
\item any proper congruence of $A$ is included in a maximal congruence of $A$;
\item any maximal congruence of $A$ is prime.\end{itemize}\label{folclor}\end{lemma}

\begin{remark} By Lemma \ref{folclor}, if $A$ is non--trivial and ${\cal C}$ is congruence--modular and semi--degenerate, then $A$ has maximal congruences, thus it has prime congruences.\end{remark}

\begin{proposition}{\rm \cite{koll}} ${\cal C}$ is semi--degenerate iff, for all members $M$ of ${\cal C}$, $\nabla _M\in {\cal K}(M)$.\label{2.6}\end{proposition}

\begin{proposition}{\rm \cite[Theorem 8.5, p. 85]{fremck}} If ${\cal C}$ is congruence--modular, then the following are equivalent:\begin{enumerate}
\item for any algebra $M$ from ${\cal C}$, $[\nabla _M,\nabla _M]_M=\nabla _M$;
\item for any algebra $M$ from ${\cal C}$ and any $\theta \in {\rm Con}(M)$, $[\theta ,\nabla _M]_M=\theta $;
\item ${\cal C}$ has no skew congruences, that is, for any algebras $M$ and $N$ from ${\cal C}$, ${\rm Con}(M\times N)=\{\theta \times \zeta \ |\ \theta \in {\rm Con}(M),\zeta \in {\rm Con}(N)\}$.\end{enumerate}\label{prodcongr}\end{proposition}

\begin{lemma}\begin{enumerate}
\item\label{distribsemid1} If ${\cal C}$ is congruence--modular and semi--degenerate, then ${\cal C}$ fulfills the equivalent conditions from Proposition \ref{prodcongr}.
\item\label{distribsemid2} If ${\cal C}$ is congruence--distributive, then ${\cal C}$ fulfills the equivalent conditions from Proposition \ref{prodcongr}.\end{enumerate}\label{distribsemid}\end{lemma}

\begin{proof} (\ref{distribsemid1}) This is exactly \cite[Lemma 5.2]{agl}.

\noindent (\ref{distribsemid2}) Clear, from Theorem \ref{distrib}.\end{proof}

\begin{lemma}{\rm \cite[Lemma $1.11$]{bak}, \cite[Proposition $1.2$]{urs5}} If $f$ is surjective, then, for any $a,b\in A$, any $X\subseteq A^2$, any $\theta \in {\rm Con}(A)$ and any $\alpha ,\beta \in [{\rm Ker}(f))$:\begin{enumerate}
\item\label{fsurjcongr0} $f(\theta \vee {\rm Ker}(f))=Cg_B(f(\theta ))$; $f(\alpha \vee \beta )=f(\alpha )\vee f(\beta )$;
\item\label{fsurjcongr1} $f(Cg_A(a,b)\vee {\rm Ker}(f))=Cg_B(f(a),f(b))$; $f(Cg_A(X)\vee {\rm Ker}(f))=Cg_B(f(X))$;
\item\label{fsurjcongr2} $(Cg_A(a,b)\vee \theta )/\theta =Cg_{A/\theta }(a/\theta ,b/\theta )$; $(Cg_A(X)\vee \theta )/\theta =Cg_{A/\theta }(X/\theta )$.\end{enumerate}\label{fsurjcongr}\end{lemma}

We say that $A$ has {\em principal commutators} iff, for all $\alpha ,\beta \in {\rm PCon}(A)$, we have $[\alpha ,\beta ]_A\in {\rm PCon}(A)$, that is iff ${\rm PCon}(A)$ is closed with respect to the commutator of $A$. Following \cite{agl}, we say that ${\cal C}$ has {\em principal commutators} iff each member of ${\cal C}$ has principal commutators. We say that ${\cal C}$ has {\em associative commutators} iff, for each member $M$ of ${\cal C}$, the commutator of $M$ is an associative binary operation on ${\rm Con}(M)$.

\begin{remark} $\displaystyle {\cal K}(A)=\{Cg_A(\emptyset )\}\cup \{Cg_A(\{(a_1,b_1),\ldots ,(a_n,b_n)\})\ |\ n\in \N ^*,a_1,b_1,\ldots ,a_n,b_n\in A\}=\{\Delta _A\}\cup \{\bigvee _{i=1}^nCg_A(a_i,b_i)\ |\ n\in \N ^*,a_1,b_1,\ldots ,a_n,b_n\in A\}=\{\bigvee _{i=1}^nCg_A(a_i,b_i)\ |\ n\in \N ^*,a_1,b_1,\ldots ,a_n,b_n\in A\}$, since $\Delta _A\in {\rm PCon}(A)$. From this, it is immediate that ${\cal K}(A)$ is closed with respect to finite joins, and, if $A$ has principal commutators and $[\cdot ,\cdot ]_A$ is commutative and distributive w.r.t. the join (for instance if ${\cal C}$ is congruence--modular), then ${\cal K}(A)$ is also closed with respect to the commutator of $A$.\label{ka}\end{remark}

\begin{remark} If ${\cal C}$ is congruence--distributive, then, as shown by Theorem \ref{distrib}:\begin{itemize}
\item ${\cal C}$ has principal commutators iff ${\cal C}$ has the principal intersection property (PIP);
\item ${\cal K}(M)$ is closed with respect to the commutator for each member $M$ of ${\cal C}$ iff ${\cal C}$ has the compact intersection property (CIP).\end{itemize}

As a particular case of Remark \ref{ka}, if ${\cal C}$ is congruence--distributive and has the PIP, then ${\cal C}$ has the CIP.\label{pipcip}\end{remark}

\begin{example}{\cite{agl}, \cite{blkpgz}, \cite{gulo}, \cite[Theorem 2.8]{bj}, \cite{kap}, \cite{lpt}} As shown by Theorem \ref{distrib}, any congruence--distributive variety has associative commutators. The variety of commutative unitary rings is semi--degenerate, congruence--modular, with principal commutators and associative commutators, and it is not congruence--distributive. Out of the semi--degenerate congruence--distributive varieties with the CIP, we mention semi--degenerate filtral varieties. Out of the semi--degenerate congruence--distributive varieties with the PIP, we mention: bounded distributive lattices, residuated lattices (a variety which includes ${\rm G\ddot{o}del}$ algebras, product algebras, MTL--algebras, BL--algebras, MV--algebras) and semi--degenerate discriminator varieties (out of which we mention Boolean algebras, $n$--valued Post algebras, $n$--valued \L ukasiewicz algebras, $n$--valued MV--algebras, $n$--dimensional cylindric algebras, ${\rm G\ddot{o}del}$ residuated lattices).\label{asoccom}\end{example}

\section{The Stone Topologies on Prime and Maximal Spectra}
\label{stonetops}

In what follows, we present the Stone topologies on the prime and maximal spectra of ideals and filters of a bounded distributive lattice and those of congruences of an algebra with the greatest congruence compact from a congruence--modular variety; in particular, the following hold for algebras from semi--degenerate congruence--modular varieties. The results in this section are either previously known or very easy to derive from previously known results; see, for instance, \cite{joh}.

Let $L$ be a bounded distributive lattice. For any $I\in {\rm Id}(L)$ and any $a\in L$, we shall denote by $V_{{\rm Id},L}(I)={\rm Spec}_{\rm Id}(L)\cap [I)=\{P\in {\rm Spec}_{\rm Id}(L)\ |\ I\subseteq P\}$, $D_{{\rm Id},L}(I)={\rm Spec}_{\rm Id}(L)\setminus V_{{\rm Id},L}(I)=\{Q\in {\rm Spec}_{\rm Id}(L)\ |\ I\nsubseteq Q\}$, $V_{{\rm Id},L}(a)=V_{{\rm Id},L}((a])=\{P\in {\rm Spec}_{\rm Id}(L)\ |\ a\in P\}$ and $D_{{\rm Id},L}(a)=D_{{\rm Id},L}((a])={\rm Spec}_{\rm Id}(L)\setminus V_{{\rm Id},L}(a)=\{Q\in {\rm Spec}_{\rm Id}(L)\ |\ a\notin Q\}$. By replacing ${\rm Spec}_{\rm Id}(L)$ with ${\rm Spec}_{\rm Filt}(L)$, in the same way we can define $V_{{\rm Filt},L}(F)$, $D_{{\rm Filt},L}(F)$, $V_{{\rm Filt},L}(a)$ and $D_{{\rm Filt},L}(a)$ for any $F\in {\rm Filt}(L)$ and any $a\in L$.

\begin{remark} The following hold, and their duals hold for filters:\begin{itemize}
\item for any $J,K\in {\rm Id}(L)$, $V_{{\rm Id},L}(J\cap K)=V_{{\rm Id},L}(J)\cup V_{{\rm Id},L}(K)$ and $D_{{\rm Id},L}(J\cap K)=D_{{\rm Id},L}(J)\cap D_{{\rm Id},L}(K)$;
\item for any family $(J_i)_{i\in I}\subseteq {\rm Id}(L)$, $\displaystyle V_{{\rm Id},L}(\bigvee _{i\in I}J_i)=\bigcap _{i\in I}V_{{\rm Id},L}(J_i)$ and $\displaystyle D_{{\rm Id},L}(\bigvee _{i\in I}J_i)=\bigcup _{i\in I}D_{{\rm Id},L}(J_i)$;
\item thus, for any $a,b\in L$, $V_{{\rm Id},L}(a\wedge b)=V_{{\rm Id},L}(a)\cup V_{{\rm Id},L}(b)$, $D_{{\rm Id},L}(a\wedge b)=D_{{\rm Id},L}(a)\cap D_{{\rm Id},L}(b)$, $V_{{\rm Id},L}(a\vee b)=V_{{\rm Id},L}(a)\cap V_{{\rm Id},L}(b)$ and $D_{{\rm Id},L}(a\vee b)=D_{{\rm Id},L}(a)\cup D_{{\rm Id},L}(b)$;
\item if $L$ is a complete lattice, then, for any family $(a_i)_{i\in I}\subseteq L$, $\displaystyle V_{{\rm Id},L}(\bigvee _{i\in I}a_i)=\bigcap _{i\in I}V_{{\rm Id},L}(a_i)$ and $\displaystyle D_{{\rm Id},L}(\bigvee _{i\in I}J_i)=\bigcup _{i\in I}D_{{\rm Id},L}(J_i)$;
\item if $I\in {\rm Id}(L)$, then: $D_{{\rm Id},L}(I)={\rm Spec}_{\rm Id}(L)$ iff $V_{{\rm Id},L}(I)=\emptyset $ iff $I=L$;
\item $D_{{\rm Id},L}(\{0\})=\emptyset $ and $V_{{\rm Id},L}(\{0\})={\rm Spec}_{\rm Id}(L)$;
\item if $L$ is distributive (so that the Prime Ideal Theorem holds in $L$ and,  hence, any ideal of $L$ equals the intersection of the prime ideals that include it) and $I\in {\rm Id}(L)$, then: $D_{{\rm Id},L}(I)=\emptyset $ iff $V_{{\rm Id},L}(I)={\rm Spec}_{\rm Id}(L)$ iff $I=\{0\}$.\end{itemize}\label{stoneidlat}\end{remark}

As shown by Remark \ref{stoneidlat}, $\{D_{{\rm Id},L}(I)\ |\ I\in {\rm Id}(L)\}$ is a topology on ${\rm Spec}_{\rm Id}(L)$, called the {\em Stone topology}, having $\{D_{{\rm Id},L}(a)\ |\ a\in L\}$ as a basis and, obviously, $\{V_{{\rm Id},L}(I)\ |\ I\in {\rm Id}(L)\}$ as the family of closed sets and $\{V_{{\rm Id},L}(a)\ |\ a\in L\}$ as a basis of closed sets. Since ${\rm Max}_{\rm Id}(L)\subseteq {\rm Spec}_{\rm Id}(L)$, $\{D_{{\rm Id},L}(I)\cap {\rm Max}_{\rm Id}(L)\ |\ I\in {\rm Id}(L)\}$ is a topology on ${\rm Max}_{\rm Id}(L)$, which is also called the {\em Stone topology}, and it has $\{D_{{\rm Id},L}(a)\cap {\rm Max}_{\rm Id}(L)\ |\ a\in L\}$ as a basis, $\{V_{{\rm Id},L}(I)\cap {\rm Max}_{\rm Id}(L)\ |\ I\in {\rm Id}(L)\}$ as the family of closed sets and $\{V_{{\rm Id},L}(a)\cap {\rm Max}_{\rm Id}(L)\ |\ a\in L\}$ as a basis of closed sets. Dually, we have the Stone topologies on ${\rm Spec}_{\rm Filt}(L)$ and ${\rm Max}_{\rm Filt}(L)$. ${\rm Spec}_{\rm Id}(L)$, ${\rm Max}_{\rm Id}(L)$, ${\rm Spec}_{\rm Filt}(L)$ and ${\rm Max}_{\rm Filt}(L)$ are called the {\em (prime) spectrum of ideals}, {\em maximal spectrum of ideals}, {\em (prime) spectrum of filters} and {\em maximal spectrum of filters} of $L$, respectively.

Throughout the rest of this section, we shall assume that $[\cdot ,\cdot ]_A$ is commutative and distributive w.r.t. arbitrary joins. For each $\theta \in {\rm Con}(A)$, we shall denote by $V_A(\theta )={\rm Spec}(A)\cap [\theta )=\{\phi \in {\rm Spec}(A)\ |\ \theta \subseteq \phi \}$ and by $D_A(\theta )={\rm Spec}(A)\setminus V_A(\theta )=\{\psi \in {\rm Spec}(A)\ |\ \theta \nsubseteq \psi \}$. We shall also denote, for any $a,b\in A$, by $V_A(a,b)=V_A(Cg_A(a,b))=\{\phi \in {\rm Spec}(A)\ |\ (a,b)\in \phi \}$ and by $D_A(a,b)=D_A(Cg_A(a,b))=\{\psi \in {\rm Spec}(A)\ |\ (a,b)\notin \psi \}$. The proof of the following result is straightforward.

\begin{proposition}{\rm \cite{agl}} $({\rm Spec}(A),\{D_A(\theta )\ |\ \theta \in {\rm Con}(A)\})$ is a topological space, having $\{D_A(a,b)\ |\ a,b\in A\}$ as a basis and in which, for all $\alpha ,\beta \in {\rm Con}(A)$ and any family $(\alpha _i)_{i\in I}\subseteq {\rm Con}(A)$, the following hold:\begin{enumerate}
\item\label{stonetop1} $D_A(\Delta _A)=\emptyset $ and $D_A(\nabla _A)={\rm Spec}(A)$; $V_A(\Delta _A)={\rm Spec}(A)$ and $V_A(\nabla _A)=\emptyset $;
\item\label{stonetop2} $D_A([\alpha ,\beta ]_A)=D_A(\alpha \cap \beta )=D_A(\alpha )\cap D_A(\beta )=$; $V_A([\alpha ,\beta ]_A)=V_A(\alpha \cap \beta )=V_A(\alpha )\cup V_A(\beta )$;
\item\label{stonetop3} $\displaystyle D_A(\bigvee _{i\in I}\alpha _i)=\bigcup _{i\in I}D_A(\alpha _i)$; $\displaystyle V_A(\bigvee _{i\in I}\alpha _i)=\bigcap _{i\in I}V_A(\alpha _i)$.\end{enumerate}\label{stonetop}\end{proposition}

$\{D_A(\theta )\ |\ \theta \in {\rm Con}(A)\}$ is called the {\em Stone topology} on ${\rm Spec}(A)$. Obviously, its family of closed sets is $\{V_A(\theta )\ |\ \theta \in {\rm Con}(A)\}$, and $\{V_A(a,b)\ |\ a,b\in A\}$ is a basis of closed sets for this topology. The Stone topology on ${\rm Spec}(A)$ induces the {\em Stone topology} on ${\rm Max}(A)$, namely $\{D_A(\theta )\cap {\rm Max}(A)\ |\ \theta \in {\rm Con}(A)\}$.

\begin{remark} Let $\alpha ,\beta \in {\rm Con}(A)$. Then, clearly:\begin{itemize}
\item $V_A(\alpha )\subseteq V_A(\beta )$ iff ${\rm Spec}(A)\setminus D_A(\alpha )\subseteq {\rm Spec}(A)\setminus D_A(\beta )$ iff $D_A(\beta )\subseteq D_A(\alpha )$;
\item if $\alpha \subseteq \beta $, then $V_A(\beta )\subseteq V_A(\alpha )$ and $D_A(\alpha )\subseteq D_A(\beta )$.\end{itemize}\label{dvinclus}\end{remark}

\begin{proposition} If ${\cal C}$ is congruence--modular and semi--degenerate, then, for any $\alpha \in {\rm Con}(A)$: $D_A(\alpha )={\rm Spec}(A)$ iff $V_A(\alpha )=\emptyset $ iff $\alpha =\nabla _A$.\label{sdtop}\end{proposition}

\begin{proof} $D_A(\alpha )={\rm Spec}(A)$ iff ${\rm Spec}(A)\setminus D_A(\alpha )=\emptyset $ iff $V_A(\alpha )=\emptyset $. Since ${\rm Spec}(A)\subseteq {\rm Con}(A)\setminus \{\nabla _A\}$, we have $V_A(\nabla _A)=\emptyset $, which was also part of Proposition \ref{stonetop}. If $\alpha \neq \nabla _A$, then, according to Lemma \ref{folclor}, there exists a $\phi \in {\rm Spec}(A)$ such that $\alpha \subseteq \phi $, that is $V_A(\alpha )\neq \emptyset $.\end{proof}

\begin{remark} Recall that, if $f$ is surjective, then the map $\alpha \mapsto f(\alpha )$ is a lattice isomorphism from $[{\rm Ker}(f))$ to ${\rm Con}(B)$. Now assume that ${\cal C}$ is congruence--modular.

Then this map is an order isomorphism from ${\rm Max}(A)\cap [{\rm Ker}(f))$ to ${\rm Max}(B)$. Furthermore, this map is an order isomorphism from ${\rm Spec}(A)\cap [{\rm Ker}(f))$ to ${\rm Spec}(B)$ (see also \cite{agl}, \cite{gulo}, \cite{euadm}). Hence, if ${\rm Ker}(f)\subseteq \alpha \in {\rm Con}(A)$, then $V_B(f(\alpha ))=f(V_A(\alpha ))$ and $[f(\alpha ))\cap {\rm Max}(B)=f([\alpha )\cap {\rm Max}(A))$.

Therefore, for all $\theta \in {\rm Con}(A)$, the map $\alpha \mapsto \alpha /\theta $ is a lattice isomorphism from $[\theta )$ to ${\rm Con}(A/\theta )$, an order isomorphism from ${\rm Max}(A)\cap [\theta )$ to ${\rm Max}(A/\theta )$ and an order isomorphism from ${\rm Spec}(A)\cap [\theta )$ to ${\rm Spec}(A/\theta )$; hence, if $\theta \subseteq \alpha \in {\rm Con}(A)$, then $V_{A/\theta }(\alpha /\theta )=\{\psi /\theta \ |\ \psi \in V_A(\alpha )\}$ and $[\alpha /\theta )\cap {\rm Max}(A/\theta )=\{\psi /\theta \ |\ \psi \in [\alpha )\cap {\rm Max}(A)\}$.\label{specquo}\end{remark}

\section{The Construction of the Reticulation of a Universal Algebra and Related Results}
\label{reticulatia}

Throughout this section, we shall assume that $[\cdot ,\cdot ]_A$ is commutative and distributive w.r.t. arbitrary joins, and that $\nabla _A\in {\cal K}(A)$. For every $\theta \in {\rm Con}(A)$, we shall denote by $\rho _A(\theta )$ the {\em radical} of $\theta $, that is the intersection of the prime congruences of $A$ which include $\theta $: $\displaystyle \rho _A(\theta )=\bigcap \{\phi \in {\rm Spec}(A)\ |\ \theta \subseteq \phi \}=\bigcap _{\phi \in V_A(\theta )}\phi $.

\begin{remark} Let $\alpha ,\beta \in {\rm Con}(A)$ and $\phi \in {\rm Spec}(A)$. Then, clearly:\begin{enumerate}
\item\label{clara0} $V_A(\nabla _A)=\emptyset $, and thus $\rho _A(\nabla _A)=\nabla _A$;
\item\label{clara1} $\rho _A(\phi )=\phi $; moreover, $\rho _A(\alpha )=\alpha $ iff $\alpha $ is the intersection of a family of prime congruences of $A$;
\item\label{clara2} if $\alpha \subseteq \beta $, then $V_A(\alpha )\supseteq V_A(\beta )$, hence $\rho _A(\alpha )\subseteq \rho _A(\beta )$;
\item\label{clara4} if $\alpha \subseteq \phi $, then $\rho _A(\alpha )\subseteq \phi $, since $\phi \in V_A(\alpha )$.\end{enumerate}\label{clara}\end{remark}

Following \cite{agl}, for any  $\alpha ,\beta \in {\rm Con}(A)$ and every $n\in \N ^*$, we denote by $[\alpha ,\beta ]_A^1=[\alpha ,\beta ]_A$ and $[\alpha ,\beta ]_A^{n+1}=[[\alpha ,\beta ]_A^n,[\alpha ,\beta ]_A^n]_A$, and by $(\alpha ,\beta ]_A^1=[\alpha ,\beta ]_A$ and $(\alpha ,\beta ]_A^{n+1}=(\alpha ,(\alpha ,\beta ]_A^n]_A$.

\begin{lemma} For all $n\in \N ^*$, any $\alpha ,\beta \in {\rm Con}(A)$ and any family $(\alpha _i)_{i\in I}\in {\rm Con}(A)$:\begin{enumerate}
\item\label{imediata1} $\alpha \subseteq \rho _A(\alpha )$;
\item\label{imediata2} $V_A(\alpha )=V_A(\rho _A(\alpha ))$;
\item\label{imediata3} $\displaystyle V_A(\bigvee _{i\in I}\alpha _i)=V_A(\bigvee _{i\in I}\rho _A(\alpha _i))$;
\item\label{imediata4} $V_A([\alpha ,\beta ]_A^n)=V_A([\alpha ,\beta ]_A)=V_A(\alpha \cap \beta )=V_A(\alpha )\cup V_A(\beta )$;
\item\label{imediata5} $V_A([\alpha ,\alpha ]_A^n)=V_A([\alpha ,\alpha ]_A)=V_A(\alpha )$.\end{enumerate}\label{imediata}\end{lemma}

\begin{proof} (\ref{imediata1}) Trivial.

\noindent (\ref{imediata2}) By (\ref{imediata1}) and Remark \ref{dvinclus}, $V_A(\rho _A(\alpha ))\subseteq V_A(\alpha )$. If $\phi \in V_A(\alpha )$, then  $\phi \in V_A(\rho _A(\alpha ))$, according to Remark \ref{clara}, (\ref{clara4}), thus $V_A(\alpha )\subseteq V_A(\rho _A(\alpha ))$. Hence $V_A(\alpha )=V_A(\rho _A(\alpha ))$.

\noindent (\ref{imediata3}) By (\ref{imediata2}) and Proposition \ref{stonetop}, (\ref{stonetop3}), $\displaystyle V_A(\bigvee _{i\in I}\alpha _i)=\bigcap _{i\in I}V_A(\alpha _i)=\bigcap _{i\in I}V_A(\rho _A(\alpha _i))=V_A(\bigvee _{i\in I}\rho _A(\alpha _i))$.

\noindent (\ref{imediata4}) By Proposition \ref{stonetop}, (\ref{stonetop2}). $V_A([\alpha ,\beta ]_A)=V_A(\alpha \cap \beta )=V_A(\alpha )\cup V_A(\beta )$ Now we prove that $V_A([\alpha ,\beta ]_A^n)=V_A(\alpha )\cup V_A(\beta )$ by induction on $n\in \N ^*$. $V_A([\alpha ,\beta ]_A^1)=V_A([\alpha ,\beta ]_A)=V_A(\alpha )\cup V_A(\beta )$. Now let $n\in \N ^*$ such that $V_A([\theta ,\zeta ]_A^n)=V_A(\theta )\cup V_A(\zeta )$ for all $\theta ,\zeta \in {\rm Con}(A)$. Then $V_A([\alpha ,\beta ]_A^{n+1})=V_A([[\alpha ,\beta ]_A^n,[\alpha ,\beta ]_A^n]_A)=V_A([\alpha ,\beta ]_A^n)\cup V_A([\alpha ,\beta ]_A^n)=V_A([\alpha ,\beta ]_A^n)=V_A(\alpha )\cup V_A(\beta )$.

\noindent (\ref{imediata5}) By (\ref{imediata4}).\end{proof}

\begin{proposition} For all $\alpha ,\beta ,\theta \in {\rm Con}(A)$, the following hold:\begin{enumerate}
\item\label{esential1} $\rho _A(\alpha )\subseteq \rho _A(\beta )$ iff $\alpha \subseteq \rho _A(\beta )$ iff $V_A(\alpha )\supseteq V_A(\beta )$;
\item\label{esential2} $\rho _A(\alpha )=\rho _A(\beta )$ iff $V_A(\alpha )=V_A(\beta )$;
\item\label{esential3} if $\theta \subseteq \alpha $, then $\rho _{A/\theta }(\alpha /\theta )=\rho _A(\alpha )/\theta $;
\item\label{esential4} $\rho _{A/\theta }(\Delta _{A/\theta })=\rho _A(\theta )/\theta $;
\item\label{esential5} $\rho _{A/\theta }((\alpha \vee \theta )/\theta )=\rho _A(\alpha \vee \theta )/\theta $.\end{enumerate}\label{esential}\end{proposition}

\begin{proof} (\ref{esential1}) Clearly, if $V_A(\alpha )\supseteq V_A(\beta )$, then $\rho _A(\alpha )\subseteq \rho _A(\beta )$. If $\rho _A(\alpha )\subseteq \rho _A(\beta )$, then, since $\alpha \subseteq \rho _A(\alpha )$, it follows that $\alpha \subseteq \rho _A(\beta )$. Finally, if $\alpha \subseteq \rho _A(\beta )$, then $V_A(\alpha )\supseteq V_A(\rho _A(\beta ))=V_A(\beta )$, by Remark \ref{clara}, (\ref{clara2}), and Lemma \ref{imediata}, (\ref{imediata2}).

\noindent (\ref{esential2}) By (\ref{esential1}).

\noindent (\ref{esential3}) If $\theta \subseteq \alpha $, then we may write: $\displaystyle \rho _{A/\theta }(\alpha /\theta )=\bigcap _{\psi \in V_{A/\theta }(\alpha /\theta )}\psi =\bigcap _{\phi \in V_A(\alpha )}\phi /\theta =(\bigcap _{\phi \in V_A(\alpha )}\phi )/\theta =\rho _A(\alpha )/\theta $.

\noindent (\ref{esential4}) By (\ref{esential3}), $\rho _{A/\theta }(\Delta _{A/\theta })=\rho _{A/\theta }(\theta /\theta )=\rho _A(\theta )/\theta $.

\noindent (\ref{esential5}) By (\ref{esential3}).\end{proof}

\begin{proposition} For any $n\in \N ^*$, any $\alpha \in {\rm Con}(A)$ and any family $(\alpha _i)_{i\in I}\subseteq {\rm Con}(A)$:\begin{enumerate}
\item\label{onrho1} if ${\cal C}$ is congruence--modular and semi--degenerate, then: $\rho _A(\alpha )=\nabla _A$ iff $\alpha =\nabla _A$;
\item\label{onrho2} $\rho _A([\alpha ,\beta ]_A^n)=\rho _A([\alpha ,\beta ]_A)=\rho _A(\alpha \cap \beta )=\rho _A(\alpha )\cap \rho _A(\beta )$;
\item\label{onrho0} $\rho _A([\alpha ,\alpha ]_A^n)=\rho _A([\alpha ,\alpha ]_A)=\rho _A(\alpha )$;
\item\label{onrho3} $\rho _A(\rho _A(\alpha ))=\rho _A(\alpha )$;
\item\label{onrho4} $\displaystyle \rho _A(\bigvee _{i\in I}\rho _A(\alpha _i))=\rho _A(\bigvee _{i\in I}\alpha _i)$;
\item\label{onrho5} if ${\cal C}$ is congruence--modular and semi--degenerate, then: $\displaystyle \bigvee _{i\in I}\rho _A(\alpha _i)=\nabla _A$ iff $\displaystyle \bigvee _{i\in I}\alpha _i=\nabla _A$.\end{enumerate}\label{onrho}\end{proposition}

\begin{proof} (\ref{onrho1}) By Lemma \ref{imediata}, (\ref{imediata1}), $\nabla _A\subseteq \rho _A(\nabla _A)$, thus $\rho _A(\nabla _A)=\nabla _A$. If $\alpha \neq \nabla _A$, then there exists $\phi \in V_A(\alpha )$, thus $\rho _A(\alpha )\subseteq \phi \subsetneq \nabla _A$.

\noindent (\ref{onrho2}) By Remark \ref{clara}, (\ref{clara2}), Lemma \ref{imediata}, (\ref{imediata4}), and Proposition \ref{stonetop}, (\ref{stonetop2}), $\displaystyle \rho _A([\alpha ,\beta ]_A^n)=\rho _A([\alpha ,\beta ]_A)=\rho _A(\alpha \cap \beta )=\bigcap _{\phi \in V_A(\alpha \cap \beta )}\phi =\bigcap _{\phi \in V_A(\alpha )\cup V_A(\beta )}=\bigcap _{\phi \in V_A(\alpha )}\phi \cap \bigcap _{\phi \in V_A(\beta )}\phi =\rho _A(\alpha )\cap \rho _A(\beta )$.

\noindent (\ref{onrho0}) By (\ref{onrho2}).

\noindent (\ref{onrho3}) By Remark \ref{clara}, (\ref{clara2}), and Lemma \ref{imediata}, (\ref{imediata2}).

\noindent (\ref{onrho4}) By Remark \ref{clara}, (\ref{clara2}), and Lemma \ref{imediata}, (\ref{imediata3}), $\displaystyle \rho _A(\bigvee _{i\in I}\rho _A(\alpha _i))=\rho _A(\bigvee _{i\in I}\alpha _i)$.

\noindent (\ref{onrho5}) By (\ref{onrho4}) and (\ref{onrho1}), $\displaystyle \bigvee _{i\in I}\rho _A(\alpha _i)=\nabla _A$ iff $\displaystyle \rho _A(\bigvee _{i\in I}\rho _A(\alpha _i))=\nabla _A$ iff $\displaystyle \rho _A(\bigvee _{i\in I}\alpha _i)=\nabla _A$ iff $\displaystyle \bigvee _{i\in I}\alpha _i=\nabla _A$.\end{proof}

The {\em radical congruences} of $A$ are the congruences $\alpha $ of $A$ such that $\alpha =\rho _A(\alpha )$. Let us denote by ${\rm RCon}(A)$ the set of the radical congruences of $A$.

\begin{remark} By Remark \ref{clara}, (\ref{clara1}), ${\rm Spec}(A)\subseteq {\rm RCon}(A)$; moreover, the elements of ${\rm RCon}(A)$ are exactly the intersections of prime congruences of $A$.\label{radspec}\end{remark}

\begin{remark} ${\rm RCon}(A)=\{\alpha \in {\rm Con}(A)\ |\ \alpha =\rho _A(\alpha )\}=\{\rho _A(\alpha )\ |\ \alpha \in {\rm Con}(A)\}$. Indeed, the first of these equalities is the definition of ${\rm RCon}(A)$ and the second equality follows from Proposition \ref{onrho}, (\ref{onrho3}).\label{radical}\end{remark}

\begin{proposition} If the commutator of $A$ equals the intersection, in particular if ${\cal C}$ is congruence--distributive, then ${\rm RCon}(A)={\rm Con}(A)$.\label{radcongrdistrib}\end{proposition}

\begin{proof} By \cite[Lemma $1.6$]{agl}, the radical congruences of $A$ coincide to its semiprime congruences, that is the congruences $\theta $ of $A$ such that, for all $\alpha \in {\rm Con}(A)$, $[\alpha ,\alpha ]_A\subseteq \theta $ implies $\alpha \subseteq \theta $. Clearly, if $[\cdot ,\cdot ]_A=\cap $, then every congruence of $A$ is semiprime, and thus radical.\end{proof}

Most of the previous results on the radicals of congruences are known, but, for the sake of completeness, we have provided short proofs for them. For any $\alpha ,\beta \in {\rm Con}(A)$, let us denote by $\alpha \dotvee \beta =\rho _A(\alpha \vee \beta )$. For any family $(\alpha _i)_{i\in I}\subseteq {\rm Con}(A)$, we shall denote by $\displaystyle 
\stackrel{\bullet }{\bigvee _{i\in I}}\alpha _i=\rho _A(\bigvee _{i\in I}\alpha _i)$.

\begin{proposition} $({\rm RCon}(A),\dotvee ,\cap ,\rho _A(\Delta _A),\rho _A(\nabla _A)=\nabla _A)$ is a bounded lattice, orderred by set inclusion. Moreover, it is a complete lattice, in which the arbitrary join is given by the $\displaystyle 
\stackrel{\bullet }{\bigvee }$ defined above.\label{radlat}\end{proposition}

\begin{proof} Of course, $\cap $ is idempotent, commutative and associative, and, clearly, $\dotvee $ is commutative. Now let $\alpha ,\beta ,\gamma \in {\rm Con}(A)$ and $R=\{\rho _A(\alpha ),\rho _A(\beta ),\rho _A(\gamma )\}\subseteq {\rm RCon}(A)$; we shall use Proposition \ref{onrho}, (\ref{onrho2}), (\ref{onrho3}) and (\ref{onrho4}): $\alpha ,\beta ,\gamma \in {\rm Con}(A)$, $\rho _A(\alpha )\dotvee \rho _A(\alpha )=\rho _A(\rho _A(\alpha )\vee \rho _A(\alpha ))=\rho _A(\alpha \vee \alpha )=\rho _A(\alpha )$, so $\dotvee $ is idempotent; $\displaystyle \rho _A(\alpha )\dotvee (\rho _A(\beta )\dotvee \rho _A(\gamma ))=\rho _A(\alpha )\dotvee \rho _A(\rho _A(\beta )\vee \rho _A(\gamma ))=\rho _A(\alpha )\dotvee \rho _A(\beta \vee \gamma )=\rho _A(\rho _A(\alpha )\vee \rho _A(\rho _A(\beta \vee \gamma )))=\rho _A(\rho _A(\alpha )\vee \rho _A(\beta \vee \gamma ))=\rho _A(\alpha \vee (\beta \vee \gamma ))=\rho _A(\alpha \vee \beta \vee \gamma)=\rho _A(\rho _A(\alpha )\vee \rho _A(\beta )\vee \rho _A(\gamma ))=\stackrel{\bullet }{\bigvee _{\theta \in R}}\theta $, thus, by the commutativity of $\dotvee $, we also have $\displaystyle (\rho _A(\alpha )\dotvee \rho _A(\beta ))\dotvee \rho _A(\gamma )=\rho _A(\gamma )\dotvee (\rho _A(\alpha )\dotvee \rho _A(\beta ))=\stackrel{\bullet }{\bigvee _{\theta \in R}}\theta $, hence $\rho _A(\alpha )\dotvee (\rho _A(\beta )\dotvee \rho _A(\gamma ))=(\rho _A(\alpha )\dotvee \rho _A(\beta ))\dotvee \rho _A(\gamma )$, so $\dotvee $ is associative; $\rho _A(\alpha )\dotvee (\rho _A(\alpha )\cap \rho _A(\beta ))=\rho _A(\alpha )\dotvee \rho _A(\alpha \cap \beta )=\rho _A(\rho _A(\alpha )\vee \rho _A(\alpha \cap \beta ))=\rho _A(\alpha \vee (\alpha \cap \beta ))=\rho _A(\alpha )$ and $\rho _A(\alpha )\cap (\rho _A(\alpha )\dotvee \rho _A(\beta ))=\rho _A(\alpha \cap (\rho _A(\alpha )\dotvee \rho _A(\beta )))=\rho _A(\alpha \cap \rho _A(\rho _A(\alpha )\vee \rho _A(\beta )))=\rho _A(\alpha \cap \rho _A(\alpha \vee \beta ))=\rho _A(\rho _A(\alpha \cap \rho _A(\alpha \vee \beta )))=\rho _A(\rho _A(\alpha ))\cap \rho _A(\rho _A(\alpha \vee \beta )))=\rho _A(\alpha )\cap \rho _A(\alpha \vee \beta ))=\rho _A(\alpha \cap (\alpha \vee \beta ))=\rho _A(\alpha )$, so the absorption laws hold. Of course, for all $\theta ,\zeta \in {\rm RCon}(A)$, $\theta \cap \zeta =\theta $ iff $\theta \subseteq \zeta $. Therefore $({\rm RCon}(A),\dotvee ,\cap )$ is a lattice, orderred by set inclusion. From Remark \ref{clara}, (\ref{clara2}) and (\ref{clara0}), we obtain that this lattice has $\rho _A(\Delta _A)$ as first element and $\rho _A(\nabla _A)=\nabla _A$ as last element.

Now let us consider a family $(\alpha _i)_{i\in I}\subseteq {\rm Con}(A)$, $M=\{\rho _A(\alpha _i)\ |\ i\in I\}\subseteq {\rm RCon}(A)$ and let us denote by $\displaystyle \theta = 
\stackrel{\bullet }{\bigvee _{i\in I}}\rho _A(\alpha _i)=\rho _A(\bigvee _{i\in I}\rho _A(\alpha _i))=\rho _A(\bigvee _{i\in I}\alpha _i)$, by Proposition \ref{onrho}, (\ref{onrho4}). Then $\theta \in {\rm RCon}(A)$ and $\rho _A(\alpha _i)\subseteq \theta $ for all $i\in I$. Now, if $\zeta \in {\rm RCon}(A)$ and $\rho _A(\alpha _i)\subseteq \zeta $ for all $i\in I$, then $\displaystyle \bigvee _{i\in I}\rho _A(\alpha _i)\subseteq \zeta $, so, by Remark \ref{clara}, (\ref{clara2}), and Proposition \ref{onrho}, (\ref{onrho4}), $\displaystyle \zeta =\rho _A(\zeta )\supseteq \rho _A(\bigvee _{i\in I}\rho _A(\alpha _i))=\stackrel{\bullet }{\bigvee _{i\in I}}\rho _A(\alpha _i)=\theta $. Therefore $\displaystyle \theta =\sup (M)$ in the bounded lattice ${\rm RCon}(A)$, hence this lattice is complete.\end{proof}

Let us define a binary relation $\equiv _A $ on ${\rm Con}(A)$ by: $\alpha \equiv _A \beta $ iff $\rho _A(\alpha )=\rho _A(\beta )$, for any $\alpha ,\beta \in {\rm Con}(A)$. $\equiv _A \cap ({\cal K}(A))^2$ shall also be denoted by $\equiv _A $.

\begin{remark} Clearly, $\equiv _A $ is an equivalence on ${\rm Con}(A)$, thus also on ${\cal K}(A)$. On ${\rm RCon}(A)$, $\equiv _A $ coincides to the equality, that is to $\Delta _{{\rm RCon}(A)}$, because, for any $\alpha ,\beta \in {\rm Con}(A)$, $\rho _A(\alpha )\equiv _A \rho _A(\beta )$ iff $\rho _A(\rho _A(\alpha ))=\rho _A(\rho _A(\beta ))$ iff $\rho _A(\alpha )=\rho _A(\beta )$. So, trivially, $\equiv _A $ is a congruence of the lattice ${\rm RCon}(A)$.

On ${\rm Con}(A)$, $\equiv _A $ preserves the commutator, $\cap $, $\vee $ and $\dotvee $, even $\bigvee $ and $\stackrel{\bullet }{\bigvee }$ over arbitrary families of congruences, in particular it is a congruence of the lattice ${\rm Con}(A)$. Indeed, if $\alpha ,\alpha ^{\prime},\beta ,\beta ^{\prime}\in {\rm Con}(A)$ such that $\alpha \equiv _A \alpha ^{\prime}$ and $\beta \equiv _A \beta ^{\prime}$, that is $\rho _A(\alpha )=\rho _A(\alpha ^{\prime})$ and $\rho _A(\beta )=\rho _A(\beta ^{\prime})$, then, by Proposition \ref{onrho}, (\ref{onrho2}), $\rho _A([\alpha ,\beta ]_A)=\rho _A(\alpha \cap \beta )=\rho _A(\alpha )\cap \rho _A(\beta )=\rho _A(\alpha ^{\prime})\cap \rho _A(\beta ^{\prime})=\rho _A(\alpha ^{\prime}\cap \beta ^{\prime})=\rho _A([\alpha ^{\prime},\beta ^{\prime}]_A)$, thus $[\alpha ,\beta ]_A\equiv _A [\alpha ^{\prime},\beta ^{\prime}]_A\equiv _A \alpha \cap \beta \equiv _A \alpha ^{\prime}\cap \beta ^{\prime}$. Now, if $(\alpha _i)_{i\in I}\subseteq {\rm Con}(A)$ and $(\alpha _i^{\prime})_{i\in I}\subseteq {\rm Con}(A)$ such that, for all $i\in I$, $\alpha _i\equiv _A \alpha _i^{\prime}$, that is $\rho _A(\alpha _i)=\rho _A(\alpha _i^{\prime})$, then, by Proposition \ref{onrho}, (\ref{onrho3}) and (\ref{onrho4}), $\displaystyle \rho _A(\stackrel{\bullet }{\bigvee _{i\in I}}\alpha _i)=\rho _A(\rho _A(\bigvee _{i\in I}\alpha _i))=\rho _A(\bigvee _{i\in I}\alpha _i)=\rho _A(\bigvee _{i\in I}\rho _A(\alpha _i))=\rho _A(\bigvee _{i\in I}\rho _A(\alpha _i^{\prime}))=\rho _A(\bigvee _{i\in I}\alpha _i^{\prime})=\rho _A(\rho _A(\bigvee _{i\in I}\alpha _i^{\prime}))=\rho _A(\stackrel{\bullet }{\bigvee _{i\in I}}\alpha _i^{\prime})$, hence $\displaystyle \stackrel{\bullet }{\bigvee _{i\in I}}\alpha _i\equiv _A \stackrel{\bullet }{\bigvee _{i\in I}}\alpha _i^{\prime}\equiv _A \bigvee _{i\in I}\alpha _i\equiv _A \bigvee _{i\in I}\alpha _i^{\prime}$.

Moreover, as shown by Proposition \ref{onrho}, (\ref{onrho2}), (\ref{onrho3}) and (\ref{onrho4}), just as in the calculations above, for all $\alpha ,\beta \in {\rm Con}(A)$ and all $(\alpha _i)_{i\in I}\subseteq {\rm Con}(A)$, $[\alpha ,\beta ]_A\equiv _A \alpha \cap \beta $ and $\displaystyle \stackrel{\bullet }{\bigvee _{i\in I}}\alpha _i\equiv _A \bigvee _{i\in I}\alpha _i$.

Note, also, that, for all $\alpha \in {\rm Con}(A)$, $\alpha \equiv _A \rho _A(\alpha )$, by Proposition \ref{onrho}, (\ref{onrho3}).\label{eqrho}\end{remark}

For all $\alpha \in {\rm Con}(A)$, let us denote by $\widehat{\alpha }$ the equivalence class of $\alpha $ with respect to $\equiv _A $, and let ${\cal L}(A)={\cal K}(A)/_{\textstyle \equiv _A }=\{\widehat{\theta }\ |\ \theta \in {\cal K}(A)\}$. Let $\lambda _A:{\rm Con}(A)\rightarrow {\rm Con}(A)/_{\textstyle \equiv _A }$ be the canonical surjection: $\lambda _A(\theta )=\widehat{\theta }$ for all $\theta \in {\rm Con}(A)$; we denote in the same way its restriction to ${\cal K}(A)$, with its co--domain restricted to ${\cal L}(A)$, that is the canonical surjection $\lambda _A:{\cal K}(A)\rightarrow {\cal L}(A)$. Let us define the following operations on ${\rm Con}(A)$, where the second equalities follow from Remark \ref{eqrho}, as does the fact that these operations are well defined:\begin{itemize}
\item for all $\alpha ,\beta \in {\rm Con}(A)$, $\widehat{\alpha }\vee \widehat{\beta }=\widehat{\alpha \vee \beta }=\widehat{\alpha \dotvee \beta }$ and $\widehat{\alpha }\wedge \widehat{\beta }=\widehat{\alpha \cap \beta }=\widehat{[\alpha ,\beta ]_A}$;
\item ${\bf 0}=\widehat{\Delta _A}=\widehat{\rho _A(\Delta _A)}$ and ${\bf 1}=\widehat{\nabla _A}=\widehat{\rho _A(\nabla _A)}$.\end{itemize}

\begin{remark} By Proposition \ref{onrho}, (\ref{onrho1}), if ${\cal C}$ is congruence--modular and semi--degenerate, then, for any $\alpha \in {\rm Con}(A)$, $\widehat{\alpha }={\bf 1}$ iff $\alpha =\nabla _A$.\label{coronrho}\end{remark}

\begin{lemma} $({\rm Con}(A)/_{\textstyle \equiv _A },\vee ,\wedge ,{\bf 0},{\bf 1})$ is a bounded distributive lattice and $\lambda _A:{\rm Con}(A)\rightarrow {\rm Con}(A)/_{\textstyle \equiv _A }$ is a bounded lattice morphism. Moreover, ${\rm Con}(A)/_{\textstyle \equiv _A }$ is a complete lattice, in which $\displaystyle \bigvee _{i\in I}\widehat{\alpha _i}=\widehat{\bigvee _{i\in I}\alpha _i}$ and $\displaystyle \bigwedge _{i\in I}\widehat{\alpha _i}=\widehat{\bigcap _{i\in I}\alpha _i}$ for any family $(\alpha _i)_{i\in I}\subseteq {\rm Con}(A)$, and the meet is completely distributive with respect to the join, thus ${\rm Con}(A)/_{\textstyle \equiv _A }$ is a frame.\label{latcat}\end{lemma}

\begin{proof} By Remark \ref{eqrho}, $\equiv _A $ is a congruence of the bounded lattice ${\rm Con}(A)$, hence $({\rm Con}(A)/_{\textstyle \equiv _A },\vee ,\wedge ,{\bf 0},{\bf 1})$ is a bounded lattice and the canonical surjection $\lambda _A:{\rm Con}(A)\rightarrow {\rm Con}(A)/_{\textstyle \equiv _A }$ is a bounded lattice morphism, in particular it is order--preserving. It is straightforward, from the fact that the lattice ${\rm Con}(A)$ is complete and the surjectivity of the lattice morphism $\lambda _A$, that the lattice ${\rm Con}(A)/_{\textstyle \equiv _A }$ is complete and its joins and meets of arbitrary families of elements have the form in the enunciation. By Proposition \ref{1.3}, for any families $(\alpha _i)_{i\in I}$ and $(\beta _j)_{j\in J}$ of congruences of $A$, $\displaystyle (\bigvee _{i\in I}\widehat{\alpha _i})\wedge (\bigvee _{j\in J}\widehat{\beta _j})=((\bigvee _{i\in I}\alpha _i)\cap (\bigvee _{j\in J}\beta _j))\widehat{\ }=([\bigvee _{i\in I}\alpha _i,\bigvee _{j\in J}\beta _j]_A)\widehat{\ }=(\bigvee _{i\in I}\bigvee _{j\in J}[\alpha _i,\beta _j]_A)\widehat{\ }=\bigvee _{i\in I}\bigvee _{j\in J}\widehat{[\alpha _i,\beta _j]_A}=\bigvee _{i\in I}\bigvee _{j\in J}(\widehat{\alpha _i}\wedge \widehat{\beta _j})$, that is the meet is completely distributive with respect to the join in ${\rm Con}(A)/_{\textstyle \equiv _A }$, thus ${\rm Con}(A)/_{\textstyle \equiv _A }$ is a frame, in particular it is a bounded distributive lattice.\end{proof}

We shall denote by $\leq $ the partial order of the lattice ${\rm Con}(A)/_{\textstyle \equiv _A }$.

\begin{proposition} $({\rm RCon}(A),\dotvee ,\cap ,\rho _A(\Delta _A),\rho _A(\nabla _A)=\nabla _A)$ is a frame, isomorphic to ${\rm Con}(A)/_{\textstyle \equiv _A }$.\end{proposition}

\begin{proof} Let $\varphi :{\rm Con}(A)/_{\textstyle \equiv _A }\rightarrow {\rm RCon}(A)$, for all $\alpha \in {\rm Con}(A)$, $\varphi(\widehat{\alpha })=\rho _A(\alpha )$. If $\alpha ,\beta \in {\rm Con}(A)$, then the following equivalences hold: $\widehat{\alpha }=\widehat{\beta }$ iff $\alpha \equiv _A \beta $ iff $\rho _A(\alpha )=\rho _A(\beta )$ iff $\varphi(\widehat{\alpha })=\varphi(\widehat{\beta })$, hence $\varphi $ is well defined and injective. By Remark \ref{radical}, $\varphi $ is surjective. By Proposition \ref{onrho}, (\ref{onrho2}) and (\ref{onrho4}), for all $\alpha ,\beta \in {\rm Con}(A)$, $\varphi (\widehat{\alpha }\wedge \widehat{\beta})=\varphi (\widehat{\alpha \cap \beta })=\rho _A(\alpha \cap \beta)=\rho _A(\alpha )\cap \rho _A(\beta )$ and $\varphi (\widehat{\alpha }\vee \widehat{\beta})=\varphi (\widehat{\alpha \vee \beta })=\rho _A(\alpha \vee \beta)=\rho _A(\alpha )\vee \rho _A(\beta )$ (actually, Proposition \ref{onrho}, (\ref{onrho4}), and Lemma \ref{latcat} show that $\varphi $ preserves arbitrary joins). Therefore $\varphi $ is a lattice isomorphism, thus an order isomorphism, hence it preserves arbitrary joins and meets. From this and Lemma \ref{latcat} we obtain that ${\rm RCon}(A)$ is a frame and $\varphi $ is a frame isomorphism.\end{proof}

Throughout the rest of this section, we shall assume that ${\cal K}(A)$ is closed with respect to the commutator.

\begin{proposition} ${\cal L}(A)$ is a bounded sublattice of ${\rm Con}(A)/_{\textstyle \equiv _A }$, thus it is a bounded distributive lattice.\label{reticd01}\end{proposition}

\begin{proof} Since $\nabla _A\in {\cal K}(A)$, we have ${\bf 1}=\widehat{\nabla _A}\in {\cal L}(A)$. By Remark \ref{ka}, $\Delta _A\in {\cal K}(A)$, thus ${\bf 0}=\widehat{\Delta _A}\in {\cal L}(A)$. If ${\cal K}(A)$ is closed with respect to the commutator, then, for each $\alpha ,\beta \in {\cal K}(A)$, we have $[\alpha ,\beta ]_A\in {\cal K}(A)$, thus $\widehat{\alpha }\wedge \widehat{\beta }=\widehat{[\alpha ,\beta ]_A}\in {\cal L}(A)$ Again by Remark \ref{ka}, for each $\alpha ,\beta \in {\cal K}(A)$, $\widehat{\alpha }\vee \widehat{\beta }=\widehat{\alpha \vee \beta }\in {\cal L}(A)$. Hence ${\cal L}(A)$ is a bounded sublattice of ${\rm Con}(A)/_{\textstyle \equiv _A }$, which is distributive by Lemma \ref{latcat}, thus ${\cal L}(A)$ is a bounded distributive lattice.\end{proof}

For any $\theta \in {\rm Con}(A)$ and any $I\in {\rm Id}({\cal L}(A))$, we shall denote by:\begin{itemize}
\item $\theta ^*=\{\widehat{\alpha }\ |\ \alpha \in {\cal K}(A),\alpha \subseteq \theta \}=\lambda _A({\cal K}(A)\cap (\theta ])\subseteq {\cal L}(A)$, where $(\theta ]=(\theta ]_{{\rm Con}(A)}\in {\rm PId}({\rm Con}(A))$;
\item $\displaystyle I_*=\bigvee \{\alpha \in {\cal K}(A)\ |\ \widehat{\alpha }\in I\}=\bigvee _{\alpha \in \lambda _A^{-1}(I)}\alpha \in {\rm Con}(A)$; note that $\lambda _A^{-1}(I)$ is non--empty, because $\Delta _A\in {\cal K}(A)$ and $\widehat{\Delta _A}=\widehat{\rho _A(\Delta _A)}={\bf 0}\in I$.\end{itemize}

\begin{lemma} For all $\theta \in {\rm Con}(A)$:\begin{itemize}
\item $\theta ^*\subseteq (\widehat{\theta }\, ]_{{\rm Con}(A)/_{\textstyle \equiv _A }}\cap {\cal L}(A)$ and $\theta ^*\in {\rm Id}({\cal L}(A))$;
\item if $\theta \in {\cal K}(A)$, then $\theta ^*=(\widehat{\theta }\, ]_{{\rm Con}(A)/_{\textstyle \equiv _A }}\cap {\cal L}(A)=(\widehat{\theta }\, ]_{{\cal L}(A)}\in {\rm PId}({\cal L}(A))$.\end{itemize}\label{kstar}\end{lemma}

\begin{proof} Let $\theta \in {\rm Con}(A)$, and, in this proof, let us denote by $\langle \widehat{\theta }\rangle =(\widehat{\theta }\, ]_{{\rm Con}(A)/_{\textstyle \equiv _A }}$ and, in the case when $\theta \in {\cal K}(A)$, by $(\widehat{\theta }\, ]=(\widehat{\theta }\, ]_{{\cal L}(A)}$. $\theta ^*=\{\widehat{\alpha }\ |\ \alpha \in (\theta ]\cap {\cal K}(A)\}$.

For all $\alpha \in (\theta ]\cap {\cal K}(A)$, we have $\widehat{\alpha }\in {\cal L}(A)$ and $\alpha \subseteq \theta $, thus $\widehat{\alpha }\leq \widehat{\theta }$ in ${\rm Con}(A)/_{\textstyle \equiv _A }$, hence $\widehat{\alpha }\in \langle \widehat{\theta }\rangle \cap {\cal L}(A)$, therefore $\theta ^*\subseteq \langle \widehat{\theta }\rangle \cap {\cal L}(A)$. $\Delta _A\in {\cal K}(A)$ and $\Delta _A\subseteq \theta $, thus $\widehat{\Delta _A}\in \theta ^*$, so $\theta ^*$ is non--empty. Since ${\cal K}(A)$ is closed w.r.t. $[\cdot ,\cdot ]_A$, $\alpha \vee \beta ,[\alpha ,\beta ]_A\in {\cal K}(A)$ for any $\alpha ,\beta \in {\cal K}(A)$. Let $x,y\in \theta ^*$, which means that $x=\widehat{\alpha }$ and $y=\widehat{\beta }$ for some $\alpha ,\beta \in {\cal K}(A)\cap (\theta ]$. Then $\alpha \vee \beta \in {\cal K}(A)\cap (\theta ]$, thus $x\vee y=\widehat{\alpha }\vee \widehat{\beta }=\widehat{\alpha \vee \beta }\in \theta ^*$. Now let $x\in \theta ^*$ and $y\in {\cal L}(A)$ such that $x\geq y$, so that $y=x\wedge y$. Then $x=\widehat{\alpha }$ for some $\alpha \in {\cal K}(A)\cap (\theta ]$ and $y=\widehat{\beta }$ for some $\beta \in {\cal K}(A)$. Thus $[\alpha ,\beta ]_A\in {\cal K}(A)$ and $[\alpha ,\beta ]_A\subseteq \alpha \cap \beta \subseteq \alpha \subseteq \theta $, hence $[\alpha ,\beta ]_A\in {\cal K}(A)\cap (\theta ]$, therefore $y=x\wedge y=\widehat{\alpha }\wedge \widehat{\beta }=\widehat{[\alpha ,\beta ]_A}\in \theta ^*$. Hence $\theta ^*\in {\rm Id}({\cal L}(A))$.

Now assume that $\theta \in {\cal K}(A)$, so that $\widehat{\theta }\in {\cal L}(A)$. By the above, $\theta ^*\subseteq \langle \widehat{\theta }\rangle \cap {\cal L}(A)=(\widehat{\theta }\, ]$. Let $x\in (\widehat{\theta }\, ]$, so that there exists an $\alpha \in {\cal K}(A)$ with $\widehat{\alpha }=x\leq \widehat{\theta }$, thus $\widehat{[\alpha ,\theta ]_A}=\widehat{\alpha }\cap \widehat{\theta }=\widehat{\alpha }=x$. But $[\alpha ,\theta ]_A\in {\cal K}(A)\cap (\theta ]$, so $x=\widehat{[\alpha ,\theta ]_A}\in \theta ^*$. Therefore we also have $(\widehat{\theta }\, ]\subseteq \theta ^*$, hence $\theta ^*=(\widehat{\theta }\, ]\in {\rm PId}({\cal L}(A))$.\end{proof}

By the above, we have two functions:\begin{itemize}
\item $\theta \in {\rm Con}(A)\mapsto \theta ^*\in {\rm Id}({\cal L}(A))$;
\item $I\in {\rm Id}({\cal L}(A))\mapsto I_*\in {\rm Con}(A)$.\end{itemize}

\begin{lemma} The two functions above are order--preserving.\label{fcts}\end{lemma}

\begin{proof} For any $\theta ,\zeta \in {\rm Con}(A)$ such that $\theta \subseteq \zeta $, we have $(\theta ]\subseteq (\zeta ]$, hence $\theta ^*\subseteq \zeta ^*$. For any $I,J\in {\rm Id}({\cal L}(A))$ such that $I\subseteq J$, we have $\lambda _A^{-1}(I)\subseteq \lambda _A^{-1}(J)$, thus $I_*\subseteq J_*$.\end{proof}

\begin{lemma} Let $\alpha \in {\cal K}(A)$ and $I\in {\rm Id}({\cal L}(A))$. Then: $\alpha \subseteq I_*$ iff $\widehat{\alpha }\in I$.\label{lema5}\end{lemma}

\begin{proof}``$\Leftarrow $:`` If $\widehat{\alpha }\in I$, then $\alpha \in \lambda _A^{-1}(I)$, thus $\alpha \subseteq I_*$.

\noindent ``$\Rightarrow $:`` If $\displaystyle \alpha \subseteq I_*=\bigvee \{\beta \in {\cal K}(A)\ |\ \widehat{\beta }\in I\}$, then, since $\alpha \in {\cal K}(A)$, it follows that there exist an $n\in \N ^*$ and $\beta _1,\ldots ,\beta _n\in {\cal K}(A)$ such that $\widehat{\beta _1},\ldots ,\widehat{\beta _n}\in I$ and $\displaystyle \alpha \subseteq \bigvee _{i=1}^n\beta _i$, hence $\displaystyle \widehat{\alpha }\subseteq \widehat{\bigvee _{i=1}^n\beta _i}=\bigvee _{i=1}^n\widehat{\beta _i}\in I$, thus $\widehat{\alpha }\in I$.\end{proof}

\begin{lemma}\begin{enumerate}
\item\label{lema6(1)} For any $\theta \in {\rm Con}(A)$, $\theta \subseteq (\theta ^*)_*$.
\item\label{lema6(2)} For any $I\in {\rm Id}({\cal L}(A))$, $I=(I_*)^*$.\end{enumerate}\label{lema6}\end{lemma}

\begin{proof} (\ref{lema6(1)}) Let $\theta \in {\rm Con}(A)$. For any $(a,b)\in \theta $, $Cg_A(a,b)\in {\rm PCon}(A)\subseteq {\cal K}(A)$ and $Cg_A(a,b)\subseteq \theta $, thus $Cg_A(a,b)\in {\cal K}(A)\cap (\theta ]$, hence $\widehat{Cg_A(a,b)}\in \theta ^*$, therefore $Cg_A(a,b)\subseteq (\theta ^*)_*$ by Lemmas \ref{kstar} and \ref{lema5}, so $(a,b)\in (\theta ^*)_*$. Hence $\theta \subseteq (\theta ^*)_*$.

\noindent (\ref{lema6(2)}) For any $x\in {\cal L}(A)$, by Lemma \ref{lema5}, the following equivalences hold: $x\in (I_*)^*$ iff there exists an $\alpha \in {\cal K}(A)$ such that $\alpha \subseteq I_*$ and $x=\widehat{\alpha }$ iff there exists an $\alpha \in {\cal K}(A)$ such that $\widehat{\alpha }\in I$ and $x=\widehat{\alpha }$ iff $x\in I$. Therefore $(I_*)^*=I$.\end{proof}

\begin{proposition}\begin{enumerate}
\item\label{prop7(1)} The map $I\in {\rm Id}({\cal L}(A))\mapsto I_*\in {\rm Con}(A)$ is injective.

\item\label{prop7(2)} The map $\theta \in {\rm Con}(A)\mapsto \theta ^*\in {\rm Id}({\cal L}(A))$ is surjective.\end{enumerate}\label{prop7}\end{proposition}

\begin{proof} (\ref{prop7(1)}) Let $I,J\in {\rm Id}({\cal L}(A))$ such that $I_*=J_*$. Then $(I_*)^*=(J_*)^*$, so $I=J$ by Lemma \ref{lema6}, (\ref{lema6(2)}).

\noindent (\ref{prop7(2)}) Let $I\in {\rm Id}({\cal L}(A))$, and denote $\theta =I_*\in {\rm Con}(A)$. Then $\theta ^*=(I_*)^*=I$ by Lemma \ref{lema6}, (\ref{lema6(2)}).\end{proof}

\begin{lemma} For any $\phi \in {\rm Spec}(A)$, $\phi =(\phi ^*)_*$.\label{lema8}\end{lemma}

\begin{proof} Let $\phi \in {\rm Spec}(A)$. Then $\phi \subseteq (\phi ^*)_*$ by Lemma \ref{lema6}, (\ref{lema6(1)}).

Now let $\beta \in {\cal K}(A)$ such that $\widehat{\beta }\in \phi ^*=\{\widehat{\alpha }\ |\ \alpha \in {\cal K}(A),\alpha \subseteq \phi \}$, which means that $\widehat{\beta }=\widehat{\alpha }$ for some $\alpha \in {\cal K}(A)$ with $\alpha \subseteq \phi $. Since $\widehat{\beta }=\widehat{\alpha }$, we have $\rho _A(\beta )=\rho _A(\alpha )$, while $\alpha \subseteq \phi $ gives us $\rho _A(\alpha )\subseteq \rho _A(\phi )=\phi $, where the last equality follows from the fact that $\phi \in {\rm Spec}(A)$. Hence $\beta \subseteq \rho _A(\beta )\subseteq \phi $. Therefore $\displaystyle (\phi ^*)_*=\bigvee \{\gamma \in {\cal K}(A)\ |\ \widehat{\gamma }\in \phi ^*\}\subseteq \phi $.

Hence $\phi =(\phi ^*)_*$.\end{proof}

\begin{lemma} For any $\phi \in {\rm Spec}(A)$, we have $\phi ^*\in {\rm Spec}_{\rm Id}({\cal L}(A))$.\label{lema9}\end{lemma}

\begin{proof} Let $\phi \in {\rm Spec}(A)$. Then $\phi ^*\in {\rm Id}({\cal L}(A))={\rm Id}({\cal K}(A)/_{\textstyle \equiv _A })$. Let $\alpha ,\beta \in {\cal K}(A)$ such that $\widehat{[\alpha ,\beta ]_A}=\widehat{\alpha }\wedge \widehat{\beta }\in \phi ^*=\{\widehat{\gamma }\ |\ \gamma \in {\cal K}(A),\gamma \subseteq \phi \}$. Then there exists a $\gamma \in {\cal K}(A)$ such that $\gamma \subseteq \phi $ and $\widehat{\gamma }=\widehat{[\alpha ,\beta ]_A}$, thus $\rho _A(\gamma )=\rho _A([\alpha ,\beta ]_A)$ and $\rho _A(\gamma )\subseteq \rho _A(\phi )=\phi $ since $\phi \in {\rm Spec}(A)$. Hence $[\alpha ,\beta ]_A\subseteq \rho _A([\alpha ,\beta ]_A)\subseteq \phi $, hence $\alpha \subseteq \phi $ or $\beta \subseteq \phi $ since $\phi \in {\rm Spec}(A)$. But this means that $\widehat{\alpha }\in \phi ^*$ or $\widehat{\beta }\in \phi ^*$. Therefore $\phi ^*\in {\rm Spec}_{\rm Id}({\cal L}(A))$.\end{proof}

\begin{lemma} For any $P\in {\rm Spec}_{\rm Id}({\cal L}(A))$, we have $P_*\in {\rm Spec}(A)$.\label{lema10}\end{lemma}

\begin{proof} Let $P\in {\rm Spec}_{\rm Id}({\cal L}(A))$. Then $P_*\in {\rm Con}(A)$. Let $\alpha ,\beta \in {\rm PCon}(A)$ such that $[\alpha ,\beta ]_A\subseteq P_*$. Then $\alpha ,\beta \in {\cal K}(A)$, so that $[\alpha ,\beta ]_A\in {\cal K}(A)$, and $\displaystyle [\alpha ,\beta ]_A\subseteq \bigvee \{\gamma \in {\cal K}(A)\ |\ \widehat{\gamma }\in P\}$, hence there exist an $n\in \N ^*$ and $\gamma _1,\ldots ,\gamma _n\in {\cal K}(A)$ such that $\widehat{\gamma _1},\ldots ,\widehat{\gamma _n}\in P$ and $\displaystyle [\alpha ,\beta ]_A\subseteq \bigvee _{i=1}^n\gamma _i$. But then $\displaystyle \widehat{\bigvee _{i=1}^n\gamma _i}=\bigvee _{i=1}^n\widehat{\gamma _i}\in P$, hence $\widehat{\alpha }\wedge \widehat{\beta }=\widehat{[\alpha ,\beta ]_A}\in P$, thus $\widehat{\alpha }\in P$ or $\widehat{\beta }\in P$ since $P\in {\rm Spec}_{\rm Id}({\cal L}(A))$. By Lemma \ref{lema5}, it follows that $\alpha \subseteq P_*$ or $\beta \subseteq P_*$. Therefore $P_*\in {\rm Spec}(A)$.\end{proof}

By Lemmas \ref{lema9} and \ref{lema10}, we have these restrictions of the functions defined above:\begin{itemize}
\item $u:{\rm Spec}(A)\rightarrow {\rm Spec}_{\rm Id}({\cal L}(A))$, for all $\phi \in {\rm Spec}(A)$, $u(\phi )=\phi ^*$;
\item $v:{\rm Spec}_{\rm Id}({\cal L}(A))\rightarrow {\rm Spec}(A)$, for all $P\in {\rm Spec}_{\rm Id}({\cal L}(A))$, $v(P)=P_*$.\end{itemize}

\begin{proposition} $u$ and $v$ are homeomorphisms, inverses of each other, between the prime spectrum of $A$ and the prime spectrum of ideals of ${\cal L}(A)$, endowed with the Stone topologies.\label{prop11}\end{proposition}

\begin{proof} By Lemma \ref{lema6}, (\ref{lema6(2)}), for all $P\in {\rm Spec}_{\rm Id}({\cal L}(A))$, we have $u(v(P))=P$. By Lemma \ref{lema8}, for all $\phi \in {\rm Spec}(A)$, we have $v(u(\phi ))=\phi $. Thus $u$ and $v$ are bijections and they are inverses of each other.

Let $\theta \in {\rm Con}(A)$ and $\phi \in V_A(\theta )$, that is $\phi \in {\rm Spec}(A)$ and $\theta \subseteq \phi $. Then, by Lemmas \ref{lema10} and \ref{fcts}, $\phi ^*\in {\rm Spec}_{\rm Id}({\cal L}(A))$ and $\theta ^*\subseteq \phi ^*$, so $\phi ^*\in V_{{\rm Id},{\cal L}(A)}(\theta ^*)$, and we have $u(\phi )=\phi ^*$. Hence $u(V_A(\theta ))\subseteq V_{{\rm Id},{\cal L}(A)}(\theta ^*)$. Now let $P\in V_{{\rm Id},{\cal L}(A)}(\theta ^*)$, that is $P\in {\rm Spec}_{\rm Id}({\cal L}(A))$ and $\theta ^*\subseteq P$. Then, by Lemma \ref{lema6}, (\ref{lema6(1)}), and Lemmas \ref{fcts} and \ref{lema10}, $\theta \subseteq (\theta ^*)_*\subseteq P_*\in {\rm Spec}(A)$, thus $P_*\in V_A(\theta )$, and we have $u(P_*)=u(v(P))=P$. Hence $V_{{\rm Id},{\cal L}(A)}(\theta ^*)\subseteq u(V_A(\theta ))$. Therefore $u(V_A(\theta ))=V_{{\rm Id},{\cal L}(A)}(\theta ^*)$, thus $u$ is closed, hence $u$ is open, so $v$ is continuous. 

Now let $I\in {\rm Id}({\cal L}(A))$. Then, according to Proposition \ref{prop7}, (\ref{prop7(2)}), $I=\theta ^*$ for some $\theta \in {\rm Con}(A)$. By the above, $u(V_A(\theta ))=V_{{\rm Id},{\cal L}(A)}(\theta ^*)=V_{{\rm Id},{\cal L}(A)}(I)$, hence $v(V_{{\rm Id},{\cal L}(A)}(I))=v(u(V_A(\theta )))=V_A(\theta )$, therefore $v$ is closed, hence $v$ is open, thus $u$ is continuous.

Hence $u$ and $v$ are homeomorphisms.\end{proof}

\begin{corollary}[existence of the reticulation] ${\cal L}(A)$ is a reticulation for the algebra $A$.\end{corollary}

\begin{proposition}{\rm \cite{bal},\cite{gratzer}} If $L$ and $M$ are bounded distributive lattices whose prime spectra of ideals, endowed with the Stone topologies, are homeomorphic, then $L$ and $M$ are isomorphic.\end{proposition}

\begin{corollary}[uniqueness of the reticulation] The reticulation of $A$ is unique up to a lattice isomorphism.\end{corollary}

\begin{corollary} If ${\cal C}$ is congruence--modular and semi--degenerate, then $u$ and $v$ induce homeomorphisms, inverses of each other, between the maximal spectrum of $A$ and the maximal spectrum of ideals of ${\cal L}(A)$, endowed with the Stone topologies.\label{homeomaxspec}\end{corollary}

\begin{proof} By Lemma \ref{folclor}, Proposition \ref{prop11} and the fact that, as Lemma \ref{fcts} ensures us, $u$ and $v$ are order--preserving, and hence they are order isomorphisms between the posets $({\rm Spec}(A),\subseteq )$ and $({\rm Spec}_{\rm Id}({\cal L}(A)),\subseteq )$.\end{proof}

\begin{proposition}{\rm \cite[Proposition $4.1$]{agl}} For any $\theta \in {\rm Con}(A)$, $\rho _A(\theta )=\{(a,b)\in A^2\ |\ (\exists \, n\in \N ^*)\, ([Cg_A(a,b),$\linebreak $Cg_A(a,b)]_A^n\subseteq \theta )\}$, so $\rho _A(\Delta _A)=\{(a,b)\in A^2\ |\ (\exists \, n\in \N ^*)\, ([Cg_A(a,b),Cg_A(a,b)]_A^n=\Delta _A)\}$.\label{rho}\end{proposition}

\begin{proposition} For any $\theta \in {\rm Con}(A)$, $(\theta ^*)_*=\rho _A(\theta )$.\label{prop12}\end{proposition}

\begin{proof} For every $\beta \in {\cal K}(A)$ such that $\widehat{\beta }\in \theta ^*=\{\widehat{\gamma }\ |\ \gamma \in {\cal K}(A),\gamma \subseteq \theta \}$, there exists an $\alpha \in {\cal K}(A)$ such that $\alpha \subseteq \theta $ and $\widehat{\alpha }=\widehat{\beta }$, thus $\beta \subseteq \rho _A(\beta )=\rho _A(\alpha )\subseteq \rho _A(\theta )$. Therefore $\displaystyle (\theta ^*)_*=\bigvee \{\gamma \in {\cal K}(A)\ |\ \widehat{\gamma }\in \theta ^*\}\subseteq \rho _A(\theta )$. Now let $(a,b)\in \rho _A(\theta )$, so that, according to Proposition \ref{rho}, Lemma \ref{lema6}, (\ref{lema6(1)}), and Lemma \ref{lema5}, for some $n\in \N ^*$, $[Cg_A(a,b),Cg_A(a,b)]_A^n\subseteq \theta \subseteq (\theta ^*)_*$, hence $([Cg_A(a,b),Cg_A(a,b)]_A^n)\widehat{\ }\in \theta ^*$. But $\rho _A([Cg_A(a,b),Cg_A(a,b)]_A^n)=\rho _A(Cg_A(a,b))$, thus $\widehat{Cg_A(a,b)}=([Cg_A(a,b),Cg_A(a,b)]_A^n)\widehat{\ }\in \theta ^*$, hence $(a,b)\in Cg_A(a,b)\subseteq (\theta ^*)_*$ by Lemma \ref{lema5}. Therefore $\rho _A(\theta )\subseteq (\theta ^*)_*$. Hence $(\theta ^*)_*=\rho _A(\theta )$.\end{proof}

\begin{corollary}\begin{enumerate}
\item\label{cor13(1)} For all $\theta \in {\rm Con}(A)$, $\rho _A(\theta )^*=\theta ^*$.
\item\label{cor13(2)} For all $I\in {\rm Id}({\cal L}(A))$, $\rho _A(I_*)=I_*$.\end{enumerate}\label{cor13}\end{corollary}

\begin{proof} (\ref{cor13(1)}) By Lemma \ref{lema6}, (\ref{lema6(2)}), and Proposition \ref{prop12}, $\theta ^*=((\theta ^*)_*)^*=\rho _A(\theta )^*$.

\noindent (\ref{cor13(2)}) By Proposition \ref{prop12} and Lemma \ref{lema6}, (\ref{lema6(2)}), we have $\rho _A(I_*)=((I_*)^*)_*=I_*$.\end{proof}

\begin{corollary} The maps:\begin{itemize}
\item $\theta \in {\rm RCon}(A)\mapsto \theta ^*\in {\rm Id}({\cal L}(A))$,

\item $I\in {\rm Id}({\cal L}(A))\mapsto I_*\in {\rm RCon}(A)$\end{itemize}

\noindent are frame isomorphisms and inverses of each other.\label{cor14}\end{corollary}

\begin{proof} By Corollary \ref{cor13}, (\ref{cor13(2)}), for all $I\in {\rm Id}({\cal L}(A))$, we have $I_*\in {\rm RCon}(A)$, hence the second map above is well defined. By Lemma \ref{lema6}, (\ref{lema6(2)}), for all $I\in {\rm Id}({\cal L}(A))$, $(I_*)^*=I$. By Proposition \ref{prop12}, for all $\theta \in {\rm RCon}(A)$, $\theta =\rho _A(\theta )=(\theta ^*)_*$. Hence these functions are inverses of each other, thus they are bijections. By Lemma \ref{fcts}, these maps are order--preserving, thus they are order isomorphisms, hence they preserve arbitrary joins and meets, therefore they are frame isomorphisms.\end{proof}

\section{Some Examples, Particular Cases and Preservation of Finite Direct Products}
\label{examples}

Throughout this section, we shall assume that $[\cdot ,\cdot ]_A$ is commutative and distributive w.r.t. arbitrary joins and $\nabla _A\in {\cal K}(A)$. These hypotheses are sufficient for the following results we cite from other works to hold. We shall denote by ${\cal HSP}(A)$ the variety generated by $A$. In the following examples, we determine the prime spectra by using \cite[Proposition 1.2]{agl}, which says that, for each proper congruence $\phi $ of $A$: $\phi $ is prime iff $\phi $ is meet--irreducible and semiprime. So, if we know that ${\cal HSP}(A)$ is congruence--modular, then we only have to calculate $[\alpha ,\alpha ]_A$ for every $\alpha \in {\rm Con}(A)$. The complete tables of the commutators for the following algebras show that their commutators are commutative and distributive w.r.t. the join. Of course, since each of the algebras $M$ from the following examples is finite, we have $\nabla _M\in {\cal K}(M)$. We have used the method in \cite{meo} to calculate the commutators, excepting those in groups, where we have used the commutators on normal subgroups; recall that the variety of groups is congruence--modular \cite{mcks}. Following \cite{agl}, we say that $A$ is: {\em Abelian} iff $[\nabla _A,\nabla _A]_A=\Delta _A$; {\em solvable} iff $[\nabla _A,\nabla _A]_A^n=\Delta _A$ for some $n\in \N ^*$; {\em nilpotent} iff $(\nabla _A,\nabla _A]_A^n=\Delta _A$ for some $n\in \N ^*$. For any $n\in \N ^*$, we shall denote by ${\cal L}_n$ the $n$--element chain. By $\oplus $ we shall denote the ordinal sum of bounded lattices.

\begin{remark} If ${\rm Spec}(A)=\emptyset $, then $\rho _A(\alpha )=\nabla _A$ for all $\alpha \in {\rm Con}(A)$, hence $\equiv _A=\nabla _{{\cal K}(A)}$, thus ${\cal L}(A)={\cal K}(A)/\nabla _{{\cal K}(A)}\cong {\cal L}_1$. If ${\rm Spec}(A)=\{\phi \}$ for some $\phi \in {\rm Con}(A)\setminus \{\nabla _A\}$, then: $\rho _A(\theta )=\phi =\rho _A(\Delta _A)$ for all $\theta \in (\phi ]$, and $\rho _A(\theta )=\nabla _A=\rho _A(\nabla _A)$ for all $\theta \in {\rm Con}(A)\setminus (\phi ]$, therefore, since $\Delta _A,\nabla _A\in {\cal K}(A)$, ${\cal L}(A)={\cal K}(A)/\equiv _A\cong {\cal L}_2$.

Obviously, if $A$ is Abelian, then $A$ is nilpotent and solvable and ${\rm Spec}(A)=\emptyset $. Moreover, by \cite[Proposition $1.3$]{agl}, if $A$ is solvable or nilpotent, then ${\rm Spec}(A)=\emptyset $. Thus, if $A$ is solvable or nilpotent, in particular if $A$ is Abelian, then ${\cal L}(A)\cong {\cal L}_1$. For instance, according to \cite{mcks}, any Abelian group is an Abelian algebra, hence its reticulation is trivial.

If $A$ is simple, that is ${\rm Con}(A)=\{\Delta _A,\nabla _A\}\subseteq {\cal K}(A)\subseteq {\rm Con}(A)$, so that ${\cal K}(A)={\rm Con}(A)=\{\Delta _A,\nabla _A\}$, thus ${\cal L}(A)=\{{\bf 0},{\bf 1}\}$, so we are situated in one of the following two cases: either $A$ is Abelian, so that ${\cal L}(A)\cong {\cal L}_1$, or the commutator of $A$ equals the intersection, so that ${\rm Spec}(A)=\{\Delta _A\}$ and thus ${\cal L}(A)\cong {\cal L}_2$.\label{retictriv}\end{remark}

\begin{proposition} If the commutator of $A$ equals the intersection, in particular if ${\cal C}$ is congruence--distributive, then ${\cal K}(A)$ is a bounded sublattice of the bounded distributive lattice ${\rm Con}(A)$ and $\lambda _A:{\cal K}(A)\rightarrow {\cal L}(A)$ is a lattice isomorphism, thus we may take ${\cal L}(A)={\cal K}(A)$.\label{reticdistrib}\end{proposition}

\begin{proof} Assume that $[\cdot ,\cdot ]_A=\cap $. $\Delta _A\in {\rm PCon}(A)\subseteq {\cal K}(A)$. By Remark \ref{ka}, ${\cal K}(A)$ is closed w.r.t. the join, and we are under the assumptions that $\nabla _A\in {\cal K}(A)$ and ${\cal K}(A)$ is closed w.r.t. the commutator, so w.r.t. the intersection. Hence ${\cal K}(A)$ is a bounded sublattice of ${\rm Con}(A)$. By Proposition \ref{radcongrdistrib}, $\equiv _A=\Delta _{{\cal K}(A)}$, thus ${\cal L}(A)={\cal K}(A)/\Delta _{{\cal K}(A)}\cong{\cal K}(A)$ and the canonical surjection $\lambda _A:{\cal K}(A)\rightarrow {\cal L}(A)$ is a lattice isomorphism.\end{proof}

\begin{remark} If ${\rm Con}(A)={\cal K}(A)$, in particular if $A$ is finite, then ${\cal L}(A)={\rm Con}(A)/\equiv _A$, so, if, furthermore, the commutator of $A$ equals the intersection, in particular if ${\cal C}$ is congruence--distributive, then ${\cal L}(A)\cong {\rm Con}(A)$ by Proposition \ref{reticdistrib}, thus we may take ${\cal L}(A)={\rm Con}(A)$.

As a fact that may be interesting by its symmetry, if $A$ is finite and its commutator equals the intersection, so that ${\rm Con}(A)$ is a finite distributive lattice, then ${\cal L}({\rm Con}(A))={\rm Con}({\rm Con}(A))={\rm Con}({\cal L}(A))$. It might also be interesting to find weaker conditions on $A$ under which ${\cal L}({\rm Con}(A))\cong {\rm Con}({\cal L}(A))$.\label{finreticdistrib}\end{remark}

\begin{remark} By Proposition \ref{reticdistrib}, if $A$ is a residuated lattice, then ${\cal L}(A)={\cal K}(A)$. If we denote by ${\rm Filt}(A)$ the set of the filters of $A$ and by ${\rm PFilt}(A)$ the set of the principal filters of $A$, then, since ${\rm Con}(A)\cong {\rm Filt}(A)$ and the finitely generated filters of $A$ are principal filters \cite{gal}, \cite{jits}, it follows that ${\cal L}(A)={\cal K}(A)\cong {\rm PFilt}(A)$, which is the dual of the reticulation of a residuated lattice obtained in \cite{eu1}, \cite{eu}, \cite{eu2}, where the reticulation has the prime spectrum of filters homeomorphic to the prime spectrum of filters, thus to that of congruences of $A$ by the above, so this duality to the construction of ${\cal L}(A)$ from Section \ref{reticulatia} was to be expected.\end{remark}

\begin{remark} If $A$ is a commutative unitary ring and ${\rm Id}(A)$ is its lattice of ideals, then it is well known that ${\rm Id}(A)\cong {\rm Con}(A)$. If, for all $I\in {\rm Id}(A)$, we denote by $\sqrt{I}$ the intersection of the prime filters of $A$ which include $I$, then \cite[Lemma, p. 1861]{bell} shows that, for any $J\in {\rm Id}(A)$, there exists a finitely generated ideal $K$ of $A$ such that $\sqrt{J}=\sqrt{K}$. From this, it immediately follows that the lattice ${\cal L}(A)$ is isomorphic to the reticulation of $A$ constructed in \cite{bell}.\end{remark}

\begin{remark} Let $n,k\in \N ^*$ and assume that ${\cal C}$ is congruence--modular, $S$ is a subalgebra of $A$, $\alpha ,\beta \in {\rm Con}(A)$, $M_1,\ldots ,M_n$ are algebras from ${\cal C}$, $\displaystyle M=\prod _{i=1}^nM_i$ and, for all $i\in \overline{1,n}$, $\alpha _i,\beta _i\in {\rm Con}(M_i)$.

From Lemma \ref{1.8}, it is immediate that $[\alpha \cap S^2,\beta \cap S^2]_S^k\subseteq [\alpha ,\beta ]_A^k\cap S^2$ and $(\alpha \cap S^2,\beta \cap S^2]_S^k\subseteq (\alpha ,\beta ]_A^k\cap S^2$. Hence, if $A$ is Abelian or solvable or nilpotent, then $S$ is Abelian or solvable or nilpotent, respectively.

From Proposition \ref{comutprod}, it is immediate that $\displaystyle [\prod _{i=1}^n\alpha _i,\prod _{i=1}^n\beta _i]_M^k=\prod _{i=1}^n[\alpha _i,\beta _i]_{M_i}^k$ and $\displaystyle (\prod _{i=1}^n\alpha _i,\prod _{i=1}^n\beta _i]_M^k=\prod _{i=1}^n(\alpha _i,\beta _i]_{M_i}^k$. From this, it is easy to prove that: $M$ is Abelian or solvable or nilpotent iff $M_1,\ldots ,M_n$ are Abelian or solvable or nilpotent, respectively.\label{comutsup}\end{remark}

\begin{example} For any group $(G,\cdot )$, any $x\in G$ and any normal subgroup $H$ of $G$, let us denote by $\left<x\right>$ the subgroup of $G$ generated by $x$ and by $\equiv _H$ the congruence of $G$ associated to $H$: $\equiv _H=\{(y,z)\in G^2\ |\ yz^{-1}\in H\}$.

As shown by the following commutators calculations, the cuaternions group, $C_8=\{1,-1,i,-i,j,$\linebreak $-j,k,-k\}$, is a solvable algebra which is not Abelian, while the group $S_3=\{1,t,u,v,c,d\}$ of the permutations of the set $\overline{1,3}$, where $1=id_{\overline{1,3}}$, $t=(1\,2)$, $u=(1\,3)$, $v=(2\,3)$, $c=(1\,2\,3)$ and $d=c\circ c$, has ${\rm Spec}(S_3)=\emptyset $, without being solvable or nilpotent. The following are the subgroups of $C_8$, respectively $S_3$, all of which are normal, and the proper ones are cyclic, thus Abelian: $\left<1\right>$, $\left<-1\right>$, $\left<i\right>$, $\left<j\right>$, $\left<k\right>$ and $C_8$, respectively $\left<1\right>$, $\left<t\right>$, $\left<u\right>$, $\left<v\right>$, $\left<c\right>$ and $S_3$, so $C_8$ and $S_3$ have the following congruence lattices and commutators, which suffice to conclude that ${\rm Spec}(C_8)={\rm Spec}(S_3)=\emptyset $, since we are in a congruence--modular variety, and thus ${\cal L}(C_8)\cong {\cal L}(S_3)\cong {\cal L}_1$, by Remark \ref{retictriv}:\vspace*{-20pt}

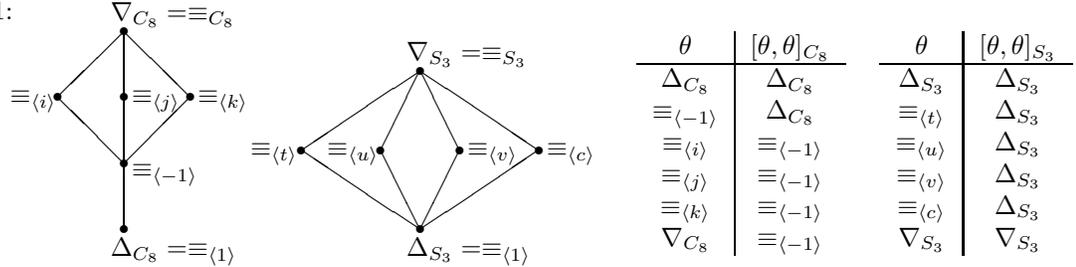
\begin{center}
\begin{tabular}{cccc}
\begin{picture}(100,100)(0,0)
\put(45,0){$\Delta _{C_8}=\equiv _{\left<1\right>}$}
\put(45,89){$\nabla _{C_8}=\equiv _{C_8}$}
\put(50,10){\circle*{3}}
\put(50,10){\line(0,1){75}}
\put(50,35){\circle*{3}}

\put(25,60){\circle*{3}}
\put(50,60){\circle*{3}}
\put(75,60){\circle*{3}}
\put(50,85){\circle*{3}}
\put(50,35){\line(-1,1){25}}
\put(50,35){\line(1,1){25}}
\put(50,85){\line(-1,-1){25}}
\put(50,85){\line(1,-1){25}}
\put(7,58){$\equiv _{\left<i\right>}$}
\put(53,58){$\equiv _{\left<j\right>}$}
\put(78,58){$\equiv _{\left<k\right>}$}
\put(53,30){$\equiv _{\left<-1\right>}$}
\end{picture}
& 
\begin{picture}(100,100)(0,0)
\put(45,0){$\Delta _{S_3}=\equiv _{\left<1\right>}$}
\put(45,74){$\nabla _{S_3}=\equiv _{S_3}$}
\put(50,10){\circle*{3}}
\put(50,70){\circle*{3}}

\put(35,40){\circle*{3}}
\put(65,40){\circle*{3}}
\put(5,40){\circle*{3}}
\put(95,40){\circle*{3}}
\put(50,10){\line(-3,2){45}}
\put(50,10){\line(3,2){45}}
\put(50,10){\line(-1,2){15}}
\put(50,10){\line(1,2){15}}
\put(50,70){\line(-1,-2){15}}
\put(50,70){\line(1,-2){15}}
\put(50,70){\line(-3,-2){45}}
\put(50,70){\line(3,-2){45}}
\put(-14,38){$\equiv _{\left<t\right>}$}
\put(15,38){$\equiv _{\left<u\right>}$}
\put(68,38){$\equiv _{\left<v\right>}$}
\put(98,38){$\equiv _{\left<c\right>}$}
\end{picture}
& 
\begin{picture}(80,100)(0,0)
\put(20,40){\begin{tabular}{c|c}
$\theta $ & $[\theta ,\theta ]_{C_8}$\\ \hline 
$\Delta _{C_8}$ & $\Delta _{C_8}$\\ 
$\equiv _{\left<-1\right>}$ & $\Delta _{C_8}$\\ 
$\equiv _{\left<i\right>}$ & $\equiv _{\left<-1\right>}$\\ 
$\equiv _{\left<j\right>}$ & $\equiv _{\left<-1\right>}$\\ 
$\equiv _{\left<k\right>}$ & $\equiv _{\left<-1\right>}$\\ 
$\nabla _{C_8}$ & $\equiv _{\left<-1\right>}$\end{tabular}}
\end{picture}
& 
\begin{picture}(80,100)(0,0)
\put(20,40){\begin{tabular}{c|c}
$\theta $ & $[\theta ,\theta ]_{S_3}$\\ \hline 
$\Delta _{S_3}$ & $\Delta _{S_3}$\\ 
$\equiv _{\left<t\right>}$ & $\Delta _{S_3}$\\ 
$\equiv _{\left<u\right>}$ & $\Delta _{S_3}$\\ 
$\equiv _{\left<v\right>}$ & $\Delta _{S_3}$\\ 
$\equiv _{\left<c\right>}$ & $\Delta _{S_3}$\\ 
$\nabla _{S_3}$ & $\nabla _{S_3}$\end{tabular}}
\end{picture}
\end{tabular}
\end{center}\vspace*{-5pt}

Notice, also, that $C_8$ is solvable, as we have announced, thus, according to Remark \ref{comutsup}, so is any finite direct product whose factors are subgroups of $C_8$, which, of course, is Abelian if all those subgroups are proper.\label{reticgr}\end{example}

\begin{example} This is the algebra from \cite[Example 6.3]{urs2} and \cite[Example 4.2]{urs3}: $U=(\{0,a,b,c,d\},+)$, with $+$ defined by the following table, which has the congruence lattice represented below, where $U/\alpha =\{\{0,a\},\{b,c,d\}\}$, $U/\beta =\{\{0,b\},\{a,c,d\}\}$, $U/\gamma =\{\{0,c,d\},\{a,b\}\}$ and $U/\delta =\{\{0\},\{a\},\{b\},\{c,d\}\}$:

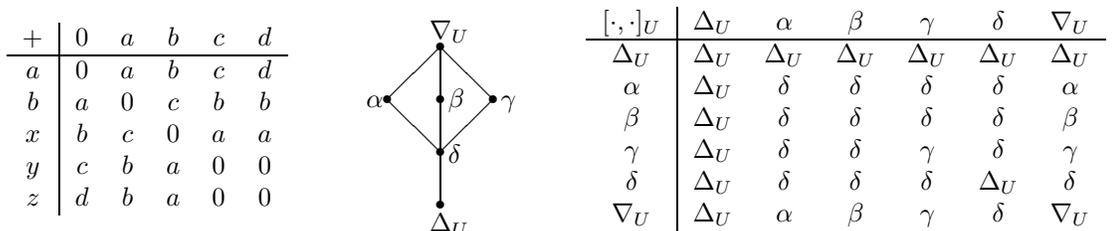
\begin{center}\begin{tabular}{ccc}
\begin{picture}(100,70)(0,0)
\put(0,30){
\begin{tabular}{c|ccccc}
$+$ & $0$ & $a$ & $b$ & $c$ & $d$\\ \hline  
$a$ & $0$ & $a$ & $b$ & $c$ & $d$\\ 
$b$ & $a$ & $0$ & $c$ & $b$ & $b$\\ 
$x$ & $b$ & $c$ & $0$ & $a$ & $a$\\ 
$y$ & $c$ & $b$ & $a$ & $0$ & $0$\\ 
$z$ & $d$ & $b$ & $a$ & $0$ & $0$\end{tabular}}\end{picture}
&\hspace*{12pt}
\begin{picture}(80,70)(0,0)
\put(36,-10){$\Delta _U$}
\put(36,63){$\nabla _U$}
\put(40,0){\circle*{3}}

\put(40,20){\circle*{3}}
\put(20,40){\circle*{3}}
\put(40,40){\circle*{3}}
\put(60,40){\circle*{3}}
\put(40,60){\circle*{3}}
\put(40,20){\line(-1,1){20}}
\put(40,20){\line(1,1){20}}
\put(40,0){\line(0,1){60}}
\put(40,60){\line(-1,-1){20}}
\put(40,60){\line(1,-1){20}}
\put(12,37){$\alpha $}
\put(43,37){$\beta $}
\put(63,37){$\gamma $}
\put(43,16){$\delta $}
\end{picture}
&
\begin{picture}(200,70)(0,0)
\put(0,30){
\begin{tabular}{c|cccccc}
$[\cdot ,\cdot ]_U$ & $\Delta _U$ & $\alpha $ & $\beta $ & $\gamma $ & $\delta $ & $\nabla _U$\\ \hline 
$\Delta _U$ & $\Delta _U$ & $\Delta _U$ & $\Delta _U$ & $\Delta _U$ & $\Delta _U$ & $\Delta _U$\\ 
$\alpha $ & $\Delta _U$ & $\delta $ & $\delta $ & $\delta $ & $\delta $ & $\alpha $\\ 
$\beta $ & $\Delta _U$ & $\delta $ & $\delta $ & $\delta $ & $\delta $ & $\beta $\\ 
$\gamma $ & $\Delta _U$ & $\delta $ & $\delta $ & $\gamma $ & $\delta $ & $\gamma $\\ 
$\delta $ & $\Delta _U$ & $\delta $ & $\delta $ & $\delta $ & $\Delta _U$ & $\delta $\\ 
$\nabla _U$ & $\Delta _U$ & $\alpha $ & $\beta $ & $\gamma $ & $\delta $ & $\nabla _U$\end{tabular}}\end{picture}\end{tabular}
\end{center}

$U$ is not Abelian, nor is it solvable or nilpotent, as shown by the table of $[\cdot , \cdot ]_U$ above, but ${\rm Spec}(A)=\emptyset $, thus ${\cal L}(U)\cong {\cal L}_1$ by Remark \ref{retictriv}.\end{example}

\begin{example} Let $M=(\{a,b,x,y,z\},+)$ and $N=(\{a,b,c,x,y\},+)$, with $+$ defined by the following tables. Then ${\rm Con}(M)$ and ${\rm Con}(N)$ have the Hasse diagrams below, where:\begin{itemize}
\item $M/\alpha =\{\{a,b\},\{x,y,z\}\}$, $M/\beta =\{\{a,b\},\{x,y\},\{z\}\}$, $M/\gamma =\{\{a,b\},\{x,z\},\{y\}\}$, $M/\delta =\{\{a,b\},\{x\},$\linebreak $\{y,z\}\}$ and $M/\varepsilon =\{\{a,b\},\{x\},\{y\},\{z\}\}$;
\item $N/\chi =\{\{a,b,c\},\{x,y\}\}$, $N/\chi _1=\{\{a,b,c\},\{x\},\{y\}\}$, $N/\xi =\{\{a,b\},\{c\},\{x,y\}\}$, $N/\xi _1=\{\{a,b\},\{c\},\{x\},\{y\}\}$, $N/\psi =\{\{a\},\{b,c\},\{x,y\}\}$, $N/\psi _1=\{\{a\},\{b,c\},\{x\},\{y\}\}$ and $N/\phi =\{\{a\},\{b\},\{c\},\{x,y\}\}$.\end{itemize}

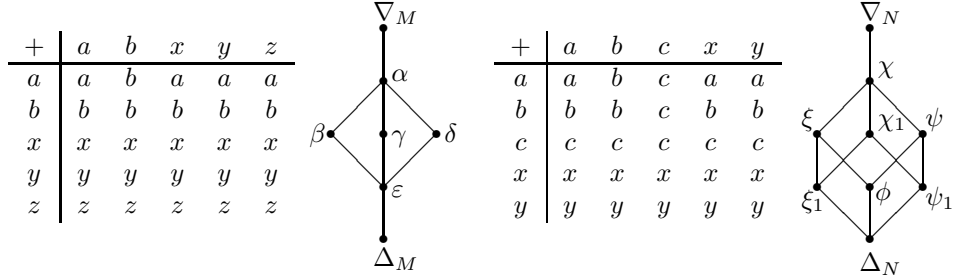
\begin{center}
\begin{tabular}{cccc}
\begin{picture}(80,85)(0,0)
\put(-10,40){\begin{tabular}{c|ccccc}
$+$ & $a$ & $b$ & $x$ & $y$ & $z$\\ \hline  
$a$ & $a$ & $b$ & $a$ & $a$ & $a$\\ 
$b$ & $b$ & $b$ & $b$ & $b$ & $b$\\ 
$x$ & $x$ & $x$ & $x$ & $x$ & $x$\\ 
$y$ & $y$ & $y$ & $y$ & $y$ & $y$\\ 
$z$ & $z$ & $z$ & $z$ & $z$ & $z$\end{tabular}}
\end{picture}
&

\begin{picture}(80,85)(0,0)
\put(36,-10){$\Delta _M$}
\put(36,83){$\nabla _M$}
\put(40,0){\circle*{3}}
\put(40,20){\circle*{3}}
\put(20,40){\circle*{3}}
\put(40,40){\circle*{3}}
\put(60,40){\circle*{3}}
\put(40,60){\circle*{3}}
\put(40,80){\circle*{3}}
\put(40,20){\line(-1,1){20}}
\put(40,20){\line(1,1){20}}
\put(40,0){\line(0,1){80}}
\put(40,60){\line(-1,-1){20}}
\put(40,60){\line(1,-1){20}}

\put(12,37){$\beta $}
\put(43,37){$\gamma $}
\put(63,37){$\delta $}
\put(43,60){$\alpha $}
\put(43,16){$\varepsilon $}
\end{picture}
&
\begin{picture}(80,85)(0,0)
\put(-10,40){\begin{tabular}{c|ccccc}
$+$ & $a$ & $b$ & $c$ & $x$ & $y$\\ \hline  

$a$ & $a$ & $b$ & $c$ & $a$ & $a$\\ 
$b$ & $b$ & $b$ & $c$ & $b$ & $b$\\ 
$c$ & $c$ & $c$ & $c$ & $c$ & $c$\\ 
$x$ & $x$ & $x$ & $x$ & $x$ & $x$\\ 
$y$ & $y$ & $y$ & $y$ & $y$ & $y$\end{tabular}}
\end{picture}
&
\begin{picture}(80,85)(0,0)

\put(40,0){\circle*{3}}
\put(36,-12){$\Delta _N$}
\put(40,0){\line(1,1){20}}
\put(40,0){\line(-1,1){20}}
\put(40,0){\line(0,1){20}}
\put(20,20){\circle*{3}}

\put(40,20){\circle*{3}}
\put(60,20){\circle*{3}}
\put(20,40){\circle*{3}}
\put(40,40){\circle*{3}}
\put(60,40){\circle*{3}}
\put(40,60){\circle*{3}}
\put(40,80){\circle*{3}}
\put(20,20){\line(0,1){20}}
\put(20,20){\line(1,1){20}}
\put(60,20){\line(-1,1){20}}
\put(60,20){\line(0,1){20}}
\put(40,20){\line(1,1){20}}
\put(40,20){\line(-1,1){20}}

\put(40,80){\line(0,-1){40}}
\put(40,60){\line(-1,-1){20}}
\put(40,60){\line(1,-1){20}}
\put(14,12){$\xi _1$}
\put(42,15){$\phi $}
\put(61,13){$\psi _1$}
\put(14,43){$\xi $}
\put(43,43){$\chi _1$}
\put(43,63){$\chi $}
\put(61,43){$\psi $}
\put(36,83){$\nabla _N$}
\end{picture}
\end{tabular}
\end{center}

Note that, despite the fact that $M$ is congruence--modular and $N$ is congruence--distributive, neither ${\cal HSP}(M)$, nor ${\cal HSP}(N)$ is congruence--modular, because $S=(\{a,b\},+)\cong ({\cal L}_2,\max )\cong (\Z _2,\cdot )$ is a subalgebra of both $M$ and $N$, and it can be easily checked that $S^2$ is not congruence--modular. Thus neither ${\cal HSP}(M)$, nor ${\cal HSP}(N)$ is semi—degenerate, which is also obvious from the fact that $(\{a\},+)$ is a subalgebra of both $M$ and $N$.

We have: $[\theta ,\zeta ]_M=\varepsilon $ for all $\theta ,\zeta \in [\varepsilon )$ and, of course, $[\Delta _M,\theta ]_M=[\theta ,\Delta _M]_M=\Delta _M$ for all $\theta \in {\rm Con}(M)$, hence ${\rm Spec}(M)=\{\Delta _M\}$ and thus ${\cal L}(M)\cong {\cal L}_2$, while $[\cdot ,\cdot ]_N$ is given by the following table, thus ${\rm Spec}(N)=\{\psi ,\xi \}$, so $\rho _N$ is defined as follows and hence ${\cal L}(M)\cong {\cal L}_2^2$:

\begin{center}
\begin{tabular}{ccc}
\begin{picture}(300,105)(0,0)
\put(0,50){\begin{tabular}{c|ccccccccc}
$[\cdot ,\cdot ]_N$ & $\Delta _N$ & $\psi $ & $\psi _1$ & $\phi $ & $\xi $ & $\xi _1$ & $\chi $ & $\chi _1$ & $\nabla _N$\\ \hline 
$\Delta _N$ & $\Delta _N$ & $\Delta _N$ & $\Delta _N$ & $\Delta _N$ & $\Delta _N$ & $\Delta _N$ & $\Delta _N$ & $\Delta _N$ & $\Delta _N$\\ 
$\psi $ & $\Delta _N$ & $\psi _1$ & $\psi _1$ & $\Delta _N$ & $\Delta _N$ & $\Delta _N$ & $\psi _1$ & $\psi _1$ & $\psi _1$\\ 
$\psi _1$ & $\Delta _N$ & $\psi _1$ & $\psi _1$ & $\Delta _N$ & $\Delta _N$ & $\Delta _N$ & $\psi _1$ & $\psi _1$ & $\psi _1$\\ 
$\phi $ & $\Delta _N$ & $\Delta _N$ & $\Delta _N$ & $\Delta _N$ & $\Delta _N$ & $\Delta _N$ & $\Delta _N$ & $\Delta _N$ & $\Delta _N$\\ 
$\xi $ & $\Delta _N$ & $\Delta _N$ & $\Delta _N$ & $\Delta _N$ & $\xi _1$ & $\xi _1$ & $\xi _1$ & $\xi _1$ & $\xi _1$\\ 
$\xi _1$ & $\Delta _N$ & $\Delta _N$ & $\Delta _N$ & $\Delta _N$ & $\xi _1$ & $\xi _1$ & $\xi _1$ & $\xi _1$ & $\xi _1$\\ 
$\chi $ & $\Delta _N$ & $\psi _1$ & $\psi _1$ & $\Delta _N$ & $\xi _1$ & $\xi _1$ & $\chi _1$ & $\chi _1$ & $\chi _1$\\ 
$\chi _1$ & $\Delta _N$ & $\psi _1$ & $\psi _1$ & $\Delta _N$ & $\xi _1$ & $\xi _1$ & $\chi _1$ & $\chi _1$ & $\chi _1$\\ 
$\nabla _N$ & $\Delta _N$ & $\psi _1$ & $\psi _1$ & $\Delta _N$ & $\xi _1$ & $\xi _1$ & $\chi _1$ & $\chi _1$ & $\chi _1$
\end{tabular}}
\end{picture}
&
\begin{picture}(60,105)(0,0)
\put(0,50){\begin{tabular}{c|c}
$\theta $ & $\rho _N(\theta )$\\ \hline 
$\Delta _N$ & $\phi $\\ 
$\psi $ & $\psi $\\ 
$\psi _1$ & $\psi $\\ 
$\phi $ & $\phi $\\ 
$\xi $ & $\xi $\\ 
$\xi _1$ & $\xi $\\ 
$\chi $ & $\nabla _N$\\ 

$\chi _1$ & $\nabla _N$\\ 
$\nabla _N$ & $\nabla _N$
\end{tabular}}
\end{picture}
&
\begin{picture}(80,100)(0,0)
\put(38,10){${\bf 0}$}
\put(40,20){\circle*{3}}
\put(20,40){\circle*{3}}
\put(60,40){\circle*{3}}
\put(40,60){\circle*{3}}
\put(40,20){\line(-1,1){20}}

\put(40,20){\line(1,1){20}}
\put(40,60){\line(-1,-1){20}}
\put(40,60){\line(1,-1){20}}
\put(13,37){$\widehat{\xi }$}
\put(63,37){$\widehat{\psi }$}
\put(38,63){${\bf 1}$}
\put(28,-2){${\cal L}(N)$}

\end{picture}

\end{tabular}
\end{center}
\end{example}

\begin{example} Here are some finite congruence--distributive examples, thus in which the reticulations are isomorphic to the congruence lattices. Regarding the preservation properties fulfilled by the reticulation, these examples show that there is no embedding relation between the reticulation of an algebra and those of its subalgebras: if ${\cal E}$ is the following bounded lattice, then, for instance, $\{0,x,y,1\}={\cal L}_4={\cal L}_2\oplus {\cal L}_2\oplus {\cal L}_2$, ${\cal D}=\{0,a,x,b,1\}$ and ${\cal P}=\{0,a,x,y,1\}$ are bounded sublattices of ${\cal E}$. We have: ${\cal L}({\cal E})\cong {\rm Con}({\cal E})=\{\Delta _{\cal E},\mu ,\nabla _{\cal E}\}\cong {\cal L}_3$, where ${\cal E}/\mu =\{\{0\},\{a\},\{x,y\},\{b\},\{1\}\}$, ${\cal L}({\cal L}_4)\cong {\rm Con}({\cal L}_4)={\rm Con}({\cal L}_2\oplus {\cal L}_2\oplus {\cal L}_2)\cong {\rm Con}({\cal L}_2)^3\cong {\cal L}_2^3$, ${\cal L}({\cal D})\cong {\rm Con}({\cal D})=\{\Delta _{\cal D},\nabla _{\cal D}\}\cong {\cal L}_2$ and ${\cal L}({\cal P})\cong {\rm Con}({\cal P})=\{\Delta _{\cal P},\alpha ,\beta ,\gamma ,\nabla _{\cal P}\}\cong {\cal L}_2\oplus {\cal L}_2^2$, where ${\cal P}/\alpha =\{\{0,x,y\},\{a,1\}\}$, ${\cal P}/\beta =\{\{0,a\},\{x,y,1\}\}$ and ${\cal P}/\gamma =\{\{0\},\{a\},\{x,y\},\{1\}\}$:\vspace*{-9pt}

\begin{center}
\begin{tabular}{ccc}
\begin{picture}(80,70)(0,0)
\put(40,20){\circle*{3}}
\put(20,40){\circle*{3}}
\put(40,33){\circle*{3}}
\put(40,47){\circle*{3}}
\put(60,40){\circle*{3}}
\put(40,60){\circle*{3}}
\put(40,20){\line(-1,1){20}}
\put(40,20){\line(1,1){20}}
\put(40,20){\line(0,1){40}}
\put(40,60){\line(-1,-1){20}}
\put(40,60){\line(1,-1){20}}
\put(13,37){$a$}
\put(43,30){$x$}
\put(43,44){$y$}
\put(63,37){$b$}
\put(38,10){$0$}
\put(38,63){$1$}
\put(-5,36){${\cal E}:$}
\end{picture}
&
\begin{picture}(80,70)(0,0) 
\put(36,59){$\nabla _{\cal E}$}
\put(43,36){$\mu $}
\put(40,20){\line(0,1){36}}
\put(40,20){\circle*{3}}
\put(40,38){\circle*{3}}
\put(40,56){\circle*{3}}
\put(36,10){$\Delta _{\cal E}$}
\end{picture}
&
\begin{picture}(80,70)(0,0)
\put(36,10){$\Delta _{\cal P}$}

\put(42,30){$\gamma $}
\put(40,20){\circle*{3}}
\put(40,35){\circle*{3}}
\put(30,45){\circle*{3}}
\put(50,45){\circle*{3}}
\put(40,55){\circle*{3}}
\put(40,20){\line(0,1){15}}
\put(40,35){\line(-1,1){10}}
\put(40,35){\line(1,1){10}}
\put(40,55){\line(-1,-1){10}}
\put(40,55){\line(1,-1){10}}
\put(22,43){$\alpha $}
\put(53,42){$\beta $}
\put(36,58){$\nabla _{\cal P}$}
\end{picture}
\end{tabular}
\end{center}\label{reticlat}\end{example}\vspace*{-30pt}

\begin{remark} By \cite[Lemma $3.3$]{cblp}, in any variety, arbitrary intersections commute with arbitrary direct products of congruences. If ${\cal C}$ is congruence--modular and $M$ is an algebra from ${\cal C}$ such that $A\times M$ has no skew congruences, then ${\rm Spec}(A\times M)=\{\phi \times \nabla _M\ |\ \phi \in {\rm Spec}(A)\}\cup \{\nabla _A\times \psi \ |\ \psi \in {\rm Spec}(M)\}$. This follows from Proposition \ref{comutprod} in the same way as in the congruence--distributive case, treated in \cite[Proposition $3.5$,(ii)]{cblp}.\label{conprod}\end{remark}

\begin{proposition}[the reticulation preserves finite direct products without skew congruences] Let $M$ be an algebra from ${\cal C}$ such that the direct product $A\times M$ has no skew congruences. Then:\begin{enumerate}
\item\label{reticprod1} for all $\emptyset \neq X\subseteq A^2$ and all $\emptyset \neq Y\subseteq M^2$, $Cg_{A\times M}(X\times Y)=Cg_A(X)\times Cg_M(Y)$, and the map $(\alpha ,\mu )\mapsto \alpha \times \mu $ is a lattice isomorphism from ${\rm Con}(A)\times {\rm Con}(M)$ to ${\rm Con}(A\times M)$;
\item\label{reticprod2} ${\rm PCon}(A\times M)=\{\alpha \times \mu \ |\ \alpha \in {\rm PCon}(A),\mu \in {\rm PCon}(M)\}$ and ${\cal K}(A\times M)=\{\alpha \times \mu \ |\ \alpha \in {\cal K}(A),\mu \in {\cal K}(M)\}$.\end{enumerate}

If ${\cal C}$ is congruence--modular and $\nabla _M\in {\cal K}(M)$, then:\begin{itemize}
\item for all $\alpha \in {\rm Con}(A)$ and all $\mu \in {\rm Con}(M)$, $\rho _{A\times M}(\alpha \times \mu )=\rho _A(\alpha )\times \rho _{M}(\mu )$;
\item $\equiv _{A\times M}=\equiv _A\times \equiv _M$ and ${\cal L}(A\times M)\cong {\cal L}(A)\times {\cal L}(M)$.\end{itemize}\label{reticprod}\end{proposition}

\begin{proof} $A\times M$ has no skew congruences, that is ${\rm Con}(A\times M)=\{\alpha \times \mu \ |\ \alpha \in {\rm Con}(A),\mu \in {\rm Con}(M)\}$.

\noindent (\ref{reticprod1}) By Remark \ref{conprod}, $Cg_{A\times M}(X\times Y)=\bigcap \{\theta \in {\rm Con}(A\times M)\ |\ X\times Y\subseteq \theta \}=\bigcap \{\alpha \times \mu \ |\ \alpha \in {\rm Con}(A),\mu \in {\rm Con}(M),X\times Y\subseteq \alpha \times \mu \}=\bigcap \{\alpha \times \mu \ |\ \alpha \in {\rm Con}(A),\mu \in {\rm Con}(M),X\subseteq \alpha ,Y\subseteq \mu \}=(\bigcap \{\alpha \in {\rm Con}(A)\ |\ X\subseteq \alpha \})\times (\bigcap \{\mu \in {\rm Con}(M) \ |\ Y\subseteq \mu \})=Cg_A(X)\times Cg_M(Y)$.

This also shows that the map $(\alpha ,\mu )\mapsto \alpha \times \mu $ is a lattice isomorphism from ${\rm Con}(A)\times {\rm Con}(M)$ to ${\rm Con}(A\times M)$, because it is clearly injective, it is surjective by the above, it preserves the intersection by Remark \ref{conprod} and, for all $\alpha ,\beta \in {\rm Con}(A)$ and all $\mu ,\nu \in {\rm Con}(M)$, $(\alpha \times \mu )\vee (\beta \times \nu )=Cg_{A\times M}((\alpha \times \mu )\cup (\beta \times \nu ))\subseteq Cg_{A\times M}((\alpha \cup \beta )\times (\mu \cup \nu ))=Cg_A(\alpha \cup \beta )\times Cg_M(\mu \cup \nu )=(\alpha \vee \beta )\times (\mu \vee \nu )$, since, clearly, $(\alpha \times \mu )\cup (\beta \times \nu )\subseteq (\alpha \cup \beta )\times (\mu \cup \nu )$, but, also, $(\alpha \times \mu )\vee (\beta \times \nu )\in {\rm Con}(A\times M)$, thus $(\alpha \times \mu )\cup (\beta \times \nu )\subseteq (\alpha \times \mu )\vee (\beta \times \nu )=\gamma \vee \sigma $ for some $\gamma \in {\rm Con}(A)$ and $\sigma \in {\rm Con}(M)$, so $\alpha \times \mu \subseteq \gamma \vee \sigma $ and $\beta \times \nu \subseteq \gamma \vee \sigma $, hence $\alpha \subseteq \gamma $, $\beta \subseteq \gamma $, $\mu \subseteq \sigma $ and $\nu \subseteq \sigma $, so $(\alpha \vee \beta )\subseteq \gamma $ and $(\mu \vee \nu )\subseteq \sigma $, hence $(\alpha \vee \beta )\times (\mu \vee \nu )\subseteq \gamma \vee \sigma =(\alpha \times \mu )\vee (\beta \times \nu )\subseteq (\alpha \vee \beta )\times (\mu \vee \nu )$, therefore $(\alpha \vee \beta )\times (\mu \vee \nu )=(\alpha \times \mu )\vee (\beta \times \nu )$.

\noindent (\ref{reticprod2}) By (\ref{reticprod1}), for all $a,b\in A$ and all $u,v\in M$, $Cg_{A\times M}((a,u),(b,v))=Cg_A(a,b)\times Cg_M(u,v)$, hence the expression of ${\rm PCon}(A\times M)$ in the enunciation. From this and the second statement in (\ref{reticprod1}), we obtain: $\displaystyle {\cal K}(A\times M)=\{Cg_{A\times M}(\{(a_1,u_1),\ldots ,(a_n,u_n)\})\ |\ n\in \N ^*,a_1,\ldots ,a_n\in A,u_1,\ldots ,u_n\in M\}=\{\bigvee _{i=1}^nCg_{A\times M}(a_i,u_i)\ |\ n\in \N ^*,a_1,\ldots ,a_n\in A,u_1,\ldots ,u_n\in M\}=\{\bigvee _{i=1}^n(Cg_A(a_i)\times Cg_M(u_i))\ |\ n\in \N ^*,a_1,\ldots ,a_n\in A,u_1,\ldots ,u_n\in M\}=\{(\bigvee _{i=1}^nCg_A(a_i))\times (\bigvee _{i=1}^nCg_M(u_i))\ |\ n\in \N ^*,a_1,\ldots ,a_n\in A,u_1,\ldots ,u_n\in M\}=\{Cg_A(a_1,\ldots ,a_n)\times Cg_M(u_1,\ldots ,u_n)\ |\ n\in \N ^*,a_1,\ldots ,a_n\in A,u_1,\ldots ,u_n\in M\}=\{\alpha \times \mu \ |\ \alpha \in {\cal K}(A),\mu \in {\cal K}(M)\}$, since the above also hold if some of the elements $a_1,\ldots ,a_n$ or $u_1,\ldots ,u_n$ coincide.

Now assume that ${\cal C}$ is congruence--modular and $\nabla _M\in {\cal K}(M)$. Then, by Remark \ref{conprod}, for any $\alpha \in {\rm Con}(A)$ and any $\mu \in {\rm Con}(M)$, $\rho _{A\times M}(\alpha \times \mu )=\bigcap \{\chi \in {\rm Spec}(A\times M)\ |\ \alpha \times \mu \subseteq \chi \}=\bigcap \{\phi \times \nabla _M\ |\ \phi \in {\rm Spec}(A),\alpha \times \mu \subseteq \phi \times \nabla _M\}\cap \bigcap \{\nabla _A\times \psi \ |\ \psi \in {\rm Spec}(M),\alpha \times \mu \subseteq \nabla _A\times \psi \}=\bigcap \{\phi \times \nabla _M\ |\ \phi \in {\rm Spec}(A),\alpha \subseteq \phi \}\cap \bigcap \{\nabla _A\times \psi \ |\ \psi \in {\rm Spec}(M),\mu \subseteq \psi \}=(\bigcap \{\phi \ |\ \phi \in {\rm Spec}(A),\alpha \subseteq \phi \}\times \nabla _M)\cap (\nabla _A\times \bigcap \{\psi \ |\ \psi \in {\rm Spec}(M),\mu \subseteq \psi \})=(\rho _A(\alpha )\times \nabla _M)\cap (\nabla _A\times \rho _M(\mu ))=(\rho _A(\alpha )\cap \nabla _A)\times (\nabla _M\cap \rho _M(\mu ))=\rho _A(\alpha )\times \rho _M(\mu )$. Hence, for all $\theta ,\zeta \in {\rm Con}(A\times M)$, we have: $\theta =\alpha \times \mu $ and $\zeta =\beta \times \nu $ for some $\alpha ,\beta \in {\rm Con}(A)$ and $\mu ,\nu \in {\rm Con}(M)$, and thus: $\theta \equiv _{A\times M}\zeta $ iff $\rho _{A\times M}(\theta )=\rho _{A\times M}(\zeta )$ iff $\rho _{A\times M}(\alpha \times \mu )=\rho _{A\times M}(\beta \times \nu )$ iff $\rho _A(\alpha )\times \rho _M(\mu )=\rho _A(\beta )\times \rho _M(\nu )$ iff $\rho _A(\alpha )=\rho _A(\beta )=\rho _A(\beta )$ and $\rho _M(\mu )=\rho _M(\nu )$ iff $\alpha \equiv _A\beta $ and $\mu \equiv _M\nu $.

Now let $\varphi :{\cal L}(A)\times {\cal L}(M)\rightarrow {\cal L}(A\times M)$, for all $\alpha \in {\cal K}(A)$ and all $\mu \in {\cal K}(M)$, $\varphi (\widehat{\alpha },\widehat{\mu })=\widehat{\alpha \times \mu }$. By (\ref{reticprod2}), $\varphi $ is well defined and surjective and fulfills: $\varphi ((\widehat{\alpha },\widehat{\mu })\vee (\widehat{\beta },\widehat{\nu }))=\varphi (\widehat{\alpha }\vee \widehat{\beta },\widehat{\mu }\vee \widehat{\nu })=\varphi (\widehat{\alpha \vee \beta },\widehat{\mu \vee \nu })=((\alpha \vee \beta )\times (\mu \vee \nu ))^{\wedge }=((\alpha \times \mu )\vee (\beta \times \nu ))^{\wedge }=\widehat{(\alpha \times \mu )}\vee \widehat{(\beta \times \nu )}=\varphi (\widehat{\alpha },\widehat{\mu })\vee \varphi (\widehat{\beta },\widehat{\nu })$ and, similarly, $\varphi ((\widehat{\alpha },\widehat{\mu })\wedge (\widehat{\beta },\widehat{\nu }))=\varphi (\widehat{\alpha },\widehat{\mu })\wedge \varphi (\widehat{\beta },\widehat{\nu })$. By the form of $\equiv _{A\times M}$ above, $\varphi $ is injective. Hence $\varphi $ is a lattice isomorphism.\end{proof}

\begin{example} Let ${\cal V}$ be the variety generated by the variety of lattices and that of groups. Then, according to \cite[Theorem $1$, Lemma $1$, Proposition $3$]{chaj} and \cite{klp}, ${\cal V}$ is congruence--modular and any algebra $M$ from ${\cal V}$ is of the form $M=(L,\vee ,\wedge )\times (G,\cdot ,\star )$, where $(L,\vee ,\wedge )$ is a lattice, $(G,\cdot )$ is a group and $x\star y=x^{-1}\cdot y$ for all $x,y\in G$, and the direct product above has no skew congruences, thus, by Proposition \ref{reticprod}, ${\rm Con}(M)\cong {\rm Con}(L)\times {\rm Con}(G)$ and ${\cal L}(M)\cong {\cal L}(L)\times {\cal L}(G)$, since each congruence of the group $G$ also preserves the operation $\star $. Thus, for instance, in we consider the lattice ${\cal P}$ from Example \ref{reticlat} and the group $(S_3,\circ )$ from Example \ref{reticgr}, and we denote $\sigma \star \tau =\sigma ^{-1}\circ \tau $ for all $\sigma ,\tau \in S_3$, and $M=({\cal P},\vee ,\wedge )\times (S_3,\circ ,\star )$, then $M$ is a finite algebra from ${\cal V}$ which is not congruence--distributive, because ${\rm Con}(M)\cong {\rm Con}({\cal P})\times {\rm Con}(S_3)$ and ${\rm Con}(S_3)$ is not distributive, and ${\cal L}(M)\cong {\cal L}({\cal P})\times {\cal L}(S_3)\cong {\cal L}({\cal P})\times {\cal L}_1\cong {\cal L}({\cal P})\cong {\rm Con}({\cal P})\cong {\cal L}_2\oplus {\cal L}_2^2$.\end{example}

\section{Further Results on The Commutator}
\label{further}

Throughout this section, we shall assume that $[\cdot ,\cdot ]_A$ is commutative and distributive w.r.t. arbitrary joins and $\nabla _A\in {\cal K}(A)$.

\begin{lemma} For all $n\in \N ^*$ and all $\alpha ,\beta \in {\rm Con}(A)$, $[\alpha ,\beta ]_A^{n+1}=[[\alpha ,\beta ]_A,[\alpha ,\beta ]_A]_A^n$.\label{lema2.9}\end{lemma}

\begin{proof} Let $\alpha ,\beta \in {\rm Con}(A)$. We proceed by induction on $n$. By its definition, $[\alpha ,\beta ]_A^2=[[\alpha ,\beta ]_A,[\alpha ,\beta ]_A]_A$. Now let $n\in \N ^*$ such that $[\alpha ,\beta ]_A^{n+1}=[[\alpha ,\beta ]_A,[\alpha ,\beta ]_A]_A^n$. Then, by the induction hypothesis, $[\alpha ,\beta ]_A^{n+2}=[[\alpha ,\beta ]_A^{n+1},[\alpha ,\beta ]_A^{n+1}]_A=[[[\alpha ,\beta ]_A,[\alpha ,\beta ]_A]_A^n,[[\alpha ,\beta ]_A,[\alpha ,\beta ]_A]_A^n]_A=[[\alpha ,\beta ]_A,[\alpha ,\beta ]_A]_A^{n+1}$.\end{proof}

\begin{lemma} If the commutator of $A$ is associative, then, for any $n\in \N ^*$ and all $\alpha ,\beta \in {\rm Con}(A)$, $[\alpha ,\beta ]_A^{n+1}=[[\alpha ,\alpha ]_A^n,[\beta ,\beta ]_A^n]_A$.\label{commassoc}\end{lemma}

\begin{proof} Assume that the commutator of $A$ is associative, and let us also use its commutativity, along with Lemma \ref{lema2.9}. Let $\alpha ,\beta \in {\rm Con}(A)$. We apply induction on $n$. For $n=1$, $[\alpha ,\beta ]_A^2=[[\alpha ,\beta ]_A,[\alpha ,\beta ]_A]_A=[[\alpha ,\alpha ]_A,[\beta ,\beta ]_A]_A$. Now let $n\in \N ^*$ such that $[\alpha ,\beta ]_A^{n+1}=[[\alpha ,\alpha ]_A^n,[\beta ,\beta ]_A^n]_A$. Then $[\alpha ,\beta ]_A^{n+2}=[[\alpha ,\beta ]_A^{n+1},[\alpha ,\beta ]_A^{n+1}]_A=[[[\alpha ,\alpha ]_A^n, $\linebreak $[\beta ,\beta ]_A^n]_A,[[\alpha ,\alpha ]_A^n,[\beta ,\beta ]_A^n]_A]_A=[[[\alpha ,\alpha ]_A^n,[\alpha ,\alpha ]_A^n]_A,[[\beta ,\beta ]_A^n,[\beta ,\beta ]_A^n]_A]_A=[[\alpha ,\alpha ]_A^{n+1},[\beta ,\beta ]_A^{n+1}]_A$.\end{proof}

\begin{lemma} For all $n,k\in \N ^*$ and all $\alpha ,\beta ,\phi ,\psi ,\alpha _1,\alpha _2,\ldots ,\alpha _k\in {\rm Con}(A)$:\begin{enumerate}
\item\label{aritmcomut1} if $\alpha \subseteq \beta $ and $\phi \subseteq \psi $, then $[\alpha ,\phi ]_A^n\subseteq [\beta ,\psi ]_A^n$;
\item\label{aritmcomut2} if $k\leq n$, then $[\alpha ,\beta ]_A^n\subseteq [\alpha ,\beta ]_A^k$;
\item\label{aritmcomut3} if $k\geq 2$ and $n\geq 2$, then $[\alpha ,\beta ]_A^{k\cdot n}\subseteq [[\alpha ,\beta ]_A^k,[\alpha ,\beta ]_A^k]_A^n$;
\item\label{aritmcomut4} $[\alpha \vee \beta ,\alpha \vee \beta ]_A^n\subseteq \alpha \vee [\beta ,\beta ]_A^n$;
\item\label{aritmcomut5} $[\alpha \vee \beta ,\alpha \vee \beta ]_A^{n\cdot k}\subseteq [\alpha ,\alpha ]_A^k\vee [\beta ,\beta ]_A^n$;
\item\label{aritmcomut6} $[\alpha \vee \beta ,\alpha \vee \beta ]_A^{n^2}\subseteq [\alpha ,\alpha ]_A^n\vee [\beta ,\beta ]_A^n$;
\item\label{aritmcomut7} $[\alpha _1\vee \ldots \vee \alpha _k,\alpha _1\vee \ldots \vee \alpha _k]_A^{n^k}\subseteq [\alpha _1,\alpha _1]_A^n\vee \ldots \vee [\alpha _k,\alpha _k]_A^n$.\end{enumerate}\label{aritmcomut}\end{lemma}

\begin{proof} (\ref{aritmcomut1}) By Proposition \ref{1.3}, through induction on $n$.

\noindent (\ref{aritmcomut2}) For all $p\in \N ^*$, $[\alpha ,\beta ]_A^{p+1}=[[\alpha ,\beta ]_A^p,[\alpha ,\beta ]_A^p]_A\subseteq [\alpha ,\beta ]_A^p$, hence the inclusion in the enunciation.

\noindent (\ref{aritmcomut3}) Assume that $n\geq 2$. We apply induction on $k$, (\ref{aritmcomut2}) and Lemma \ref{lema2.9}. For $k=2$, we have: $[[\alpha ,\beta ]_A^2,[\alpha ,\beta ]_A^2]_A^n=[[[\alpha ,\beta ]_A, [\alpha ,\beta ]_A]_A,[[\alpha ,\beta ]_A,[\alpha ,\beta ]_A]_A]_A^n=[\alpha ,\beta ]_A^{n+2}\supseteq [\alpha ,\beta ]_A^{2n}$. Now take a $k\geq 2$ that fulfills the inclusion in the enunciation for all $\alpha ,\beta \in {\rm Con}(A)$. Then $[[\alpha ,\beta ]_A^{k+1},[\alpha ,\beta ]_A^{k+1}]_A^n=[[[\alpha ,\beta ]_A,[\alpha ,\beta ]_A]_A^k,[[\alpha ,\beta ]_A,[\alpha ,\beta ]_A]_A^k]_A^n\supseteq [[\alpha ,\beta ]_A,[\alpha ,\beta ]_A]_A^{k\cdot n}=[\alpha ,\beta ]_A^{k\cdot n+1}\supseteq [\alpha ,\beta ]_A^{k\cdot n+n}=[\alpha ,\beta ]_A^{(k+1)\cdot n}$.

\noindent (\ref{aritmcomut4}) We apply induction on $n$. For $n=1$ and all $\alpha ,\beta \in {\rm Con}(A)$, we have $[\alpha \vee \beta ,\alpha \vee \beta ]_A=[\alpha ,\alpha ]_A\vee [\alpha ,\beta ]_A\vee [\beta ,\alpha ]_A\vee [\beta ,\beta ]_A\subseteq \alpha \vee [\beta ,\beta ]_A$. Now let $n\in \N ^*$ such that $[\alpha \vee \beta ,\alpha \vee \beta ]_A^n\subseteq \alpha \vee [\beta ,\beta ]_A^n$ for all $\alpha ,\beta \in {\rm Con}(A)$. Then, by the induction hypothesis and the case $n=1$, we have, for all $\alpha ,\beta \in {\rm Con}(A)$: $[\alpha \vee \beta ,\alpha \vee \beta ]_A^{n+1}=[[\alpha \vee \beta ,\alpha \vee \beta ]_A^n,[\alpha \vee \beta ,\alpha \vee \beta ]_A^n]_A\subseteq [\alpha \vee [\beta ,\beta ]_A^n,\alpha \vee [\beta ,\beta ]_A^n]_A\subseteq \alpha \vee [[\beta ,\beta ]_A^n,[\beta ,\beta ]_A^n]_A=\alpha \vee [\beta ,\beta ]_A^{n+1}$.

\noindent (\ref{aritmcomut5}) We apply Lemma \ref{aritmcomut}. For $n=1$, $[\alpha ,\alpha ]_A^k\vee [\beta ,\beta ]_A\supseteq [\alpha \vee \beta ,\alpha \vee \beta ]_A^k$. For $k=1$, $[\alpha ,\alpha ]_A\vee [\beta ,\beta ]_A^n\supseteq [\alpha \vee \beta ,\alpha \vee \beta ]_A^n$. For $k\geq 2$ and $n\geq 2$, $[\alpha ,\alpha ]_A^k\vee [\beta ,\beta ]_A^n\supseteq [[\alpha ,\alpha ]_A^k\vee \beta ,[\alpha ,\alpha ]_A^k\vee \beta ]_A^n\supseteq [[\alpha \vee \beta ,\alpha \vee \beta ]_A^k,[\alpha \vee \beta ,\alpha \vee \beta ]_A^k]_A^n\supseteq [\alpha \vee \beta ,\alpha \vee \beta ]_A^{k\cdot n}$.

\noindent (\ref{aritmcomut6}) Take $k=n$ in (\ref{aritmcomut5}).

\noindent (\ref{aritmcomut7}) We apply induction on $k$. The statement is trivial for $k=1$. Let $k\in \N ^*$ that fulfills the equality in the enunciation for any congruences of $A$, and let $\alpha _1,\ldots ,\alpha _k,\alpha _{k+1}\in {\rm Con}(A)$. By (\ref{aritmcomut5}), it follows that $[\alpha _1\vee \ldots \vee \alpha _k\vee \alpha _{k+1},\alpha _1\vee \ldots \vee \alpha _k\vee \alpha _{k+1}]_A^{n^{\scriptstyle k+1}}=[(\alpha _1\vee \ldots \vee \alpha _k)\vee \alpha _{k+1},(\alpha _1\vee \ldots \vee \alpha _k)\vee \alpha _{k+1}]_A^{n^{\scriptstyle k}\cdot n}\subseteq [\alpha _1\vee \ldots \vee \alpha _k,\alpha _1\vee \ldots \vee \alpha _k]_A^{n^{\scriptstyle k}}\vee =[\alpha _{k+1},\alpha _{k+1}]_A^n\subseteq [\alpha _1,\alpha _1]_A^n\vee \ldots \vee [\alpha _k,\alpha _k]_A^n\vee [\alpha _{k+1},\alpha _{k+1}]_A^n$.\end{proof}

For all $\theta ,\zeta \in {\rm Con}(A)$, we shall denote by $\displaystyle \theta \rightarrow \zeta =\bigvee \{\alpha \in {\rm Con}(A)\ |\ [\theta ,\alpha ]_A\subseteq \zeta \}$ and by $\displaystyle \theta ^{\perp }=\theta \rightarrow \Delta _A=\bigvee \{\alpha \in {\rm Con}(A)\ |\ [\theta ,\alpha ]_A=\Delta _A\}$.

\begin{remark} For all $\theta ,\zeta \in {\rm Con}(A)$, $\theta \rightarrow \zeta =\max \{\alpha \in {\rm Con}(A)\ |\ [\theta ,\alpha ]_A\subseteq \zeta \}$, because, if we denote by $M=\{\alpha \in {\rm Con}(A)\ |\ [\theta ,\alpha ]_A\subseteq \zeta \}$, then $\displaystyle [\theta ,\theta \rightarrow \zeta ]_A=[\theta ,\bigvee _{\alpha \in M}\alpha ]_A=\bigvee _{\alpha \in M}[\theta ,\alpha ]_A\subseteq \zeta $, hence $\theta \rightarrow \zeta \in M$.\end{remark}

\begin{lemma} For all $\alpha ,\beta ,\gamma \in {\rm Con}(A)$, $[\alpha ,\beta ]_A\subseteq \gamma $ iff $\alpha \subseteq \beta \rightarrow \gamma $.\label{reziduatie}\end{lemma}

\begin{proof} ``$\Rightarrow $:`` $\displaystyle \beta \rightarrow \gamma =\bigvee \{\theta \in {\rm Con}(A)\ |\ [\beta ,\theta ]_A\subseteq \gamma \}$. Since $[\beta ,\alpha ]_A=[\alpha ,\beta ]_A\subseteq \gamma $, it follows that $\alpha \subseteq \beta \rightarrow \gamma $.

\noindent ``$\Leftarrow $:`` We have $\displaystyle \alpha \subseteq \beta \rightarrow \gamma =\bigvee \{\theta \in {\rm Con}(A)\ |\ [\beta ,\theta ]_A\subseteq \gamma \}$, hence $\displaystyle [\beta ,\alpha ]_A=[\alpha ,\beta ]_A\subseteq [\beta ,\beta \rightarrow \gamma ]_A=[\beta ,\bigvee \{\theta \in {\rm Con}(A)\ |\ [\beta ,\theta ]_A\subseteq \gamma \}]_A=\bigvee \{[\beta ,\theta ]_A\ |\ \theta \in {\rm Con}(A),[\beta ,\theta ]_A\subseteq \gamma \}\subseteq \gamma $.\end{proof}

For the following results, recall, also, the equivalences in Proposition \ref{prodcongr}.

\begin{lemma} For all $\alpha ,\beta \in {\rm Con}(A)$ such that $[\alpha ,\nabla _A]_A=\alpha $: $\alpha \rightarrow \beta =\nabla _A$ iff $\alpha \subseteq \beta $.\label{implicnabla}\end{lemma}

\begin{proof} $\alpha \rightarrow \beta =\nabla _A$ iff $\nabla _A\subseteq \alpha \rightarrow \beta $ iff $\alpha =[\nabla _A,\alpha ]_A\subseteq \beta $, according to Lemma \ref{reziduatie}.\end{proof}

\begin{remark} By the above, if $[\cdot ,\cdot ]_A$ is associative, then $({\rm Con}(A),\vee ,\cap ,[\cdot ,\cdot ]_A,\rightarrow ,\Delta _A,\nabla _A)$ is a residuated lattice, and, if a ${\cal C}$ is a congruence--distributive variety, then $({\rm Con}(A),\vee ,\cap ,[\cdot ,\cdot ]_A,\rightarrow ,\Delta _A,\nabla _A)$ is, moreover, a ${\rm G\ddot{o}del}$ algebra.\end{remark}

\begin{proposition} If $[\theta ,\nabla _A]_A=\theta $ for all $\theta \in {\rm Con}(A)$, for any $\alpha ,\beta ,\gamma \in {\rm Con}(A)$:\begin{enumerate}
\item\label{prop4.1(1)} if $\alpha \vee \beta =\nabla _A$, then $[\alpha ,\beta ]_A=\alpha \cap \beta $;
\item\label{prop4.1(2)} if $\alpha \vee \beta =\alpha \vee \gamma =\nabla _A$, then $\alpha \vee [\beta ,\gamma ]_A=\alpha \vee (\beta \cap \gamma )=\nabla _A $;
\item\label{prop4.1(3)} if $\alpha \vee \beta =\nabla _A$, then $[\alpha ,\alpha ]_A^n\vee [\beta ,\beta ]_A^n=\nabla _A$ for all $n\in \N ^*$.\end{enumerate}\label{prop4.1}\end{proposition}

\begin{proof} (\ref{prop4.1(1)}) Assume that $\alpha \vee \beta =\nabla _A$. Since $\displaystyle (\alpha \cap \beta )\rightarrow [\alpha ,\beta ]_A=\bigvee \{\theta \in {\rm Con}(A)\ |\ [\alpha \cap \beta ,\theta ]_A\subseteq [\alpha ,\beta ]_A\}$ and $[\alpha \cap \beta ,\beta ]_A\subseteq [\alpha ,\beta ]_A$ and $[\alpha ,\alpha \cap \beta ]_A\subseteq [\alpha ,\beta ]_A$, it follows that $\alpha \subseteq (\alpha \cap \beta )\rightarrow [\alpha ,\beta ]_A$ and $\beta \subseteq (\alpha \cap \beta )\rightarrow [\alpha ,\beta ]_A$, hence $\nabla _A=\alpha \vee \beta \subseteq (\alpha \cap \beta )\rightarrow [\alpha ,\beta ]_A$, therefore $(\alpha \cap \beta )\rightarrow [\alpha ,\beta ]_A=\nabla _A$, thus $\alpha \cap \beta \subseteq [\alpha ,\beta ]_A$ by Lemma \ref{implicnabla}. Since the converse inclusion always holds, it follows that $\alpha \cap \beta =[\alpha ,\beta ]_A$.

\noindent (\ref{prop4.1(2)}) Assume that $\alpha \vee \beta =\alpha \vee \gamma =\nabla _A$, so that $\nabla _A=[\nabla _A,\nabla _A]_A=[\alpha \vee \beta ,\alpha \vee \gamma ]_A=[\alpha ,\alpha ]_A\vee [\beta ,\alpha ]_A\vee [\alpha ,\gamma ]_A\vee [\beta ,\gamma ]_A\subseteq \alpha \vee [\beta ,\gamma ]_A\subseteq \alpha \vee (\beta \cap \gamma )\subseteq \nabla _A$, hence $\alpha \vee [\beta ,\gamma ]_A=\alpha \vee (\beta \cap \gamma )=\nabla _A$.

\noindent (\ref{prop4.1(3)}) We apply induction on $n$. Assume that $\alpha \vee \beta =\nabla _A$, so that, by (\ref{prop4.1(2)}), $\alpha \vee [\beta ,\beta ]_A=\nabla _A$, thus $[\alpha ,\alpha ]_A\vee [\beta ,\beta ]_A=\nabla _A$, hence the implication holds in the case $n=1$. Now, if $n\in \N ^*$ fulfills the implication in the enunciation for all $\alpha ,\beta \in {\rm Con}(A)$, and assume that $\alpha \vee \beta =\nabla _A$, so that $[\alpha ,\alpha ]_A^n\vee [\beta ,\beta ]_A^n=\nabla _A$. Then, by the case $n=1$, it follows that $[\alpha ,\alpha ]_A^{n+1}\vee [\beta ,\beta ]_A^{n+1}=[[\alpha ,\alpha ]_A^n,[\alpha ,\alpha ]_A^n]_A\vee [[\beta ,\beta ]_A^n,[\beta ,\beta ]_A^n]_A=\nabla _A$.\end{proof}

\begin{lemma} If $[\gamma ,\nabla _A]_A=\gamma $ for all $\gamma \in {\rm Con}(A)$, then, for all $\alpha \in {\cal B}({\rm Con}(A))$ and all $\theta \in {\rm Con}(A)$, $[\alpha ,\theta ]_A=\alpha \cap \theta $.\label{meetcomm}\end{lemma}

\begin{proof} Let $\theta \in {\rm Con}(A)$ and $\alpha \in {\cal B}({\rm Con}(A))$, so that there exists a $\beta \in {\rm Con}(A)$ with $\alpha \vee \beta =\nabla _A$ and $\alpha \cap \beta =\Delta _A$. Then the following hold: $[\alpha ,\theta ]_A\subseteq \alpha \cap \theta =[\nabla _A,\alpha \cap \theta]_A=[\alpha \vee \beta ,\alpha \cap \theta]_A=[\alpha ,\alpha \cap \theta]_A\vee [\beta ,\alpha \cap \theta]_A\subseteq [\alpha ,\alpha \cap \theta]_A\vee (\beta \cap \alpha \cap \theta )\subseteq [\alpha ,\theta ]_A\vee \Delta _A=[\alpha ,\theta ]_A$, hence $[\alpha ,\theta ]_A=\alpha \cap \theta $. We have followed the argument from \cite[Lemma $4$]{jip}.\end{proof}

\begin{remark} By Lemma \ref{meetcomm}, if $[\gamma ,\nabla _A]_A=\gamma $ for all $\gamma \in {\rm Con}(A)$, then, in ${\cal B}({\rm Con}(A))$, the commutator of $A$ equals the intersection, in particular the intersection in ${\cal B}({\rm Con}(A))$ is distributive with respect to the join.\label{comutinters}\end{remark}

\begin{lemma}\begin{enumerate}
\item\label{fsurjk1} If $f$ is surjective, then:\begin{itemize}

\item $f({\rm PCon}(A)\cap [{\rm Ker}(f)))\subseteq f(\{\alpha \vee {\rm Ker}(f)\ |\ \alpha \in {\rm PCon}(A)\})={\rm PCon}(B)$; 
\item $f({\cal K}(A)\cap [{\rm Ker}(f)))\subseteq f(\{\alpha \vee {\rm Ker}(f)\ |\ \alpha \in {\cal K}(A)\})={\cal K}(B)$;
\item if ${\cal C}$ is congruence--modular and semi--degenerate, then $f({\cal B}({\rm Con}(A))\cap [{\rm Ker}(f)))\subseteq f(\{\alpha \vee {\rm Ker}(f)\ |\ \alpha \in {\cal B}({\rm Con}(A))\})\subseteq {\cal B}({\rm Con}(B))$.\end{itemize}
\item\label{fsurjk2} For all $\theta \in {\rm Con}(A)$:\begin{itemize}
\item $\{\alpha /\theta \ |\ \alpha \in {\rm PCon}(A)\cap [\theta )\}\subseteq \{(\alpha \vee \theta )/\theta \ |\ \alpha \in {\rm PCon}(A)\})={\rm PCon}(A/\theta )$; 
\item $\{\alpha /\theta \ |\ \alpha \in {\cal K}(A)\cap [\theta )\}\subseteq \{(\alpha \vee \theta )/\theta \ |\ \alpha \in {\cal K}(A)\})={\cal K}(A/\theta )$;
\item if ${\cal C}$ is congruence--modular and semi--degenerate, then $\{\alpha /\theta \ |\ \alpha \in {\cal B}({\rm Con}(A))\cap [\theta )\}\subseteq \{(\alpha \vee \theta )/\theta \ |\ \alpha \in {\cal B}({\rm Con}(A))\})\subseteq {\cal B}({\rm Con}(A/\theta ))$.\end{itemize}\end{enumerate}\label{fsurjk}\end{lemma}

\begin{proof} The first inclusion in each statement is trivial.

\noindent (\ref{fsurjk1}) By Lemma \ref{fsurjcongr}, (\ref{fsurjcongr1}), for the statements on principal and on compact congruences. Now let $\alpha \in {\cal B}({\rm Con}(A))$, so that $\alpha \vee \beta =\nabla _A$ and $[\alpha ,\beta ]_A=\Delta _A$ for some $\beta \in {\rm Con}(A)$, hence, by Lemma \ref{fsurjcongr}, (\ref{fsurjcongr0}), and Remark \ref{fsurjcomut}, $f(\alpha \vee {\rm Ker}(f))\vee f(\beta \vee {\rm Ker}(f))=f(\alpha \vee {\rm Ker}(f)\vee \beta \vee {\rm Ker}(f))=f(\nabla _A)=\nabla _B$ and $[f(\alpha \vee {\rm Ker}(f)),f(\beta \vee {\rm Ker}(f)]_B=f([\alpha ,\beta ]_A\vee {\rm Ker}(f))=f(\Delta _A)=\Delta _B$, therefore $f(\alpha \vee {\rm Ker}(f))\in {\cal B}({\rm Con}(B))$. 

\noindent (\ref{fsurjk2}) By (\ref{fsurjk1}) for $f=p_{\theta }$.\end{proof}

\begin{proposition}\begin{enumerate}
\item\label{nablak1} Assume that $f$ is surjective. Then: if $\nabla _A\in {\rm PCon}(A)$, then $\nabla _B\in {\rm PCon}(B)$, while, if $\nabla _A\in {\cal K}(A)$, then $\nabla _B\in {\cal K}(B)$.
\item\label{nablak2} $\nabla _A\in {\rm PCon}(A)$ iff $\nabla _{A/\theta }\in {\rm PCon}(A/\theta )$ for all $\theta \in {\rm Con}(A)$. $\nabla _A\in {\cal K}(A)$ iff $\nabla _{A/\theta }\in {\cal K}(A/\theta )$ for all $\theta \in {\rm Con}(A)$.\end{enumerate}\label{nablak}\end{proposition}

\begin{proof} (\ref{nablak1}) By Lemma \ref{fsurjk}, (\ref{fsurjk1}).

\noindent (\ref{nablak2}) By (\ref{nablak1}) for the direct implications, and the fact that $A/\Delta _A$ is isomorphic to $A$, for the converse implications.\end{proof}

\begin{lemma} If ${\cal C}$ is congruence--modular, then, for all $n\in \N ^*$ and any $\alpha ,\beta \in {\rm Con}(A)$:\begin{enumerate}
\item\label{2.10(1)} if $f$ is surjective, then $[f(\alpha \vee {\rm Ker}(f)),f(\beta \vee {\rm Ker}(f))]_B^n=f([\alpha ,\beta ]_A^n\vee {\rm Ker}(f))$;
\item\label{2.10(2)} for any $\theta \in {\rm Con}(A)$, $[(\alpha \vee \theta )/\theta ,(\beta \vee \theta )/\theta ]_{A/\theta }^n=([\alpha ,\beta ]_A^n\vee \theta )/\theta  $;
\item\label{2.10(3)} for any $\theta \in {\rm Con}(A)$ and any $X,Y\in {\cal P}(A^2)$, $[Cg_{A/\theta }(X/\theta ),Cg_{A/\theta }(Y/\theta )]_{A/\theta }^n=([Cg_A(X),Cg_A(Y)]_A^n\vee \theta )/\theta  $.\end{enumerate}\label{2.10}\end{lemma}

\begin{proof} (\ref{2.10(1)}) We proceed by induction on $n$. For $n=1$, this holds by Remark \ref{fsurjcomut}. Now take an $n\in \N ^*$ such that $[f(\alpha \vee {\rm Ker}(f)),f(\beta \vee {\rm Ker}(f))]_B^n=f([\alpha ,\beta ]_A^n\vee {\rm Ker}(f))$. Then, by the induction hypothesis and Remark \ref{fsurjcomut}, $[f(\alpha \vee {\rm Ker}(f)),f(\beta \vee {\rm Ker}(f))]_B^{n+1}=[[f(\alpha \vee {\rm Ker}(f)),f(\beta \vee {\rm Ker}(f))]_B^n,[f(\alpha \vee {\rm Ker}(f)),f(\beta \vee {\rm Ker}(f))]_B^n]_B=[f([\alpha ,\beta ]_A^n\vee {\rm Ker}(f)),f([\alpha ,\beta ]_A^n\vee {\rm Ker}(f))]_B=f([[\alpha ,\beta ]_A^n, [\alpha ,\beta ]_A^n]_A^n\vee {\rm Ker}(f))=f([\alpha ,\beta ]_A^{n+1}\vee {\rm Ker}(f))$.

\noindent (\ref{2.10(2)}) Take $f=p_{\theta }$ in (\ref{2.10(1)}).

\noindent (\ref{2.10(3)}) Take $\alpha =Cg_A(X)$ and $\beta =Cg_A(Y)$ in (\ref{2.10(2)}) and apply Lemma \ref{fsurjcongr}, (\ref{fsurjcongr2}).\end{proof}

\section{Boolean Congruences versus the Reticulation}
\label{boolean}

Throughout this section, we shall assume that $[\cdot ,\cdot ]_A$ is commutative and distributive w.r.t. arbitrary joins and $\nabla _A\in {\cal K}(A)$. We call $A$ a {\em semiprime algebra} iff $\rho _A(\Delta _A)=\Delta _A$. So $A$ is semiprime iff $\Delta _A\in {\rm RCon}(A)$.

\begin{remark} By Proposition \ref{radcongrdistrib}, if the commutator of $A$ equals the intersection, then $A$ is semiprime, hence, if ${\cal C}$ is congruence--distributive, then every member of ${\cal C}$ is semiprime.\label{allsemiprime}\end{remark}

\begin{proposition} $A/\rho _A(\Delta _A)$ is semiprime.\label{rhosemiprime}\end{proposition}

\begin{proof} By Proposition \ref{esential}, (\ref{esential4}), and Proposition \ref{onrho}, (\ref{onrho3}), $\rho _{A/\rho _A(\Delta _A)}(\Delta _{A/\rho _A(\Delta _A)})=\rho _A(\rho _A(\Delta _A))/\rho _A(\Delta _A)=\rho _A(\Delta _A)/\rho _A(\Delta _A)=\Delta _{A/\rho _A(\Delta _A)}$.\end{proof}

\begin{lemma} If $A$ is semiprime, then, for all $\alpha ,\beta \in {\rm Con}(A)$:\begin{itemize}
\item $\lambda _A(\alpha )={\bf 0}$ iff $\alpha =\Delta _A$;
\item $[\alpha ,\beta ]_A=\Delta _A$ iff $\alpha \cap \beta =\Delta _A$.\end{itemize}\label{deltasemiprime}\end{lemma}

\begin{proof} Let $\alpha ,\beta \in {\rm Con}(A)$. Since $\lambda _A(\Delta _A)={\bf 0}$ and $[\alpha ,\beta ]_A\subseteq \alpha \cap \beta $, the converse implications always hold. Now assume that $A$ is semiprime. If $\lambda _A(\alpha )={\bf 0}=\lambda _A(\Delta _A)$, then $\alpha \subseteq \rho _A(\alpha )=\rho _A(\Delta _A)=\Delta _A$, thus $\alpha =\Delta _A$. If $[\alpha ,\beta ]_A=\Delta _A$, then $\lambda _A(\alpha \cap \beta )=\lambda _A([\alpha ,\beta ]_A)=\lambda _A(\Delta _A)={\bf 0}$, hence $\alpha \cap \beta =\Delta _A$ by the above.\end{proof}

\begin{lemma} For any $\theta \in {\rm Con}(A)$, the following hold:\begin{enumerate}
\item\label{lema4.2(1)} $\displaystyle \rho _A(\theta )=\bigvee \{\alpha \in {\rm Con}(A)\ |\ (\exists \, k\in \N ^*)\, ([\alpha ,\alpha ]_A^k\subseteq \theta )\}=\bigvee \{\alpha \in {\cal K}(A)\ |\ (\exists \, k\in \N ^*)\, ([\alpha ,\alpha ]_A^k\subseteq \theta )\}=\bigvee \{\alpha \in {\rm PCon}(A)\ |\ (\exists \, k\in \N ^*)\, ([\alpha ,\alpha ]_A^k\subseteq \theta )\}$;
\item\label{lema4.2(2)} for any $\alpha \in {\cal K}(A)$, $\alpha \subseteq \rho _A(\theta )$ iff there exists a $k\in \N ^*$ such that $[\alpha ,\alpha ]_A^k\subseteq \theta $.\end{enumerate}\label{lema4.2}\end{lemma}

\begin{proof} (\ref{lema4.2(1)}) By Proposition \ref{rho} and the fact that ${\rm PCon}(A)\subseteq {\cal K}(A)\subseteq {\rm Con}(A)$, $\displaystyle \rho _A(\theta )=\bigvee \{Cg_A(a,b)\ |\ (a,b)\in A^2,(\exists \, k\in \N ^*)\, ([Cg_A(a,b),Cg_A(a,b)]_A^k\subseteq \theta )\}=\bigvee \{\alpha \in {\rm PCon}(A)\ |\ (\exists \, k\in \N ^*)\, ([\alpha ,\alpha ]_A^k\subseteq \theta )\}\subseteq \bigvee \{\alpha \in {\cal K}(A)\ |\ (\exists \, k\in \N ^*)\, ([\alpha ,\alpha ]_A^k\subseteq \theta )\}\subseteq \bigvee \{\alpha \in {\rm Con}(A)\ |\ (\exists \, k\in \N ^*)\, ([\alpha ,\alpha ]_A^k\subseteq \theta )\}\subseteq \bigvee \{\alpha \in {\rm PCon}(A)\ |\ (\exists \, k\in \N ^*)\, ([\alpha ,\alpha ]_A^k\subseteq \theta )\}$, where the last inclusion holds because, for any $\alpha \in {\rm Con}(A)$, if $k\in \N ^*$ is such that $[\alpha ,\alpha ]_A^k\subseteq \theta $, then, for any $(a,b)\in \alpha $, $[Cg_A(a,b),Cg_A(a,b)]_A^k\subseteq [\alpha ,\alpha ]_A^k\subseteq \theta $, thus $\displaystyle \alpha =\bigvee _{(a,b)\in \alpha }Cg_A(a,b)\subseteq \bigvee \{Cg_A(a,b)\ |\ (a,b)\in A^2,(\exists \, k\in \N ^*)\, ([Cg_A(a,b),Cg_A(a,b)]_A^k\subseteq \theta )\}$. Hence the equalities in the enunciation.

\noindent (\ref{lema4.2(2)}) The converse implication follows directly from (\ref{lema4.2(1)}).

For the direct implication, from (\ref{lema4.2(1)}) it follows that, for any $\alpha \in {\cal K}(A)$ such that $\alpha \subseteq \rho _A(\theta )$, there exist non--empty families $(\beta _j)_{j\in J}\subseteq {\cal K}(A)$ and $(k_j)_{j\in J}\subseteq \N ^*$ such that $\displaystyle \alpha \subseteq \bigvee _{j\in J}\beta _j$ and $[\beta _j,\beta _j]_A^{k_j}\subseteq \theta $ for all $j\in J$. Since $\alpha \in {\cal K}(A)$, it follows that there exist an $n\in \N ^*$ and $j_1,\ldots ,j_n\in J$ such that $\displaystyle \alpha \subseteq \bigvee _{i=1}^n\beta _{j_i}$. Let $j=\max \{j_1,\ldots ,j_n\}\in \N ^*$. Then $[\beta _{j_i},\beta _{j_i}]_A^k\subseteq [\beta _{j_i},\beta _{j_i}]_A^{k_i}\subseteq \theta $ for each $i\in \overline{1,n}$, thus, by Lemma \ref{aritmcomut}, (\ref{aritmcomut7}), $\displaystyle [\alpha ,\alpha ]_A^{k^n}\subseteq [\bigvee _{i=1}^n\beta _{j_i},\bigvee _{i=1}^n\beta _{j_i}]_A^{k^n}\subseteq \bigvee _{i=1}^n[\beta _{j_i},\beta _{j_i}]_A^k\subseteq \theta $.\end{proof}

\begin{proposition}\begin{enumerate}
\item\label{prop4.5(1)} $\displaystyle \rho _A(\Delta _A)=\bigvee \{\alpha \in {\rm Con}(A)\ |\ (\exists \, k\in \N ^*)\, ([\alpha ,\alpha ]_A^k=\Delta _A)\}=\bigvee \{\alpha \in {\cal K}(A)\ |\ (\exists \, k\in \N ^*)\, ([\alpha ,\alpha ]_A^k=\Delta _A)\}=\bigvee \{\alpha \in {\rm PCon}(A)\ |\ (\exists \, k\in \N ^*)\, ([\alpha ,\alpha ]_A^k=\Delta _A)\}$;
\item\label{prop4.5(2)} for any $\alpha \in {\cal K}(A)$, $\alpha \subseteq \rho _A(\Delta _A)$ iff there exists a $k\in \N ^*$ such that $[\alpha ,\alpha ]_A^k=\Delta _A$.\end{enumerate}\label{prop4.5}\end{proposition}

\begin{proof} By Lemma \ref{lema4.2}.\end{proof}

\begin{corollary} $A$ is semiprime iff, for any $\alpha \in {\cal K}(A)$ and any $k\in \N ^*$, if $[\alpha ,\alpha ]_A^k=\Delta _A$, then $\alpha =\Delta _A$.\end{corollary}

Throughout the rest of this section, we shall assume that ${\cal K}(A)$ is closed w.r.t. the commutator of $A$ and $[\theta ,\nabla _A]_A=\theta $ for all $\theta \in {\rm Con}(A)$; see also Proposition \ref{prodcongr}.

For any bounded lattice $L$, we shall denote by ${\cal B}(L)$ the set of the complemented elements of $L$. If $L$ is distributive, then ${\cal B}(L)$ is the Boolean center of $L$. Although ${\rm Con}(A)$ is not necessarily distributive, we shall call ${\cal B}({\rm Con}(A))$ the {\em Boolean Center} of ${\rm Con}(A)$. So ${\cal B}({\rm Con}(A))$ is the set of the $\alpha \in {\rm Con}(A)$ such that there exists a $\beta \in {\rm Con}(A)$ which fulfills $\alpha \vee \beta =\nabla _A$ and $\alpha \cap \beta =\Delta _A$, thus also $[\alpha ,\beta ]_A=\Delta _A$.

\begin{remark} Obviously, $\Delta _A,\nabla _A\in {\cal B}({\rm Con}(A))$.\end{remark}

\begin{lemma} ${\cal B}({\rm Con}(A))\subseteq {\cal K}(A)$.\label{lema6.1}\end{lemma}

\begin{proof} Let $\alpha \in {\cal B}({\rm Con}(A))$, so that $\alpha \vee
\beta =\nabla _A$ and $\alpha \cap \beta =\Delta _A$ for some $\beta \in {\rm Con}(A)$. Now let $\emptyset \neq (\alpha _i)_{i\in I}\subseteq {\rm Con}(A)$ such that $\displaystyle \alpha \subseteq \bigvee _{i\in I}\alpha _i$, so that $\displaystyle \beta \vee \bigvee _{i\in I}\alpha _i=\nabla _A\in {\cal K}(A)$, thus $\displaystyle \nabla _A=\beta \vee \bigvee _{j=1}^n\alpha _{i_j}$ for some $n\in \N ^*$ and some $i_1,\ldots ,i_n\subseteq I$, hence, by Proposition \ref{prop4.1}, (\ref{prop4.1(1)}), $\displaystyle \alpha =[\alpha ,\nabla _A]_A=[\alpha ,\beta \vee \bigvee _{j=1}^n\alpha _{i_j}]_A=[\alpha ,\beta ]_A\vee [\alpha ,\bigvee _{j=1}^n\alpha _{i_j}]_A=\Delta _A\vee [\alpha ,\bigvee _{j=1}^n\alpha _{i_j}]_A=[\alpha ,\bigvee _{j=1}^n\alpha _{i_j}]_A\subseteq \bigvee _{j=1}^n\alpha _{i_j}$, hence $\alpha \in {\cal K}(A)$.\end{proof}

\begin{proposition} If ${\cal K}(A)={\cal B}({\rm Con}(A))$, then ${\cal L}(A)={\cal B}({\cal L}(A))$.\label{noua8.16(1)}\end{proposition}

\begin{proof} If ${\cal K}(A)={\cal B}({\rm Con}(A))$, then ${\cal L}(A)=\lambda _A({\cal K}(A))=\lambda _A({\cal B}({\rm Con}(A)))\subseteq {\cal B}({\cal L}(A))\subseteq {\cal L}(A)$ by Lemma \ref{lema6.2}, thus ${\cal L}(A)={\cal B}({\cal L}(A))$.\end{proof}

\begin{lemma} For any $\sigma ,\theta \in {\rm Con}(A)$: $\displaystyle \theta ^{\perp }=\bigvee \{\alpha \in {\rm PCon}(A)\ |\ [\alpha ,\theta ]_A=\Delta _A\}=\bigvee \{\alpha \in {\cal K}(A)\ |\ [\alpha ,\theta ]_A=\Delta _A\}=\bigvee \{\alpha \in {\rm Con}(A)\ |\ [\alpha ,\theta ]_A=\Delta _A\}=\max \{\alpha \in {\rm Con}(A)\ |\ [\alpha ,\theta ]_A=\Delta _A\}$, thus: $\sigma \subseteq \theta ^{\perp }$ iff $[\sigma ,\theta ]_A=\Delta _A$.\label{lema5.16}\end{lemma}

\begin{proof} Let $M=\{\alpha \in {\rm Con}(A)\ |\ [\alpha ,\theta ]_A=\Delta _A\}$. For all $\alpha \in M$ and all $(a,b)\in \alpha $, $[Cg_A(a,b),\theta ]_A\subseteq [\alpha ,\theta ]_A=\Delta _A$, thus $Cg_A(a,b)\in M\cap {\rm PCon}(A)$. Hence $\displaystyle \theta ^{\perp }=\bigvee _{\alpha \in M}\alpha =\bigvee _{\alpha \in M}\bigvee _{(a,b)\in \alpha }Cg_A(a,b)\subseteq \bigvee _{\gamma \in M\cap {\rm PCon}(A)}\gamma \subseteq \bigvee _{\gamma \in M\cap {\cal K}(A)}\gamma \subseteq \bigvee _{\alpha \in M}\alpha $, therefore $\displaystyle \theta ^{\perp }=\bigvee _{\alpha \in M}\alpha =\bigvee _{\alpha \in M\cap {\cal K}(A)}\alpha =\bigvee _{\alpha \in M\cap {\rm PCon}(A)}\alpha $. Note, also, that $\displaystyle [\theta ^{\perp },\theta ]_A=[\bigvee _{\alpha \in M}\alpha ,\theta ]_A=\bigvee _{\alpha \in M}[\alpha ,\theta ]_A=\bigvee _{\alpha \in M}\Delta _A=\Delta _A$, hence $\theta ^{\perp }\in M$, thus $\theta ^{\perp }=\max (M)$. If $\sigma \subseteq \theta ^{\perp }$, then $[\sigma ,\theta ]_A\subseteq [\theta ^{\perp },\theta ]_A=\Delta _A$, thus $[\sigma ,\theta ]_A=\Delta _A$, and conversely: if $[\sigma ,\theta ]_A=\Delta _A$, then $\sigma \in M$, thus $\sigma \subseteq \max (M)=\theta ^{\perp }$.\end{proof}

\begin{lemma} $\lambda _A({\cal B}({\rm Con}(A)))={\cal B}({\rm Con}(A))/_{\textstyle \equiv _A }\subseteq {\cal B}({\cal L}(A))\subseteq {\cal B}({\rm Con}(A)/_{\textstyle \equiv _A })$ and $\lambda _A\mid _{{\cal B}({\rm Con}(A))}:{\cal B}({\rm Con}(A))$\linebreak $\rightarrow {\cal B}({\cal L}(A))$ is a Boolean morphism.\label{lema6.2}\end{lemma}

\begin{proof} $\lambda _A({\cal B}({\rm Con}(A)))={\cal B}({\rm Con}(A))/_{\textstyle \equiv _A }$. Now we use Lemma \ref{lema6.1}. Let $\alpha \in {\cal B}({\rm Con}(A))\subseteq {\cal K}(A)$, so that $\lambda _A(\alpha )\in {\cal L}(A)$ and, for some $\beta \in {\cal B}({\rm Con}(A))\subseteq {\cal K}(A)$, we have $\alpha \vee \beta =\nabla _A$ and $\alpha \cap \beta =\Delta _A$. Then $\lambda _A(\beta )\in {\cal L}(A)$, ${\bf 1}=\lambda _A(\nabla _A)=\lambda _A(\alpha \vee \beta )=\lambda _A(\alpha )\vee \lambda _A(\beta )$ and ${\bf 0}=\lambda _A(\Delta _A)=\lambda _A(\alpha \cap \beta )=\lambda _A(\alpha )\wedge \lambda _A(\beta )$, hence $\lambda _A(\alpha )\in {\cal B}({\cal L}(A))$. Therefore $\lambda _A({\cal B}({\rm Con}(A)))\subseteq {\cal B}({\cal L}(A))$. Since ${\cal L}(A)$ is a bounded sublattice of the bounded distributive lattice ${\rm Con}(A)/_{\textstyle \equiv _A }$, it follows that ${\cal B}({\cal L}(A))$ is a Boolean subalgebra of ${\cal B}({\rm Con}(A)/_{\textstyle \equiv _A })$. Hence $\lambda _A({\cal B}({\rm Con}(A)))={\cal B}({\rm Con}(A))/_{\textstyle \equiv _A }\subseteq {\cal B}({\cal L}(A))\subseteq {\cal B}({\rm Con}(A)/_{\textstyle \equiv _A })$. $\lambda _A:{\rm Con}(A)\rightarrow {\rm Con}(A)/_{\textstyle \equiv _A }$ is a (surjective) bounded lattice morphism. Hence $\lambda _A\mid _{{\cal B}({\rm Con}(A))}:{\cal B}({\rm Con}(A))\rightarrow {\cal B}({\cal L}(A))$ is well defined and it is a bounded lattice morphism, thus it is a Boolean morphism.\end{proof}

Throughout the rest of this section, ${\cal C}$ shall be congruence--modular and semi--degenerate.

\begin{proposition}\begin{enumerate}
\item\label{prop6.3(1)} The Boolean morphism $\lambda _A\mid _{{\cal B}({\rm Con}(A))}:{\cal B}({\rm Con}(A))\rightarrow {\cal B}({\cal L}(A))$ is injective.
\item\label{prop6.3(2)} If the commutator of $A$ is associative, then $\lambda _A({\cal B}({\rm Con}(A)))={\cal B}({\cal L}(A))={\cal B}({\rm Con}(A))/_{\textstyle \equiv _A }\subseteq {\cal B}({\rm Con}(A)/_{\textstyle \equiv _A })$ and $\lambda _A\mid _{{\cal B}({\rm Con}(A))}:{\cal B}({\rm Con}(A))\rightarrow {\cal B}({\cal L}(A))$ is a Boolean isomorphism.
\item\label{prop6.3(3)} If $A$ is semiprime, then $\lambda _A({\cal B}({\rm Con}(A)))={\cal B}({\cal L}(A))={\cal B}({\rm Con}(A))/_{\textstyle \equiv _A }={\cal B}({\rm Con}(A)/_{\textstyle \equiv _A })$ and $\lambda _A\mid _{{\cal B}({\rm Con}(A))}:{\cal B}({\rm Con}(A))\rightarrow {\cal B}({\cal L}(A))$ is a Boolean isomorphism.\end{enumerate}\label{prop6.3}\end{proposition}

\begin{proof} (\ref{prop6.3(1)}) By Lemma \ref{lema6.2}, $\lambda _A\mid _{{\cal B}({\rm Con}(A))}:{\cal B}({\rm Con}(A))\rightarrow {\cal B}({\cal L}(A))$ is a Boolean morphism. By Remark \ref{coronrho}, $\lambda _A(\alpha )={\bf 1}$ iff $\alpha =\nabla _A$, hence this Boolean morphism is injective.

\noindent (\ref{prop6.3(2)}) Assume that $A$ is semiprime, and let $x\in {\cal B}({\rm Con}(A)/_{\textstyle \equiv _A })$, so that $x\vee y={\bf 1}$ and $x\wedge y={\bf 0}$ for some $y\in {\cal B}({\rm Con}(A)/_{\textstyle \equiv _A })$. Hence there exist $\alpha ,\beta \in {\rm Con}(A)$ such that $x=\lambda _A(\alpha )$ and $y=\lambda _A(\beta )$, thus ${\bf 1}=x\vee y=\lambda _A(\alpha )\vee \lambda _A(\beta )=\lambda _A(\alpha \vee \beta )$ and ${\bf 0}=x\wedge y=\lambda _A(\alpha )\wedge \lambda _A(\beta )=\lambda _A(\alpha \cap \beta )$, therefore $\alpha \vee \beta =\nabla _A$ and $\alpha \cap \beta =\Delta _A$, by Remark \ref{coronrho} and Lemma \ref{deltasemiprime}. Hence $\alpha \in {\cal B}({\rm Con}(A))$, thus $x=\lambda _A(\alpha )\in \lambda _A({\cal B}({\rm Con}(A)))={\cal B}({\rm Con}(A))/_{\textstyle \equiv _A }$, therefore, by Lemma \ref{lema6.2}, ${\cal B}({\rm Con}(A)/_{\textstyle \equiv _A })\subseteq \lambda _A({\cal B}({\rm Con}(A)))={\cal B}({\rm Con}(A))/_{\textstyle \equiv _A }\subseteq {\cal B}({\cal L}(A))\subseteq {\cal B}({\rm Con}(A)/_{\textstyle \equiv _A })$, hence $\lambda _A({\cal B}({\rm Con}(A)))={\cal B}({\rm Con}(A))/_{\textstyle \equiv _A }={\cal B}({\cal L}(A))={\cal B}({\rm Con}(A)/_{\textstyle \equiv _A })$. Therefore $\lambda _A\mid _{{\cal B}({\rm Con}(A))}:{\cal B}({\rm Con}(A))\rightarrow {\cal B}({\cal L}(A))$ is surjective, so, by (\ref{prop6.3(1)}), it is a Boolean isomorphism.

\noindent (\ref{prop6.3(3)}) Assume that the commutator of $A$ is associative, and let $x\in {\cal B}({\cal L}(A))\subseteq {\cal L}(A)=\lambda _A({\cal K}(A))$, so that $x\vee y={\bf 1}$ and $x\wedge y={\bf 0}$ for some $y\in {\cal B}({\cal L}(A))$ and there exist $\alpha ,\beta \in {\cal K}(A)$ such that $x=\lambda _A(\alpha )$ and $y=\lambda _A(\beta )$. Then $\lambda _A(\alpha \vee \beta )=\lambda _A(\alpha )\vee \lambda _A(\beta )=x\vee y={\bf 1}=\lambda _A(\nabla _A)$, hence $\alpha \vee \beta =\nabla _A$ by Remark \ref{coronrho}. We also have $\lambda _A([\alpha ,\beta ]_A)=\lambda _A(\alpha )\wedge \lambda _A(\beta )=x\wedge y={\bf 0}=\lambda _A(\Delta _A)$, thus $[\alpha ,\beta ]_A\subseteq \rho _A(\alpha \cap \beta )=\rho _A(\Delta _A)$, and, since ${\cal K}(A)$ is closed with respect to the commutator, we have $[\alpha ,\beta ]_A\in {\cal K}(A)$, thus, according to Proposition \ref{prop4.5}, (\ref{prop4.5(2)}), $[[\alpha ,\alpha ]_A^k,[\beta ,\beta ]_A^k]_A=[\alpha ,\beta ]_A^{k+1}=[[\alpha ,\beta ]_A,[\alpha ,\beta ]_A]_A^k=\Delta _A$ for some $k\in \N ^*$; we have applied Lemmas \ref{lema2.9} and \ref{commassoc}. But $\alpha \vee \beta =\nabla _A$, thus $[\alpha ,\alpha ]_A^k\vee [\beta ,\beta ]_A^k=\nabla _A$, hence $[\alpha ,\alpha ]_A^k\cap [\beta ,\beta ]_A^k=[[\alpha ,\alpha ]_A^k,[\beta ,\beta ]_A^k]_A=\Delta _A$ by Proposition \ref{prop4.1}, (\ref{prop4.1(3)}) and (\ref{prop4.1(1)}). Therefore $[\alpha ,\alpha ]_A^k\in {\cal B}({\rm Con}(A))$, thus $x=\lambda _A(\alpha )=\lambda _A([\alpha ,\alpha ]_A^k)\in \lambda _A({\cal B}({\rm Con}(A)))$, hence ${\cal B}({\cal L}(A))\subseteq \lambda _A({\cal B}({\rm Con}(A)))$, thus ${\cal B}({\cal L}(A))\subseteq \lambda _A({\cal B}({\rm Con}(A)))={\cal B}({\rm Con}(A))/_{\textstyle \equiv _A }\subseteq {\cal B}({\cal L}(A))\subseteq {\cal B}({\rm Con}(A)/_{\textstyle \equiv _A })$ by Lemma \ref{lema6.2}, therefore $\lambda _A({\cal B}({\rm Con}(A)))={\cal B}({\rm Con}(A))/_{\textstyle \equiv _A }={\cal B}({\cal L}(A))\subseteq {\cal B}({\rm Con}(A)/_{\textstyle \equiv _A })$. Therefore $\lambda _A\mid _{{\cal B}({\rm Con}(A))}:{\cal B}({\rm Con}(A))\rightarrow {\cal B}({\cal L}(A))$ is surjective, so, by (\ref{prop6.3(1)}), it is a Boolean isomorphism.\end{proof}

\begin{lemma} If $A$ is semiprime and $\alpha \in {\rm Con}(A)$, then: $\alpha \in {\cal B}({\rm Con}(A))$ iff $\lambda _A(\alpha )\in {\cal B}({\cal L}(A))$.\label{lambdaboole}\end{lemma}

\begin{proof} We apply Lemma \ref{lema6.2}, which, first of all, gives us the direct implication. For the converse, assume that $\lambda _A(\alpha )\in {\cal B}({\cal L}(A))={\cal B}({\rm Con}(A)/_{\textstyle \equiv _A })$, so that there exists a $\beta \in {\rm Con}(A)$ with $\lambda _A(\alpha \vee \beta )=\lambda _A(\alpha )\vee \lambda _A(\beta )={\bf 1}=\lambda _A(\nabla _A)$ and $\lambda _A(\alpha \cap \beta )=\lambda _A(\alpha )\wedge \lambda _A(\beta )={\bf 0}$, thus $\alpha \vee \beta =\nabla _A$ and $\alpha \cap \beta =\Delta _A$ by Remark \ref{coronrho} and Lemma \ref{deltasemiprime}. Therefore $\alpha \in {\cal B}({\rm Con}(A))$.\end{proof}

For any $\Omega \subseteq {\rm Con}(A)$, let us consider the property:

\noindent $(A,\Omega )\quad $ for all $\alpha ,\beta \in \Omega $ and all $n\in \N ^*$, there exists a $k\in \N ^*$ such that $[[\alpha ,\alpha ]_A^k,[\beta ,\beta ]_A^k]_A\subseteq [\alpha ,\beta ]_A^n$

\begin{remark} By Lemma \ref{commassoc}, if the commutator of $A$ is associative, then $(A,{\rm Con}(A))$ holds.

Notice, from the proof of statement (\ref{prop6.3(3)}) from Proposition \ref{prop6.3}, that this statement, and thus the fact that $\lambda _A\mid _{{\cal B}({\rm Con}(A))}:{\cal B}({\rm Con}(A))\rightarrow {\cal B}({\cal L}(A))$ is a Boolean isomorphism, also hold if property $(A,{\cal K}(A))$ is fulfilled, instead of the associativity of the commutator of $A$.\end{remark}

\begin{openproblem} Under the current context, determine whether $(A,{\cal K}(A))$ always holds; if it doesn`t, then determine whether $(A,{\cal K}(A))$ is equivalent to the associativity of the commutator of $A$.\end{openproblem}

\begin{lemma} $({\cal B}({\rm Con}(A)),\vee ,[\cdot ,\cdot ]_A=\cap ,\perp ,\Delta _A,\nabla _A)$ is a Boolean algebra.\label{algboole}\end{lemma}

\begin{proof} We follow, in part, the argument from \cite[Lemma $4$]{jip}. Let $\alpha ,\beta \in {\cal B}({\rm Con}(A))$, so that there exist $\overline{\alpha },\overline{\beta }\in {\cal B}({\rm Con}(A))$ such that $\alpha \vee \overline{\alpha }=\beta \vee \overline{\beta }=\nabla _A$ and $\alpha \cap \overline{\alpha }=\beta \cap \overline{\beta }=\Delta _A$. Then, by Remark \ref{comutinters}, the following hold: $(\alpha \vee \beta )\cap \overline{\alpha }\cap \overline{\beta }=(\alpha \cap \overline{\alpha }\cap \overline{\beta })\vee (\beta \cap \overline{\alpha }\cap \overline{\beta })=\Delta _A\vee \Delta _A=\Delta _A$ and, since $\overline{\alpha }\cap \beta \subseteq \beta $, it follows that $\alpha \vee \beta \vee (\overline{\alpha }\cap \overline{\beta })=\alpha \vee \beta \vee (\overline{\alpha }\cap \beta )\vee (\overline{\alpha }\cap \overline{\beta })=\alpha \vee \beta \vee (\overline{\alpha }\cap (\beta \vee \overline{\beta }))=\alpha \vee \beta \vee (\overline{\alpha }\cap \nabla _A)=\alpha \vee \beta \vee \overline{\alpha }=\nabla _A$. Analogously, $(\overline{\alpha }\vee \overline{\beta })\cap \alpha \cap \beta =\Delta _A$ and $\overline{\alpha }\vee \overline{\beta }\vee (\alpha \cap \beta )=\nabla _A$. Hence $\alpha \vee \beta ,\alpha \cap \beta \in {\cal B}({\rm Con}(A))$. Clearly, $\Delta _A,\nabla _A\in {\cal B}({\rm Con}(A))$. Therefore ${\cal B}({\rm Con}(A))$ is a bounded sublattice of ${\rm Con}(A)$. By Remark \ref{comutinters}, it follows that $({\cal B}({\rm Con}(A)),\vee ,[\cdot ,\cdot ]_A=\cap ,\Delta _A,\nabla _A)$ is a bounded distributive lattice, and, by its definition, it is also complemented, thus it is a Boolean lattice. By a well--known characterization of the complement in a Boolean lattice, for any $\theta \in {\cal B}({\rm Con}(A))$, the complement of $\theta $ in ${\cal B}({\rm Con}(A))$ is $\overline{\theta }=\max \{\alpha \in {\cal B}({\rm Con}(A))\ |\ \alpha \cap \theta =\Delta _A\}=\max \{\alpha \in {\cal B}({\rm Con}(A))\ |\ [\alpha ,\theta ]_A=\Delta _A\}\subseteq \max \{\alpha \in {\rm Con}(A)\ |\ [\alpha ,\theta ]_A=\Delta _A\}=\theta ^{\perp }$ according to Lemma \ref{lema5.16}, thus $\nabla _A=\theta \vee \overline{\theta }\subseteq \theta \vee \theta ^{\perp }$, so $\theta \vee \theta ^{\perp }=\nabla _A$. Again by Lemma \ref{lema5.16}, $\Delta _A=[\theta ,\theta ^{\perp }]_A=\theta \cap \theta ^{\perp }$. Therefore $\theta ^{\perp }\in {\cal B}({\rm Con}(A))$ and $\theta ^{\perp }$ is the complement of $\theta $ in ${\cal B}({\rm Con}(A))$.\end{proof}

For any bounded lattice $L$ and any $I\in {\rm Id}(L)$, we shall denote by ${\rm Ann}(I)$ the {\em annihilator of $I$} in $L$: ${\rm Ann}(I)=\{a\in L\ |\ (\forall \, x\in I)\, (a\wedge x=0)\}$. It is immediate that, if $L$ is distributive, then ${\rm Ann}(I)\in {\rm Id}(L)$. Throughout the rest of this paper, all annihilators shall be considerred in the bounded distributive lattice ${\cal L}(A)$, so they shall be ideals of the lattice ${\cal L}(A)$. Recall that ${\cal L}(A)=\lambda _A({\cal K}(A))$.

\begin{lemma} For any $\alpha \in {\cal K}(A)$:\begin{itemize}
\item ${\rm Ann}(\alpha ^*)=\{\lambda _A(\beta )\ |\ \beta \in {\cal K}(A),\lambda _A([\alpha ,\beta ]_A)={\bf 0}\}$;
\item if $A$ is semiprime, then ${\rm Ann}(\alpha ^*)=\{\lambda _A(\beta )\ |\ \beta \in {\cal K}(A),[\alpha ,\beta ]_A=\Delta _A\}$.\end{itemize}\label{calcann}\end{lemma}

\begin{proof} By Lemma \ref{kstar}, ${\rm Ann}(\alpha ^*)={\rm Ann}((\lambda _A(\alpha )])=\{\lambda _A(\beta )\ |\ \beta \in {\cal K}(A),(\forall \, x\in (\lambda _A(\alpha )])\, (x\wedge \lambda _A(\beta )={\bf 0})\}=\{\lambda _A(\beta )\ |\ \beta \in {\cal K}(A),\lambda _A(\alpha )\wedge \lambda _A(\beta )={\bf 0})\}=\{\lambda _A(\beta )\ |\ \beta \in {\cal K}(A),\lambda _A([\alpha ,\beta ]_A)={\bf 0}\}$. By Lemma \ref{deltasemiprime}, if $A$ is semiprime, then, for any $\beta \in {\cal K}(A)$, $\lambda _A([\alpha ,\beta ]_A)={\bf 0}$ iff $[\alpha ,\beta ]_A=\Delta _A$, hence the second equality in the enunciation.\end{proof}

\begin{lemma} For any $\alpha \in {\rm Con}(A)$ and any $I\in {\rm Id}({\cal L}(A))$, if ${\rm Ann}(\alpha ^*)\subseteq I$, then $\alpha ^{\perp }\subseteq I_*$. If $A$ is semiprime and $\alpha \in {\cal K}(A)$, then the converse implication holds, as well.\label{lema5.17}\end{lemma}

\begin{proof} For the direct implication, assume that ${\rm Ann}(\alpha ^*)\subseteq I$ and let $\beta \in {\cal K}(A)$ such that $[\alpha ,\beta ]_A=\Delta _A$, hence $\lambda _A(\alpha )\wedge \lambda _A(\beta )=\lambda _A([\alpha ,\beta ]_A)=\lambda _A(\Delta _A)={\bf 0}$. Now let $x\in \alpha ^*$, so that $x=\lambda _A(\gamma )$ for some $\gamma \in {\cal K}(A)$ with $\gamma \subseteq \alpha $. Then $x=\lambda _A(\gamma )\leq \lambda _A(\alpha )$, hence $x\wedge \lambda _A(\beta )=\lambda _A(\gamma )\wedge \lambda _A(\beta )\leq \lambda _A(\alpha )\wedge \lambda _A(\beta )={\bf 0}$, so $x\wedge \lambda _A(\beta )={\bf 0}$, thus $\lambda _A(\beta )\in {\rm Ann}(\alpha ^*)\subseteq I$, therefore $\beta \subseteq I_*$ by Lemma \ref{lema5}. According to Lemma \ref{lema5.16}, $\displaystyle \alpha ^{\perp }=\bigvee \{\beta \in {\cal K}(A)\ |\ [\alpha ,\beta ]_A=\Delta _A\}\subseteq I_*$.

For the converse implication, assume that $A$ is semiprime, $\alpha \in {\cal K}(A)$ and $\alpha ^{\perp }\subseteq I_*$, and let $x\in {\rm Ann}(\alpha ^*)$, which means that $x=\lambda _A(\beta )$ for some $\beta \in {\cal K}(A)$ with $[\alpha ,\beta ]_A=\Delta _A$, according to Lemma \ref{calcann}. Hence, by Lemmas \ref{lema5.16} and \ref{lema5}, $\beta \subseteq \alpha ^{\perp }\subseteq I_*$, thus $x=\lambda _A(\beta )\in I$, therefore ${\rm Ann}(\alpha ^{\perp })\subseteq I$.\end{proof}

\begin{proposition} For any $\theta \in {\rm Con}(A)$:\begin{enumerate}
\item\label{prop5.18(1)} $(\theta ^{\perp })^*\subseteq {\rm Ann}(\theta ^*)$;
\item\label{prop5.18(2)} if $A$ is semiprime, then $(\theta ^{\perp })^*={\rm Ann}(\theta ^*)$.\end{enumerate}\label{prop5.18}\end{proposition}

\begin{proof} $(\theta ^{\perp })^*=\{\lambda _A(\alpha )\ |\ \alpha \in {\cal K}(A),\alpha \subseteq \theta ^{\perp }\}=\{\lambda _A(\alpha )\ |\ \alpha \in {\cal K}(A),[\alpha ,\theta ]_A=\Delta _A\}$, by Lemma \ref{lema5.16}. ${\rm Ann}(\theta ^*)=\{\lambda _A(\alpha )\ |\ \alpha \in {\cal K}(A),(\forall \, x\in \theta ^*)\, (\lambda _A(\alpha )\wedge x=\lambda _A(\Delta _A))\}=\{\lambda _A(\alpha )\ |\ \alpha \in {\cal K}(A),(\forall \, \beta \in {\cal K}(A))\, (\beta \subseteq \theta \Rightarrow \lambda _A([\alpha ,\beta ]_A)=\lambda _A(\alpha )\wedge \lambda _A(\beta )=\lambda _A(\Delta _A))\}=\{\lambda _A(\alpha )\ |\ \alpha \in {\cal K}(A),(\forall \, \beta \in {\cal K}(A))\, (\beta \subseteq \theta \Rightarrow \rho _A([\alpha ,\beta ]_A)=\rho _A(\Delta _A))\}$.

\noindent (\ref{prop5.18(1)}) Let $\alpha \in {\cal K}(A)$ such that $\lambda _A(\alpha )\in (\theta ^{\perp })^*$, which means that $[\alpha ,\theta ]_A=\Delta _A$. Then, for any $\beta \in {\cal K}(A)$ fulfilling $\beta \subseteq \theta $, we have $[\alpha ,\beta ]_A\subseteq [\alpha ,\theta ]_A=\Delta _A$, so $[\alpha ,\beta ]_A=\Delta _A$, thus $\rho _A([\alpha ,\beta ]_A)=\rho _A(\Delta _A))$, hence $\lambda _A(\alpha )\in {\rm Ann}(\theta ^*)$. Therefore $(\theta ^{\perp })^*\subseteq {\rm Ann}(\theta ^*)$.

\noindent (\ref{prop5.18(2)}) Assume that $A$ is semiprime and $\alpha \in {\cal K}(A)$ such that $\lambda _A(\alpha )\in {\rm Ann}(\theta ^*)$, which means that, for all $\beta \in {\cal K}(A)$ such that $\beta \subseteq \theta $, $[\alpha ,\beta ]_A\subseteq \rho _A([\alpha ,\beta ]_A)=\rho _A(\Delta _A)=\Delta _A$, so $[\alpha ,\beta ]_A=\Delta _A$. $\displaystyle \theta =\bigvee _{(a,b)\in \theta }Cg_A(a,b)\subseteq \bigvee \{\beta \in {\cal K}(A)\ |\ \beta \subseteq \theta \}\subseteq \theta $, thus $\displaystyle \theta =\bigvee \{\beta \in {\cal K}(A)\ |\ \beta \subseteq \theta \}$, so $\displaystyle [\alpha ,\theta ]_A=[\alpha ,\bigvee \{\beta \in {\cal K}(A)\ |\ \beta \subseteq \theta \}]_A=\bigvee \{[\alpha ,\beta ]_A\ |\ \beta \in {\cal K}(A),\beta \subseteq \theta \}=\bigvee \{\Delta _A\ |\ \beta \in {\cal K}(A),\beta \subseteq \theta \}=\bigvee \{\Delta _A\}=\Delta _A$, therefore $\lambda _A(\alpha )\in (\theta ^{\perp })^*$, hence ${\rm Ann}(\theta ^*)\subseteq (\theta ^{\perp })^*$, thus ${\rm Ann}(\theta ^*)=(\theta ^{\perp })^*$ by (\ref{prop5.18(1)}).\end{proof}

\begin{proposition} For any $I\in {\rm Id}({\cal L}(A))$:\begin{enumerate}
\item\label{prop5.19(1)} $(I_*)^{\perp }\subseteq {\rm Ann}(I)_*$;
\item\label{prop5.19(2)} if $A$ is semiprime, then $(I_*)^{\perp }={\rm Ann}(I)_*$.\end{enumerate}\label{prop5.19}\end{proposition}

\begin{proof} $(I_*)^{\perp }=\bigvee \{\alpha \in {\cal K}(A)\ |\ [\alpha ,I_*]_A=\Delta _A\}=\bigvee \{\alpha \in {\cal K}(A)\ |\ [\alpha ,\bigvee \{\beta \in {\cal K}(A)\ |\ \lambda _A(\beta )\in I\}]_A=\Delta _A\}=\bigvee \{\alpha \in {\cal K}(A)\ |\ \bigvee \{[\alpha ,\beta ]_A\in {\cal K}(A)\ |\ \beta \in {\cal K}(A),\lambda _A(\beta )\in I\}=\Delta _A\}=\bigvee \{\alpha \in {\cal K}(A)\ |\ (\forall \, \beta \in {\cal K}(A))\, (\lambda _A(\beta )\in I\Rightarrow [\alpha ,\beta ]_A=\Delta _A)\}$. $({\rm Ann}(I))_*=\bigvee \{\alpha \in {\cal K}(A)\ |\ \lambda _A(\alpha )\in {\rm Ann}(I)\}=\bigvee \{\alpha \in {\cal K}(A)\ |\ (\forall \, \beta \in {\cal K}(A))\, (\lambda _A(\beta )\in I\Rightarrow \lambda _A(\alpha )\wedge \lambda _A(\beta )={\bf 0})\}=\bigvee \{\alpha \in {\cal K}(A)\ |\ (\forall \, \beta \in {\cal K}(A))\, (\lambda _A(\beta )\in I\Rightarrow \lambda _A([\alpha ,\beta ]_A)=\lambda _A(\Delta _A))\}$.

\noindent (\ref{prop5.19(1)}) For all $\alpha ,\beta \in {\rm Con}(A)$, if $[\alpha ,\beta ]_A=\Delta _A$, then $\lambda _A([\alpha ,\beta ]_A)=\lambda _A(\Delta _A)$, hence $(I_*)^{\perp }\subseteq {\rm Ann}(I)_*$.

\noindent (\ref{prop5.19(2)}) If $A$ is semiprime, then, for every $\alpha ,\beta \in {\rm Con}(A)$, $\lambda _A([\alpha ,\beta ]_A)=\lambda _A(\Delta _A)$ implies $[\alpha ,\beta ]_A\subseteq \rho _A([\alpha ,\beta ]_A)=\rho _A(\Delta _A)=\Delta _A$, thus $[\alpha ,\beta ]_A=\Delta _A$, hence ${\rm Ann}(I)_*\subseteq (I_*)^{\perp }$. By (\ref{prop5.19(1)}), it follows that $(I_*)^{\perp }={\rm Ann}(I)_*$.\end{proof}

We call $A$ a {\em hyperarchimedean algebra} iff, for all $\alpha \in {\rm PCon}(A)$, there exists an $n\in \N ^*$ such that $[\alpha ,\alpha ]_A^n\in {\cal B}({\rm Con}(A))$.

\begin{remark} If $\alpha \in {\rm Con}(A)$ and $n\in \N ^*$ are such that $[\alpha ,\alpha ]_A^n\in {\cal B}({\rm Con}(A))$, then, by Remark \ref{comutinters}, $[\alpha ,\alpha ]_A^{n+1}=[[\alpha ,\alpha ]_A^n,[\alpha ,\alpha ]_A^n]_A=[\alpha ,\alpha ]_A^n\cap [\alpha ,\alpha ]_A^n=[\alpha ,\alpha ]_A^n$, thus $[\alpha ,\alpha ]_A^k=[\alpha ,\alpha ]_A^n$ for all $k\in \N $ such that $k\geq n$.\end{remark}

\begin{remark} If $[\alpha ,\alpha ]_A\in {\cal B}({\rm Con}(A))$ for all $\alpha \in {\rm PCon}(A)$, then $A$ is hyperarchimedean. Thus, if ${\rm PCon}(A)\subseteq {\cal B}({\rm Con}(A))$ and $A$ has principal commutators, then $A$ is hyperarchimedean. If the commutator of $A$ equals the intersection, for instance if ${\cal C}$ is congruence--distributive, then: $A$ is hyperarchimedean iff ${\rm PCon}(A)\subseteq {\cal B}({\rm Con}(A))$. By Lemmas \ref{algboole} and \ref{lema6.1}, the following equivalences hold: ${\rm PCon}(A)\subseteq {\cal B}({\rm Con}(A))$ iff ${\cal K}(A)\subseteq {\cal B}({\rm Con}(A))$ iff ${\cal K}(A)={\cal B}({\rm Con}(A))$.\label{ahyp}\end{remark}

\begin{remark} By Lemma \ref{algboole}, the lattice ${\rm Con}(A)$ is Boolean iff ${\rm Con}(A)={\cal B}({\rm Con}(A))$, which implies that the commutator of $A$ equals the intersection, according to Remark \ref{comutinters}, and thus, since ${\rm PCon}(A)\subseteq {\rm Con}(A)={\cal B}({\rm Con}(A))$, $A$ is hyperarchimedean, while Remark \ref{allsemiprime} ensures us that $A$ is semiprime. From Lemma \ref{lema6.1}, we obtain the following equivalences: ${\rm Con}(A)$ is a Boolean lattice iff ${\cal B}({\rm Con}(A))={\rm Con}(A)$ iff ${\cal B}({\rm Con}(A))={\cal K}(A)={\rm Con}(A)$. Of course, since ${\cal L}(A)$ is a bounded distributive lattice, ${\cal L}(A)$ is a Boolean algebra iff ${\cal L}(A)={\cal B}({\cal L}(A))$.\label{congrbool}\end{remark}

\begin{proposition}\begin{enumerate}
\item\label{8.16(2)} If $A$ is semiprime, then: ${\cal K}(A)={\cal B}({\rm Con}(A))$ iff ${\cal L}(A)={\cal B}({\cal L}(A))$.
\item\label{8.16(3)} If ${\rm Con}(A)$ is a Boolean lattice, then $A$ is hyperarchimedean and semiprime and ${\cal L}(A)$ is isomorphic to ${\rm Con}(A)$, in particular ${\cal L}(A)$ is a Boolean lattice, as well.\end{enumerate}\label{8.16}\end{proposition}

\begin{proof} (\ref{8.16(2)}) The direct implication is Proposition \ref{noua8.16(1)}. For the converse, let $\alpha \in {\cal K}(A)$, so that $\lambda _A(\alpha )\in {\cal L}(A)={\cal B}({\cal L}(A))$, thus $\alpha \in {\cal B}({\rm Con}(A))$ by Lemma \ref{lambdaboole}. Hence ${\cal K}(A)\subseteq {\cal B}({\rm Con}(A))$, thus ${\cal K}(A)={\cal B}({\rm Con}(A))$ by Lemma \ref{lema6.1}.

\noindent (\ref{8.16(3)}) By Lemma \ref{congrbool}, we obtain that $A$ is hyperarchimedean and semiprime, and ${\cal B}({\rm Con}(A))={\cal K}(A)={\rm Con}(A)$, hence ${\cal L}(A)={\cal B}({\cal L}(A))$ by Proposition \ref{noua8.16(1)}, and thus $\lambda _A:{\rm Con}(A)={\cal B}({\rm Con}(A))\rightarrow {\cal B}({\cal L}(A))={\cal L}(A)$ is a Boolean isomorphism, according to Proposition \ref{prop6.3}.\end{proof}

\begin{lemma} If $A$ is hyperarchimedean, then $A/\theta $ is hyperarchimedean for all $\theta \in {\rm Con}(A)$.\label{quohyp}\end{lemma}

\begin{proof} Let $\theta \in {\rm Con}(A)$. For any $a,b\in A$, there exists an $n\in \N ^*$ such that $[Cg_A(a,b),Cg_A(a,b)]_A^n\in {\cal B}({\rm Con}(A))$. Then, according to Lemma \ref{2.10}, (\ref{2.10(3)}), and Lemma \ref{fsurjk}, (\ref{fsurjk2}), $[Cg_{A/\theta }(a/\theta ,b/\theta ),Cg_{A/\theta }(a/\theta ,b/\theta )]_{A/\theta }^n=([Cg_A(a,b),Cg_A(a,b)]_A^n\vee \theta )/\theta \in {\cal B}({\rm Con}(A/\theta ))$, therefore $A/\theta $ is hyperarchimedean.\end{proof}

\begin{lemma} If $A$ is hyperarchimedean, then ${\cal L}(A)$ is a Boolean lattice.\label{hypretic}\end{lemma}

\begin{proof} Let $\theta \in {\cal K}(A)$, so that $\theta =\alpha _1\vee \ldots \vee \alpha _n$ for some $n\in \N ^*$ and $\alpha _1,\ldots ,\alpha _n\in {\rm PCon}(A)$. Since $A$ is hyperarchimedean, there exists a $k\in \N ^*$ such that, for all $i\in \overline{1,n}$, $[\alpha _i,\alpha _i]_A^k\in {\cal B}({\rm Con}(A))$, thus $\lambda _A(\alpha _i)=\lambda _A([\alpha _i,\alpha _i]_A^k)\in \lambda _A({\cal B}({\rm Con}(A)))\subseteq {\cal B}({\cal L}(A))$ by Lemma \ref{lema6.2}, so that $\lambda _A(\theta )=\lambda _A(\alpha _1)\vee \ldots \vee \lambda _A(\alpha _n)\in {\cal B}({\cal L}(A))$. Hence $\lambda _A({\cal K}(A))={\cal L}(A)\subseteq {\cal B}({\cal L}(A))$, thus ${\cal L}(A)={\cal B}({\cal L}(A))$, so ${\cal L}(A)$ is a Boolean lattice.\end{proof}

\section{A Reticulation Functor}
\label{functor}

Throughout this section, ${\cal C}$ shall be congruence--modular and semi--degenerate and such that, in each of its members, the set of the compact congruences is closed w.r.t. the commutator. Also, the morphism $f:A\rightarrow B$ shall be surjective, so that the map $\varphi _f:{\rm Con}(A)\rightarrow {\rm Con}(B)$, $\varphi _f(\alpha )=f(\alpha \vee {\rm Ker}(f))$ for all $\alpha \in {\rm Con}(A)$, is well defined.

\begin{remark} By Lemma \ref{fsurjk}, (\ref{fsurjk1}), $\varphi _f({\cal K}(A))={\cal K}(B)$.

For any algebra $M$ from ${\cal C}$ and any $X\subseteq M^2$, let us denote $V_M(X)=V_M(Cg_M(X))$. Then, by the proof of \cite[Proposition $2.1$]{agl} and Lemma \ref{fsurjcongr}, (\ref{fsurjcongr0}), for all $\alpha \in {\rm Con}(A)$, $\{f(\phi )\ |\ \phi \in V_A(\alpha )\}=f(V_A(\alpha ))=V_B(f(\alpha ))=V_B(Cg_B(f(\alpha )))=V_B(f(\alpha \vee {\rm Ker}(f)))=V_B(\varphi _f(\alpha ))$.\label{phif}\end{remark}\vspace*{-28pt}

\begin{center}\begin{picture}(180,90)(0,0)
\put(6,65){${\rm Con}(A)$}
\put(22,51){$\bigcup \! |$}

\put(142,51){$\bigcup \! |$}
\put(16,35){${\cal K}(A)$}
\put(16,5){${\cal L}(A)$}
\put(136,65){${\rm Con}(B)$}
\put(136,35){${\cal K}(B)$}
\put(136,5){${\cal L}(B)$}
\put(17,23){$\lambda _A$}
\put(29,33){\vector(0,-1){19}}
\put(145,23){$\lambda _B$}
\put(144,33){\vector(0,-1){19}}

\put(82,73){$\varphi _f$}
\put(82,43){$\varphi _f$}
\put(77,12){${\cal L}(f)$}
\put(39,69){\vector(1,0){96}}
\put(39,39){\vector(1,0){96}}
\put(39,9){\vector(1,0){96}}\end{picture}\end{center}\vspace*{-12pt}

Let us define ${\cal L}(f):{\cal L}(A)\rightarrow {\cal L}(B)$, for all $\alpha \in {\cal K}(A)$, ${\cal L}(f)(\widehat{\alpha })=\widehat{\varphi _f(\alpha )}$, that is ${\cal L}(f)(\lambda _A(\alpha ))=\lambda _B(f(\alpha \vee {\rm Ker}(f)))$.

\begin{proposition} ${\cal L}(f)$ is well defined and it is a surjective lattice morphism.\label{fctorl}\end{proposition}

\begin{proof} By Remark \ref{phif}, the restriction $\varphi _f\mid _{{\cal K}(A)}:{\cal K}(A)\rightarrow {\cal K}(B)$ is well defined and surjective. Let $\alpha ,\beta \in {\cal K}(A)$ such that $\lambda _A(\alpha )=\lambda _A(\beta )$, so that $\rho _A(\alpha )=\rho _A(\beta )$, thus $V_A(\alpha )=V_A(\beta )$, hence $V_B(\varphi _f(\alpha ))=f(V_A(\alpha ))=f(V_A(\beta ))=V_B(\varphi _f(\beta ))$, thus $\rho _B(\varphi _f(\alpha ))=\rho _B(\varphi _f(\beta ))$, so $\lambda _B(\varphi _f(\alpha ))=\lambda _B(\varphi _f(\beta ))$, that is ${\cal L}(f)(\lambda _A(\alpha ))={\cal L}(f)(\lambda _A(\beta ))$; we have used Proposition \ref{esential}, (\ref{esential2}), and Remark \ref{phif}. Hence ${\cal L}(f)$ is well defined. $\lambda _B:{\cal K}(B)\rightarrow {\cal L}(B)$ $\varphi _f\mid _{{\cal K}(A)}:{\cal K}(A)\rightarrow {\cal K}(B)$ are surjective, thus so is their composition, and, since ${\cal L}(f)\circ \lambda _A=\lambda _B\circ \varphi _f$, it follows that ${\cal L}(f)$ is surjective.

By Remark \ref{fsurjcomut}, Lemma \ref{fsurjcongr}, (\ref{fsurjcongr1}), and Proposition \ref{onrho}, (\ref{onrho2}) and (\ref{onrho4}), for all $\alpha ,\beta \in {\cal K}(A)$, the following hold: ${\cal L}(f)(\widehat{\alpha }\wedge \widehat{\beta })={\cal L}(f)(\lambda _A(\alpha )\wedge \lambda _A(\beta ))={\cal L}(f)(\lambda _A([\alpha ,\beta ]_A))=\lambda _B(\varphi _f([\alpha ,\beta ]_A))=\lambda _B(f([\alpha ,\beta ]_A\vee {\rm Ker}(f)))=\lambda _B([f(\alpha \vee {\rm Ker}(f)),f(\beta \vee {\rm Ker}(f))]_B))=\lambda _B(f(\alpha \vee {\rm Ker}(f)))\wedge \lambda _B(f(\beta \vee {\rm Ker}(f)))=\lambda _B(\varphi _f(\alpha ))\wedge \lambda _B(\varphi _f(\beta ))={\cal L}(f)(\lambda _A(\alpha ))\wedge {\cal L}(f)(\lambda _A(\beta ))={\cal L}(f)(\widehat{\alpha })\wedge {\cal L}(f)(\widehat{\beta })$ and ${\cal L}(f)(\widehat{\alpha }\vee \widehat{\beta })={\cal L}(f)(\lambda _A(\alpha )\vee \lambda _A(\beta ))={\cal L}(f)(\lambda _A(\alpha 
\vee \beta ))=\lambda _B(\varphi _f(\alpha 
\vee \beta ))=\lambda _B(f(\alpha \vee \beta \vee {\rm Ker}(f)))=\lambda _B(f(\alpha \vee {\rm Ker}(f)\vee \beta \vee {\rm Ker}(f)))=\lambda _B(f(\alpha \vee {\rm Ker}(f)))\vee \lambda _B(f(\beta \vee {\rm Ker}(f)))=\lambda _B(\varphi _f(\alpha ))\vee \lambda _B(\varphi _f(\beta ))={\cal L}(f)(\lambda _A(\alpha ))\vee {\cal L}(f)(\lambda _A(\beta ))={\cal L}(f)(\widehat{\alpha })\vee {\cal L}(f)(\widehat{\beta })$. Therefore ${\cal L}(f)$ is a lattice morphism.\end{proof}

\begin{remark} Clearly, if $C$ is an algebra from ${\cal C}$ and $g:B\rightarrow C$ is a surjective morphism in ${\cal C}$, then ${\cal L}(g\circ f)={\cal L}(g)\circ {\cal L}(f)$. Hence we have defined a covariant functor ${\cal L}$ from the partial category of ${\cal C}$ whose morphisms are exactly the surjective morphisms from ${\cal C}$ to the partial category of the category ${\cal D}{\bf 01}$ of bounded distributive lattices whose morphisms are exactly the surjective morphisms from ${\cal D}{\bf 01}$.\end{remark}

\begin{openproblem} Extend the definition of ${\cal L}$ to the whole category ${\cal C}$, with the image in ${\cal D}{\bf 01}$, of course.\end{openproblem}

\begin{remark} By Proposition \ref{reticdistrib}, if ${\cal C}$ is congruence--distributive, then we may take ${\cal L}(f)=\varphi _f\mid _{{\cal K}(A)}:{\cal K}(A)\rightarrow {\cal K}(B)$, with ${\cal K}(A)$ and ${\cal K}(B)$ bounded sublattices of ${\rm Con}(A)$ and ${\rm Con}(B)$, respectively.\end{remark}

For any bounded lattice morphism $h:L\rightarrow M$, let us denote by ${\rm Ker}_{\rm Id}(h)=h^{-1}(\{0\})=\{x\in L\ |\ h(x)=0\}\in {\rm Id}(L)$, so that $L/{\rm Ker}_{\rm Id}(h)\cong h(L)$ by the Main Isomorphism Theorem (for lattices and lattice ideals).

\begin{proposition}[the reticulation preserves quotients] For any $\theta \in {\rm Con}(A)$, the lattices ${\cal L}(A/\theta )$ and ${\cal L}(A)/\theta ^*$ are isomorphic.\label{reticquo}\end{proposition}

\begin{proof} Recall that $\theta ^*=\lambda _A({\cal K}(A)\cap (\theta ])=\{\widehat{\alpha }\ |\ \alpha \in {\cal K}(A),\alpha \subseteq \theta \}\in {\rm Id}({\cal L}(A))$. $p_{\theta }:A\rightarrow A/\theta $ is a surjective morphism in ${\cal C}$, so we can apply the construction above:\vspace*{-20pt}

\begin{center}\begin{picture}(180,90)(0,0)
\put(6,65){${\rm Con}(A)$}
\put(22,51){$\bigcup \! |$}
\put(142,51){$\bigcup \! |$}
\put(16,35){${\cal K}(A)$}
\put(16,5){${\cal L}(A)$}
\put(136,65){${\rm Con}(A/\theta )$}
\put(136,35){${\cal K}(A/\theta )$}
\put(136,5){${\cal L}(A/\theta )$}
\put(17,23){$\lambda _A$}
\put(29,33){\vector(0,-1){19}}
\put(145,23){$\lambda _{A/\theta }$}
\put(144,33){\vector(0,-1){19}}

\put(82,73){$\varphi _{p_{\theta }}$}
\put(82,43){$\varphi _{p_{\theta }}$}
\put(77,12){${\cal L}(p_{\theta })$}
\put(39,69){\vector(1,0){96}}
\put(39,39){\vector(1,0){96}}
\put(39,9){\vector(1,0){96}}\end{picture}\end{center}\vspace*{-7pt}

For all $\alpha \in {\rm Con}(A)$, $\varphi _{p_{\theta }}(\alpha )=p_{\theta }(\alpha \vee {\rm Ker}(p_{\theta }))=(\alpha \vee \theta )/\theta $, so, for all $\alpha \in {\cal K}(A)$, ${\cal L}(p_{\theta })(\widehat{\alpha })=\widehat{(\alpha \vee \theta )/\theta }\in {\cal L}(A/\theta )$. Thus, for any $\alpha \in {\cal K}(A)$: $\widehat{\alpha }\in {\rm Ker}_{\rm Id}({\cal L}(p_{\theta }))$ iff ${\cal L}(p_{\theta })(\widehat{\alpha })=\widehat{\Delta _{A/\theta }}$ iff $\widehat{(\alpha \vee \theta )/\theta }=\widehat{\theta /\theta }$, that is $\lambda _{A/\theta }((\alpha \vee \theta )/\theta )=\lambda _{A/\theta }(\theta /\theta )$, iff $\rho _{A/\theta }((\alpha \vee \theta )/\theta )=\rho _{A/\theta }(\theta /\theta )$ iff $\rho _A(\alpha \vee \theta )/\theta =\rho _A(\theta )/\theta $ iff $\rho _A(\alpha \vee \theta )=\rho _A(\theta )$ iff $\rho _A(\alpha \vee \theta )\subseteq \rho _A(\theta )$ iff $\alpha \vee \theta \subseteq \rho _A(\theta )$ iff $\alpha \subseteq \rho _A(\theta )$ iff $\widehat{\alpha }\in (\rho _A(\theta ))^*=\theta ^*$, hence ${\rm Ker}_{\rm Id}({\cal L}(p_{\theta }))=\theta ^*$; we have applied Proposition \ref{esential}, (\ref{esential3}), Remark \ref{clara}, (\ref{clara2}), Proposition \ref{esential}, (\ref{esential1}), and Corollary \ref{cor13}, (\ref{cor13(1)}). Proposition \ref{fctorl} ensures us that the lattice morphism ${\cal L}(p_{\theta })$ is surjective, so, from the Main Isomorphism Theorem, we obtain: ${\cal L}(A/\theta )\cong {\cal L}(A)/\theta ^*$.\end{proof}

\begin{proposition} The lattices ${\cal L}(A)$ and ${\cal L}(A/\rho _A(\Delta _A))$ are isomorphic.\label{reticsemiprime}\end{proposition}

\begin{proof} By Corollary \ref{cor13}, (\ref{cor13(1)}), and Proposition \ref{reticquo}, $\rho _A(\Delta _A)^*=\Delta _A^*$, hence the lattice ${\cal L}(A/\rho _A(\Delta _A))$ is isomorphic to ${\cal L}(A)/\rho _A(\Delta _A))^*={\cal L}(A)/\Delta _A^*$, which, in turn, is isomorphic to ${\cal L}(A/\Delta _A)$, and thus to ${\cal L}(A)$, since the algebras $A/\Delta _A$ and $A$ are isomorphic.\end{proof}

\begin{remark} Propositions \ref{rhosemiprime} and \ref{reticsemiprime} show that the reticulation of any algebra $M$ from a semi--degenerate congruence--modular variety, such that ${\cal K}(M)$ is closed with respect to the commutator of $M$ and $\nabla _M\in {\cal K}(M)$, is isomorphic to the reticulation of a semiprime algebra from the same variety.\end{remark}

\begin{corollary} ${\cal B}({\cal L}(A))$ and ${\cal B}({\rm Con}(A/\rho _A(\Delta _A)))$ are isomorphic Boolean algebras.\label{bretic}\end{corollary}

\begin{proof} By Propositions \ref{rhosemiprime}, \ref{prop6.3} and \ref{reticsemiprime}, $A/\rho _A(\Delta _A)$ is semiprime, thus the Boolean algebra  ${\cal B}({\rm Con}(A/\rho _A(\Delta _A)))$ is isomorphic to ${\cal B}({\cal L}(A/\rho _A(\Delta _A)))$, which in turn is isomorphic to ${\cal B}({\cal L}(A))$.\end{proof}

Recall the well--known {\bf Nachbin`s Theorem}, which states that, given a bounded distributive lattice $L$, we have: $L$ is a Boolean algebra iff ${\rm Max}_{\rm Id}(L)={\rm Spec}_{\rm Id}(L)$ iff ${\rm Max}_{\rm Filt}(L)={\rm Spec}_{\rm Filt}(L)$.

\begin{proposition} The following are equivalent:\begin{enumerate}
\item\label{hypbool1} $A$ is hyperarchimedean; 
\item\label{hypbool2} $A/\rho _A(\Delta _A)$ is hyperarchimedean; 
\item\label{hypbool3} ${\rm Max}(A)={\rm Spec}(A)$;
\item\label{hypbool4} ${\cal L}(A)$ is a Boolean lattice;
\item\label{hypbool5} the lattice ${\cal L}(A)$ is isomorphic to ${\cal B}({\rm Con}(A))$;
\item\label{hypbool6} the lattice ${\cal L}(A)$ is isomorphic to ${\cal B}({\rm Con}(A/\rho _A(\Delta _A)))$.\end{enumerate}\label{hypbool}\end{proposition}

\begin{proof} By Nachbin`s Theorem, Proposition \ref{prop11} and Corollary \ref{homeomaxspec}, (\ref{hypbool3}) is equivalent to (\ref{hypbool4}). Trivially, (\ref{hypbool6}) implies (\ref{hypbool4}), while the converse holds by Corollary \ref{bretic}.

If $A$ is semiprime, that is $\rho _A(\Delta _A)=\Delta _A$, so that $A/\rho _A(\Delta _A)=A/\Delta _A$ is isomorphic to $A$, then (\ref{hypbool1}) is equivalent to (\ref{hypbool2}) and (\ref{hypbool5}) is equivalent to (\ref{hypbool6}). Now let us drop the condition that $A$ is semiprime. But $A/\rho _A(\Delta _A)$ is semiprime, according to Proposition \ref{rhosemiprime}, hence, by the above, (\ref{hypbool2}) is equivalent to ${\rm Max}(A/\rho _A(\Delta _A))={\rm Spec}(A/\rho _A(\Delta _A))$ and to the fact that ${\cal L}(A/\rho _A(\Delta _A))$ is a Boolean lattice, which, in turn, is equivalent to (\ref{hypbool4}) by Proposition \ref{reticsemiprime}. But, as shown by Lemma \ref{folclor} and Remark \ref{specquo}, ${\rm Max}(A/\rho _A(\Delta _A))={\rm Spec}(A/\rho _A(\Delta _A))$ iff ${\rm Max}(A)\cap [\rho _A(\Delta _A))={\rm Spec}(A)\cap [\rho _A(\Delta _A))$ iff ${\rm Max}(A)={\rm Spec}(A)$, since ${\rm Max}(A)\subseteq {\rm Spec}(A)\subseteq [\rho _A(\Delta _A))$.\end{proof}

\end{document}